\documentclass{article}
\usepackage{fullpage, lmodern, float, mathtools, parskip}
\allowdisplaybreaks
\usepackage[utf8]{inputenc}
\usepackage{xcolor}
\usepackage{url}
\usepackage{booktabs}
\usepackage[colorlinks, citecolor = blue!80!black, linkcolor = blue!80!black, breaklinks, pdfauthor ={"Martijn G\"{o}sgens, Lukas L\"{u}chtrath, Elena Magnanini, Marc Noy, \'{E}lie de Panafieu"}]{hyperref}
\usepackage{nicefrac}
\usepackage{amssymb,amsmath,amsthm,dsfont}
\usepackage{graphicx}
\usepackage{comment}
\usepackage{mathrsfs}
\usepackage{xspace}
\usepackage{enumerate}
\usepackage{caption}
\usepackage{graphicx}
\usepackage{subcaption}
\usepackage[
    capitalize,
    noabbrev
]{cleveref}

\usepackage[raggedright]{titlesec} % Takes care of too long section titles 
\usepackage{booktabs} %% Nice tables

\usepackage[style = numeric, sorting=nyt, url = false, abbreviate=false, maxbibnames=9, sortcites=true, doi = true, backend = biber, giveninits = true]{biblatex}
\renewbibmacro{in:}{\ifentrytype{article}{}{\printtext{\bibstring{in}\intitlepunct}}}
\bibliography{main.bib}

\usepackage{titling}
    \pretitle{\centering\LARGE\scshape}
        \posttitle{\vskip 0.75cm}

    \predate{\vskip 0.75 cm \centering\large}
        \postdate{\par}

\newcommand{\proba}{\mathbb{P}}
\newcommand{\mean}{\mathbb{E}}
\renewcommand{\P}{\proba}
\newcommand{\E}{\mean}
\newcommand{\bigO}{\mathcal{O}}
\newcommand{\exactbigO}{\Theta}
\newcommand{\smallo}{o}
\newcommand{\integers}{\mathds{Z}}
\newcommand{\reals}{\mathds{R}}

\newcommand{\rationals}{\mathds{Q}}
\newcommand{\real}{\operatorname{Re}}

\newcommand{\mD}{\mathcal{D}}
\newcommand*{\eg}{\textit{e.g.}\@\xspace}

\newcommand*{\resp}{\textit{resp.}\@\xspace}
\newcommand*{\wrt}{w.r.t.\@\xspace}

\newcommand{\Var}{\operatorname{Var}}
\newcommand{\N}{\mathbb{N}}
\newcommand{\rv}[1]{\mathbf{#1}}

\newcommand{\mA}{\mathcal{A}}
\newcommand{\mB}{\mathcal{B}}
\newcommand{\mC}{\mathcal{C}}

\newcommand{\PGF}{\operatorname{PGF}}

\newcommand{\CF}{\operatorname{CF}}
\newcommand{\PG}{\rv{CG}}
\newcommand{\G}{\mathcal{G}}
\newcommand{\Icentral}{I_{\operatorname{central}}}
\newcommand{\Itail}{I_{\operatorname{tail}}}

\newtheorem{theorem}{Theorem}[section]
\newtheorem{lemma}[theorem]{Lemma}
\newtheorem{corollary}[theorem]{Corollary}
\newtheorem{proposition}[theorem]{Proposition}

\newcommand{\proofparagraph}[1]{\noindent \textbf{#1}}

\theoremstyle{definition}

\usepackage{todonotes}

\definecolor{colorElena}{rgb}{0,0,1}

\newcommand{\ER}{Erd\H{o}s-Rényi\ }

\thanksmarkseries{arabic}

\author{
Martijn Gösgens\thanks{Eindhoven University of Technology, PO Box 513, 5600 MB Eindhoven, Netherlands. \\ 
\phantom{Cor$1$}Corresponding Author: research@martijngosgens.nl}\ $^*$ 
\and
Lukas L\"{u}chtrath\thanks{Weierstrass Institute for Applied Analysis and Stochastics, Mohrenstr.\ 39, 10117 Berlin, Germany}  
\and
Elena Magnanini\textcolor{blue!80!black}{\thanksmark{2}}
\and
Marc Noy\thanks{Universitat Polit\`{e}cnica de Catalunya, C. Jordi Girona 31, 08034 Barcelona, Spain}
\and
\'{E}lie de Panafieu\thanks{Nokia Bell Labs France, 12 rue Jean Bart, 91300 Massy, France\\
\phantom{Cor}$^*$Authors listed in alphabetic order.}
}

\title{The Erdős–Rényi random graph conditioned on every component being a clique}

\date{\today}

\begin{document}

\maketitle

\begin{abstract}
%%% This version of the abstract has 147 words according to the overleaf word counter
Motivated by an application in community detection,
we consider an \ER random graph conditioned on the rare event that all connected components are fully connected.
Such graphs can be considered as partitions of vertices into cliques.
Hence, this conditional distribution defines a distribution over partitions.
We show that a popular community detection method is equivalent to Bayesian inference with this distribution as prior over the community partitions.
Using tools from analytic combinatorics, we prove limit theorems for several graph observables in this conditional distribution:
the number of cliques; 
the number of edges; 
and the degree distribution.
We consider several regimes of the connection probability $p$ as the number of vertices $n$ diverges.
For $p=\nicefrac{1}{2}$, the conditioning yields the uniform distribution over set partitions, which is well-studied, but has not been studied as a graph distribution before. 
For $p<\nicefrac{1}{2}$, we show that the number of cliques is of the order $n/\sqrt{\log n}$, while for $p>\nicefrac{1}{2}$, we prove that the graph consists of a single clique with high probability. This shows that there is a phase transition at $p=\nicefrac{1}{2}$. We additionally study the near-critical regime $p_n\downarrow\nicefrac{1}{2}$, as well as the sparse regime $p_n\downarrow0$.
Finally, we discuss the implications of these results for community detection.

\smallskip
\noindent\footnotesize{{\textbf{AMS-MSC 2020}: 05C80; 60F05; 11B73; 05A18 }

 }
\end{abstract}

\tableofcontents

\section{Introduction}

Many networks contain groups of vertices that are more densely connected to each other than the rest of the network. In network science terminology, such groups are known as \emph{communities}. For example, in social networks, communities may correspond to friend groups, while in the Wikipedia network, they may represent articles on similar topics. \emph{Community detection} is the task of partitioning the network vertices into communities~\cite{fortunato2010community,fortunato2016community}. 
Community detection is known to be an \emph{ill-defined problem}~\cite{fortunato2016community}, as there is no precise mathematical definition for the term `community'. For instance, friend groups typically consist of a small number of vertices that are almost fully connected, whereas the number of Wikipedia articles on a single topic (e.g. mathematics) may be large, and each such article only refers to a fraction of mathematics articles. These different interpretations of the word `community' have led to the development of numerous community detection methods \cite{newman2004finding,rosvall2008maps,fortunato2010community,rosvall2019different}. 

\emph{Modularity maximisation} is arguably the most widely-used community detection method~\cite{newman2004finding,reichardt2006statistical,lancichinetti2009community}.
Modularity compares the number of edges inside the communities to the number of such intra-community edges in a randomized version of the graph.
This quantity measures how well a given partition into communities matches the given graph. To detect communities, this quantity is maximised over the set of partitions. 
Modularity comes with a \emph{resolution parameter}, which controls the sizes of the detected communities~\cite{fortunato2007,kumpula2007limited,lancichinetti2011limits}.
However, it is not obvious how to choose this parameter so that the detected communities have the desired sizes~\cite{arenas2008analysis,prokhorenkova2019using}. 
For a particular value of this resolution parameter, modularity maximisation is equivalent to likelihood maximisation~\cite{newman2016equivalence}.
While this equivalence provides a theoretically justified choice for this parameter, it has been observed that the corresponding community detection method has a tendency to over-estimate the number of communities~\cite{peixoto2023descriptive,zhang2020statistical}.

%While this maximisation is NP-hard~\cite{brandes2007modularity}, there exist fast approximate maximisation algorithms~\cite{blondel2008fast,traag2019louvain}.
%

Recently, community detection methods based on Bayesian inference have gained popularity~\cite{peixoto2019bayesian,peixoto2023descriptive}. These methods maximise the \emph{posterior probability} of a community partition with respect to some \emph{prior distribution} over community partitions.
The appeal of these Bayesian methods is that it provides a statistically sound framework for community detection.
In addition, this Bayesian framework is quite versatile, allowing for overlapping communities~\cite{peixoto2015model}, ordered communities~\cite{peixoto2022ordered} and hierarchical communities~\cite{peixoto2014hierarchical}.
However, this versatility comes at the cost of complexity, making the theoretical analysis of these models extremely challenging.

In this work, we form a bridge between modularity maximisation and Bayesian inference: \cref{thm:bayesian-modularity} proves that modularity maximisation is equivalent to maximising a Bayesian posterior. We show that the corresponding prior distribution can be characterized as an \ER (ER) random graph conditioned on the rare event that every connected component is fully connected.
%This motivates the study of this conditional distribution.
Therefore, understanding this conditional distribution helps us to better understand this community detection method.
It also provides a rigorous method to choose the parameter of this method, given prior knowledge on the network to be partitioned.

A \emph{cluster graph} is a graph that is the disjoint union of complete graphs. Put differently, each connected component or \emph{cluster} of the graph forms a clique. We are interested in the \ER random graph on $n$ vertices with connection probability $p$, conditioned on the rare event that it is a cluster graph. We refer to such a graph as a \emph{random cluster graph} (RCG) throughout the manuscript. To avoid confusion, we emphasise that the random cluster graph here is not the same object as the one defined e.g.\ in~\cite{ShamirTsur2007}, which is nowadays better known as the \emph{planted partition model} or the \emph{stochastic block model}, cf.~\cite{holland1983stochastic} and \Cref{sec:Community:detection} below. Besides the motivation from community detection, there are two additional reasons that motivate the study of the RCG.

Firstly, the event that we condition on is extremely rare: for fixed $p$, the probability that an ER random graph is a cluster graph decays like $\exp(-\Theta(n^2))$ as $n\to\infty$, and hence its
properties after conditioning differ significantly from its original properties. For example, the expected edge density of an ER graph with $n$ vertices and constant connection probability $p$ equals $p$.
In contrast, we show that the RCG undergoes a phase transition at $p =\nicefrac{1}{2}$, where the expected edge density is $o(1)$ as $n\to\infty$ if $p \le\nicefrac{1}{2}$,
while it is $1-o(1)$ for $p > \nicefrac12$. This \emph{phase transition} is already evident for small $n$, as illustrated in \cref{fig:cg_samples}. The same phase transition occurs for other graph observables as well. The occurrence of such phase transitions in interacting systems is of great mathematical interest. 
%Of particular interest is the uniqueness of the phase transition, especially for global properties like connectivity and existence of a giant component. 

In a related work~\cite{MartinYeo2018}, the authors consider the ER graph conditioned on the event of being a forest, where the connection probability $p$ similarly makes the occurrence of certain forests more likely. The pure uniform case of this has been studied since the early 2000's~\cite{GrimmetWinkler2004}. The non-uniform case was recently further elaborated (in a much more general setting) in~\cite{BauerschmidtEtAl2021} under the name \emph{arboreal gas} with focus on the percolation phase transition. Here, the authors use statistical mechanics tools to deduce their results. It is an interesting direction for future work to study whether statistical mechanics tools are applicable for the RCG as well. 
Indeed, the distribution of the RCG has the form of a \emph{Gibbs measure}. The sufficient statistic of this distribution is the number of edges. 
Therefore, the \emph{Maximum Entropy Principle}~\cite{jaynes1957information}, implies that the RCG distribution maximises the \emph{Gibbs entropy} among all cluster graph distributions with a specified expected number of edges.
%Firstly, the RCG is an interesting object from a probabilistic viewpoint because it undergoes a phase transition at $p=\nicefrac{1}{2}$: we prove in \cref{thm:mainClust} that for $p>\nicefrac{1}{2}$, the RCG consists of a single clique with high probability, while for $p<\nicefrac{1}{2}$, it consists of $\bigO(n/\sqrt{\log n})$ cliques. The ER random graph (without conditioning) does not have this phase transition, since for each $p\in(0,1)$, the graph consists of a single connected component with high probability~\cite{vdHofstad2017}. 
%For \(p=\nicefrac{1}{2}\), the RCG is a cluster graph uniformly chosen among all cluster graphs. Varying the value of \(p\) then changes the probability of edges and makes the occurrence of certain cluster graphs more likely. 

\begin{figure}[t!]
    \centering
    \includegraphics[scale = 0.3]{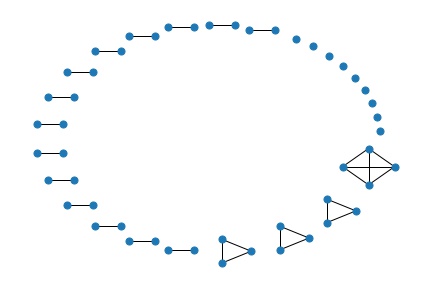}
    \includegraphics[scale = 0.3]{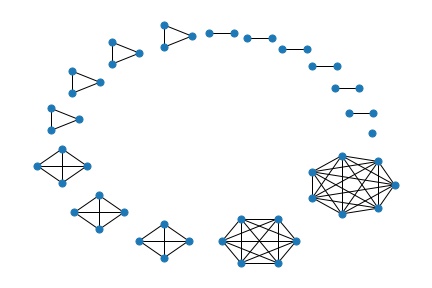}
    \includegraphics[scale = 0.3]{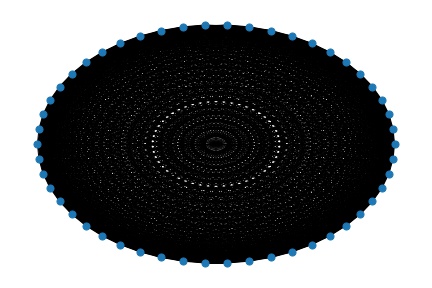}
    \caption{Samples of three random cluster graphs on $50$ vertices. The one on the left is sampled with \(p=0.25\), the middle one with \(p=0.51\) and the right one with \(p=0.53\), resulting in the complete graph.}
    \label{fig:cg_samples}
\end{figure}

Besides the probabilistic interest, the RCG is also of combinatorial interest. For \(p=\nicefrac{1}{2}\), the RCG is a cluster graph uniformly chosen among all cluster graphs.
Because of the one-to-one correspondence between cluster graphs and set partitions, the RCG therefore generalises the uniform set partition, which is a standard object in combinatorics~\cite{pittel1997random,sachkov1997probabilistic}. Varying the value of \(p\) changes the probability of edges (intra-cluster pairs), which makes the occurrence of certain partitions more likely. 
%The RCG can therefore be seen as a \emph{weighted} uniform distribution over set partitions and thus providing a natural extension of the model.

We show that the phase transition of the RCG at $p=\nicefrac{1}{2}$ translates to a divergence in the generating function of cluster graphs. Because of this divergence, we are unable to apply standard tools from combinatorics to study the supercritical regime ($p>\nicefrac{1}{2}$). Instead, we rely on recent techniques developed in~\cite{dovgal2023asymptotics}.  
In addition, the derivation of the limit laws in the critical ($p=\nicefrac{1}{2}$) and subcritical ($p<\nicefrac{1}{2}$) regimes require delicate saddle point analyses.
We first use these methods to study the regimes where the edge probability $p$ remains constant as $n\rightarrow\infty$. We prove limit theorems for the number of cliques and edges, as well as for the degree distribution in all three different regimes. 
An overview of these results is given in \cref{tab:results}.
We further study the near-critical regime \(p_n\downarrow \nicefrac{1}{2}\) and the sparse regime \(p_n\downarrow 0\). %To this end, we apply our derived combinatorial tools to some more probabilistic methods.
 
The article is structured as follows. In \cref{sec:main:results}, we present and discuss our main results. In \cref{sec:Community:detection}, we discuss the implications of these results for community detection. In \cref{sec:distributionsOfInterest}, we derive the exact distributions of the objects of interest and briefly explain the main proof strategies. The critical regime (\(p=\nicefrac{1}{2}\)) is studied in \cref{sec:critical}, the subcritical regime (\(p<\nicefrac{1}{2}\)) is studied in \cref{sec:subcritical}, and the supercritical regime (\(p>\nicefrac{1}{2}\)) is studied in \cref{sec:supercritical}. The near-critical regime is studied in Section~\ref{sec:window} and the sparse regime is studied in Section~\ref{sec:vanishing}. Throughout this paper, we illustrate some of our results through numerical experiments. The code for these experiments is available on GitHub\footnote{See \url{https://github.com/MartijnGosgens/RandomClusterGraphs}.} and we explain some implementation details in \cref{sec:numerics}.  A list of notation can be found in \cref{sec:nomenclature}.

\subsection{Main results} \label{sec:main:results}
We let $\rv{CG}_{n,p}$ denote a random cluster graph with parameters $n$ and $p$. That is, an ER random graph with $n$ vertices and connection probability $p$, conditioned on the rare event that every connected component is fully connected. All random variables associated with \(\rv{CG}_{n,p}\) are denoted in bold capitals with indices \(n,p\). Our main quantities of interest are the number of connected components (clusters) in $\rv{CG}_{n,p}$, denoted by $\rv{C}_{n,p}$, the number of edges denoted by $\rv{M}_{n,p}$, and the degree $\rv{D}_{n,p}$ of a vertex chosen %independent and 
uniformly at random from the vertex set. We further let $\rv{S}_{n,p}:=\rv{D}_{n,p}+1$ denote the \emph{size-biased} cluster size. In terms of set partitions, $\rv{C}_{n,p}$ then coincides with the number of sets (also referred to as blocks) and $\rv{M}_{n,p}$ coincides with the number of intra-cluster pairs.  Whenever convenient, we drop the $p$ in the subscripts from the notation. We consider all these random variables to be defined on a common underlying probability space. We denote the corresponding probability measure by \(\mathbb{P}\) and the corresponding expectation operator by \(\E\).

\begin{table}[]
    \centering
    \footnotesize
    \begin{tabular}{cccc}
            \toprule
				\textbf{Regime} & \textbf{\# Cliques }$\rv{C}_{n,p}$ & \textbf{\# Edges} $\rv{M}_{n,p}$ & \textbf{Degree of first vertex} $\rv{D}_{n,p}$
			\tabularnewline 
	 	     \midrule  
				$p>\nicefrac{1}{2}$ & $1$ w.h.p. & ${\binom{n}{2}}$ w.h.p. & $n-1$ w.h.p. 
            \tabularnewline
				$p=\nicefrac{1}{2}$ & Gaussian $\left(\mu\sim\frac{n}{\log n}\right)$ & Gaussian $\left(\mu\sim n\log n\right)$ & Poisson $\left(\lambda=\log n - \log\log n+o(1)\right)$ 
            \tabularnewline
				$p<\nicefrac{1}{2}$ & Gaussian $\left(\mu=\Theta\left(\tfrac{n}{\sqrt{\log n}}\right)\right)$ & Gaussian $\left(\mu=\Theta(n\sqrt{\log n})\right)$ & Periodic limit 
            \tabularnewline
				$p_n=n^{-\alpha+o(1)}$ & $\Theta(n)$ & $\Theta(n)$ & In $\{\lfloor\sqrt{2/\alpha}-1\rfloor,\lceil \sqrt{2/\alpha}-1 \rceil\}$ w.h.p.
            \tabularnewline
            \bottomrule
		\end{tabular}
    \caption{Summary of the main results. The `Gaussian' entries refer to Gaussian limit laws. That is, there are $\mu_n,\sigma_n$ so that $(X_n-\mu_n)/\sigma_n$ converges in distribution to a standard Gaussian. Further details are given in \cref{thm:mainClust,thm:mainEdges,thm:mainDegree} and \cref{thm:mainFixedDegSparse}.}
    \label{tab:results}
\end{table}

We further use the standard Landau notation and denote for positive functions \(f,g\) by \(f(x)\sim g(x)\) the fact that \(f(x)/g(x)\to 1\) as \(x\to\infty\). Further, \(f=\bigO(g)\) indicates that asymptotically \(f(x)/g(x)\) is bounded from above, \(f=o(g)\) indicates \(f(x)/g(x)\to 0\), \(f=\Omega(g)\) refers to \(g=\bigO(f)\), and finally \(f=\Theta(g)\) indicates \(f=\bigO(g)\) as well as \(f=\Omega(f)\). We denote by \(K_n\) the complete graph on \(n\) vertices. A random variable distributed according to a standard Gaussian law is denoted by \(\mathcal{N}(0,1)\).

\paragraph{Number of clusters}
As already pointed out above, there is a phase transition at \(p=\nicefrac{1}{2}\).  This is shown in our first three theorems, each devoted to one of the main quantities. 

\medskip

\begin{theorem}[Number of clusters in the RCG] \label{thm:mainClust}
	Consider the random cluster graph \(\rv{CG}_{n,p}\) on \(n\in\mathbb{N}\) vertices and ER edge probability \(p\in(0,1)\) and the number of its clusters \(\rv{C}_{n,p}\). 
	\begin{enumerate}[(i)]
		\item If \(p>\nicefrac{1}{2}\), then
			\[
				\lim_{n\to\infty}\P(\rv{C}_{n,p}=1)=1.
			\]
			Put differently, \(\rv{C}_{n,p}=K_n\) with high probability.
		\item If \(p=\nicefrac{1}{2}\), then \(\rv{C}_{n,p}\) obeys a central limit theorem, i.e.,
			\[
				\frac{\rv{C}_{n,p}-\E\rv{C}_{n,p}}{\sqrt{\Var(\rv{C}_{n,p})}}\longrightarrow \mathcal{N}(0,1),
			\]
			in distribution, as \(n\to\infty\). Moreover,
			\[
				\E \rv{C}_{n,p}\sim\frac{n}{\log n} \quad \text{ and } \quad  \operatorname{Var}(\rv{C}_{n,p})\sim\frac{n}{(\log n)^{2}}.
			\]
		\item If \(p<\nicefrac{1}{2}\), then \(\rv{C}_{n,p}\) obeys a central limit theorem, i.e.,
			\[
				\frac{\rv{C}_{n,p}-\E\rv{C}_{n,p}}{\sqrt{\Var(\rv{C}_{n,p})}}\longrightarrow \mathcal{N}(0,1),
			\]
			in distribution, as \(n\to\infty\). Moreover,
			\[
				\E \rv{C}_{n,p}\sim\sqrt{\frac{\log(1-p)-\log p}{2}}\frac{n}{\sqrt{\log n}} \quad \text{ and } \quad  \operatorname{Var}(\rv{C}_{n,p})=\Theta\left(\frac{n}{(\log n)^{3/2}}\right).
			\]
	\end{enumerate}
\end{theorem}

\medskip

Let us remark that Part~(ii) of Theorem~\ref{thm:mainClust} is a known result for uniform set partitions, cf.~\cite{harper1967stirling} and~\cite[Theorem 4.1.1]{sachkov1997probabilistic} that we only state in the theorem for completeness. Part~(i) is a consequence of \Cref{th:proba:clique:asymptotic:expansion} and Part~(iii) is proven with more precise asymptotics in \Cref{prop:CltCluster}. We will see there that the $\Theta$ notation in $\operatorname{Var}(\rv{C}_{n,p})$ hides a bounded function oscillating with $n$. However, the expressions require additional notation that we postpone to keep the theorem concise. A classical question in random graph theory is whether or not the graph contains a (unique) giant component. That is, there exists a single macroscopic connected component (of size \(\Theta(n)\)) and all other components are of strictly smaller size with high probability. By Theorem~\ref{thm:mainClust}~(i), such a giant component always exists for \(p>\nicefrac{1}{2}\). Since the largest cluster in \(\rv{C}_{n,1/2}\) is asymptotically of size \(\Theta(\log n)\)~\cite[Eq.\ (4.5.9)]{sachkov1997probabilistic}, there is no such component at criticality. Further, Part~(i) of the theorem shows that the RCG is connected with high probability if and only if \(p>\nicefrac{1}{2}\). Hence, the critical threshold for the occurrence of a giant and the critical threshold for connectivity coincide for the RCG and neither behaviour occurs at the critical value \(\nicefrac{1}{2}\). This is in contrast with the behaviour of the standard ER graph. Here, it is commonly known that a giant occurs for all \(p>1/n\) but connectivity requires \(p>\log n/n \)~\cite{vdHofstad2017}. From the latter, we immediately see that the ER graph is connected for any fixed \(p<1/2\) while the RCG decomposes in roughly \(n/\sqrt{\log n}\) components. This altogether shows the major impact of the conditioning. 

\paragraph{Number of edges.}
The situation is similar for the number of edges in \(\rv{CG}_{n,p}\). In the standard ER graph, the number of edges is binomially distributed with parameters \({n\choose 2}\) and \(p\), leading to a total of \(\Theta(n^2)\) many edges in the graph. From the previous theorem, we immediately infer that for \(p>\nicefrac{1}{2}\) all edges are present with high probability, leading to the same order. However, for \(p\leq \nicefrac{1}{2}\) the number of edges in \(\rv{CG}_{n,p}\) is drastically reduced, as shown by our next theorem. In summary, for any fixed value of \(p\), the ER graph is connected with high probability, and has a quadratic number of edges. Since for \(p>\nicefrac{1}{2}\) the probabilistic costs for adding an additional edge are smaller than removing an edge, the cheapest way of becoming a cluster graph is then just to add all missing edges for the complete graph. On the other hand, for \(p<\nicefrac{1}{2}\) it becomes cheaper to remove a lot of edges, which changes the structure drastically. This is in contrast to the ER graph conditioned on being a forest that we mentioned previously. In the sparse regime \(p=\lambda/n\) the ER graph is known to form tree-like clusters~\cite{vdHofstad2017}, that are trees with only a few additional edges. Hence, conditioning on being a forest leads to a distribution that is structurally  closer to the original ER graph than the conditioning that we consider here~\cite{BauerschmidtEtAl2021}.

\medskip

\begin{theorem}[Number of edges in the RCG] \label{thm:mainEdges}
	Consider the random cluster graph \(\rv{CG}_{n,p}\) on \(n\in\mathbb{N}\) vertices and ER edge probability \(p\in(0,1)\) and its number of edges \(\rv{M}_{n,p}\).
	\begin{enumerate}[(i)]
		\item If \(p>\nicefrac{1}{2}\), then
			\[
				\lim_{n\to\infty}\P\left(\rv{M}_{n,p}=\binom{n}{2}\right) = 1.
			\] 
		\item If \(p=\nicefrac{1}{2}\), then \(\rv{M}_{n,p}\) obeys a central limit theorem, i.e., 
			\[
				\frac{\rv{M}_{n,\nicefrac{1}{2}}-\E \rv{M}_{n,\nicefrac{1}{2}}}{\sqrt{\operatorname{Var}(\rv{M}_{n,\nicefrac{1}{2}})}}\longrightarrow \mathcal{N}(0,1)
			\]
			in distribution as \(n\to\infty\). Moreover,
			\[
				\E\rv{M}_{n,\nicefrac{1}{2}}\sim n\log n \quad \text{ and } \quad \operatorname{Var}(\rv{M}_{n,\nicefrac{1}{2}})=\Theta(n\log(n)^2).
			\]
		\item If \(p<\nicefrac{1}{2}\), then \(\rv{M}_{n,p}\) obeys a central limit theorem, i.e., 
			\[
				\frac{\rv{M}_{n,p}-\E \rv{M}_{n,p}}{\sqrt{\operatorname{Var}(\rv{M}_{n,p})}}\longrightarrow \mathcal{N}(0,1)
			\]
			in distribution as \(n\to\infty\). Moreover,
			\[\E \rv{M}_{n,p}\sim n\sqrt{\frac{\log n}{2(\log(1-p)-\log p)}} \quad \text{ and } \quad \operatorname{Var}(\rv{M}_{n,p})=\Theta\left(n\log(n)^{3/2}\right).
			\]
	\end{enumerate}
\end{theorem}

\medskip

Part~(i) of the theorem is essentially the same result as Part~(i) of Theorem~\ref{thm:mainClust}. We prove Part~(ii) in \Cref{prop:critical-edges} and Part~(iii) in \Cref{prop:CltEdges}, in which we again derive precise asymptotics for expectation and variance.
Again, the $\Theta$ notations hide bounded functions oscillating with $n$, expressed in those sections.
To the best of our knowledge, the number of intra-cluster pairs of a random uniform set partition has not been studied before. Deriving the limit law of Theorem~\ref{thm:mainEdges}~(ii) came with an additional obstacle: whereas most of the results of \cref{thm:mainClust} and~\ref{thm:mainEdges} are obtained by applying Lévy's continuity theorem, this approach is impossible for $\rv{M}_{n,\nicefrac{1}{2}}$ as its moment generating function diverges for every positive parameter value. This is a consequence of the phase transition in this model, resulting from the graph structure (see \cref{sec:distributionsOfInterest}). To overcome this technical difficulty, we derived the asymptotics of its characteristic function, which required us to perform the saddle point analysis from scratch instead of straight-forwardly applying established techniques such as Hayman admissibility.

\paragraph{Degree distribution.}
To formulate the results for the degree distributions, we recall that the total variation distance of two random variables \(Y\) and \(Z\), defined on the same sample space, is defined as \(\operatorname{d}_{TV}(Y,Z)=\sup|\P_Y(E)-\P_Z(E)|\) where the supremum is taken over all events \(E\) and \(\P_Y\) (resp.\ \(\P_Z\)) denotes the distribution of \(Y\) (resp.\ $Z$).

\begin{theorem}[Degree distribution of the RCG]\label{thm:mainDegree}
	Consider the random cluster graph \(\rv{CG}_{n,p}\) on \(n\in\mathbb{N}\) vertices and ER edge probability \(p\in(0,1)\) and the degree \(\rv{D}_{n,p}\) of a uniformly chosen vertex.
	\begin{enumerate}[(i)]
		\item If \(p>\nicefrac{1}{2}\), then 
			\[
				\lim_{n\to\infty}\P(\rv{D}_{n,p}=n-1)=1.
			\]
		\item If \(p=\nicefrac{1}{2}\), then for a Poisson random variable \(X_n\) with parameter \(\log n-\log\log n+o(1)\), we have
			\begin{enumerate}[(a)]
				\item for all \(z\in\mathbb{C}\),
					\[
						\E z^{\rv{D}_{n,\nicefrac{1}{2}}} \sim \E z^{X_n}.
					\]
					That is, the probability generating function of \(\rv{D}_{n,\nicefrac{1}{2}}\) and the one of \(X_n\) are asymptotically the same.
				\item Additionally,
					\[
						\lim_{n\to\infty}\operatorname{d}_{TV}(\rv{D}_{n,\nicefrac{1}{2}}, X_n) = 0.
					\]
			\end{enumerate}
		\item If \(p<\nicefrac{1}{2}\), then \(\E \rv{D}_{n,p}=\Theta(\sqrt{\log n})\). Moreover, for each \(\lambda\in[0,1)\) there exists a subsequence \((n_k)_{k\in\mathbb{N}}\) such that 
        \[
            \rv{D}_{n_k,p}-\left\lfloor \sqrt{\frac{2\log n_k}{\log(1-p)-\log p}}-1-\frac{1}{\log(1-p)-\log p} \right\rfloor\longrightarrow X_\lambda
        \] 
         in distribution as \(k\to\infty\), where \(X_\lambda\) is defined by
			\begin{equation}\label{eq:discrete-gauss}
				\P(X_\lambda = d)=\frac{ \big(\tfrac{p}{1-p}\big)^{(d-\lambda)^2/2}}{\sum_{d'\in\mathbb{Z}}\big(\tfrac{p}{1-p}\big)^{(d'-\lambda)^2/2}}
			\end{equation}
		for all \(d\in\mathbb{Z}\).
	\end{enumerate}
\end{theorem}

The term subtracted from \(\rv{D}_{n,p}\) in Part~\text{(iii)} is roughly the expected degree, though \(X_\lambda\) can have a small but positive mean. Because of the special structure of a random cluster graph, the degree distribution is essentially the size-biased cluster size distribution. This directly transfers to the (uniform) set partition case. To the best of our knowledge, the size-biased cluster size of a uniform set partition has not been studied before. 
As before, we observe much smaller degrees in the RCG for \(p\leq\nicefrac{1}{2}\) as we would in the unconditioned ER graph. More precisely, we observe Poisson distributed degrees with parameter \(\Theta(\log n)\) at \(p=1/2\). That is, there is a strong concentration around the expected cluster size (cf.~Theorem~\ref{thm:mainClust}). This concentration becomes even stronger in the subcritical regime, where the variance remains bounded as \(n\rightarrow\infty\). The parameter \(\lambda\) appearing in Part~(iii) is a result of some periodicity in the asymptotics of the subcritical degree distribution. Therefore, in order to formulate a limiting result without periodicity, we have to restrict ourselves to subsequences of $n$.
The limiting distribution given in~\eqref{eq:discrete-gauss} resembles a normal distribution. Indeed, it corresponds to the \emph{discrete Gaussian distribution} as defined in \cite{kemp1997characterizations}, where it is additionally shown that this distribution can be characterised as the \emph{maximum entropy distribution} among discrete distributions with specified mean and variance.
We prove Part~(ii) in \Cref{sec:critical:Degree} and Part~(iii) in \Cref{prop:subcritDegree}. An immediate consequence of Theorem~\ref{thm:mainDegree}~(ii) is a CLT for the degree distribution at criticality. 

\medskip

\begin{corollary}\label{CLT_Deg}
	The degree distribution at criticality obeys a central limit theorem, i.e.\ as \(n\to\infty\),
	\[
		\frac{\rv{D}_{n,\nicefrac{1}{2}}-\E\rv{D}_{n,\nicefrac{1}{2}}}{\sqrt{\Var(\rv{D}_{n,\nicefrac{1}{2}})}} \longrightarrow \mathcal{N}(0,1), \quad \text{ in distribution.}
	\]
\end{corollary}

\paragraph{The near-critical regime.}
The results of this paragraph concern the near-critical regime above \(p=\nicefrac{1}{2}\). It is easy to see (cf.\ \Cref{sec:distributionsOfInterest}) that the function \(p\mapsto\P(\rv{C}_{n,p}= 1)\) is strictly increasing and continuous for each \(n\). As a result, for each \(q\in(0,1)\), there exists a corresponding edge probability \(p_n(q)\) such that \(\P(\rv{C}_{n,p_n(q)}=1)=q\). Our next result quantifies the asymptotics of such a critical sequence. 

\begin{figure}[t!]
    \centering
    \includegraphics[width=0.7\linewidth]{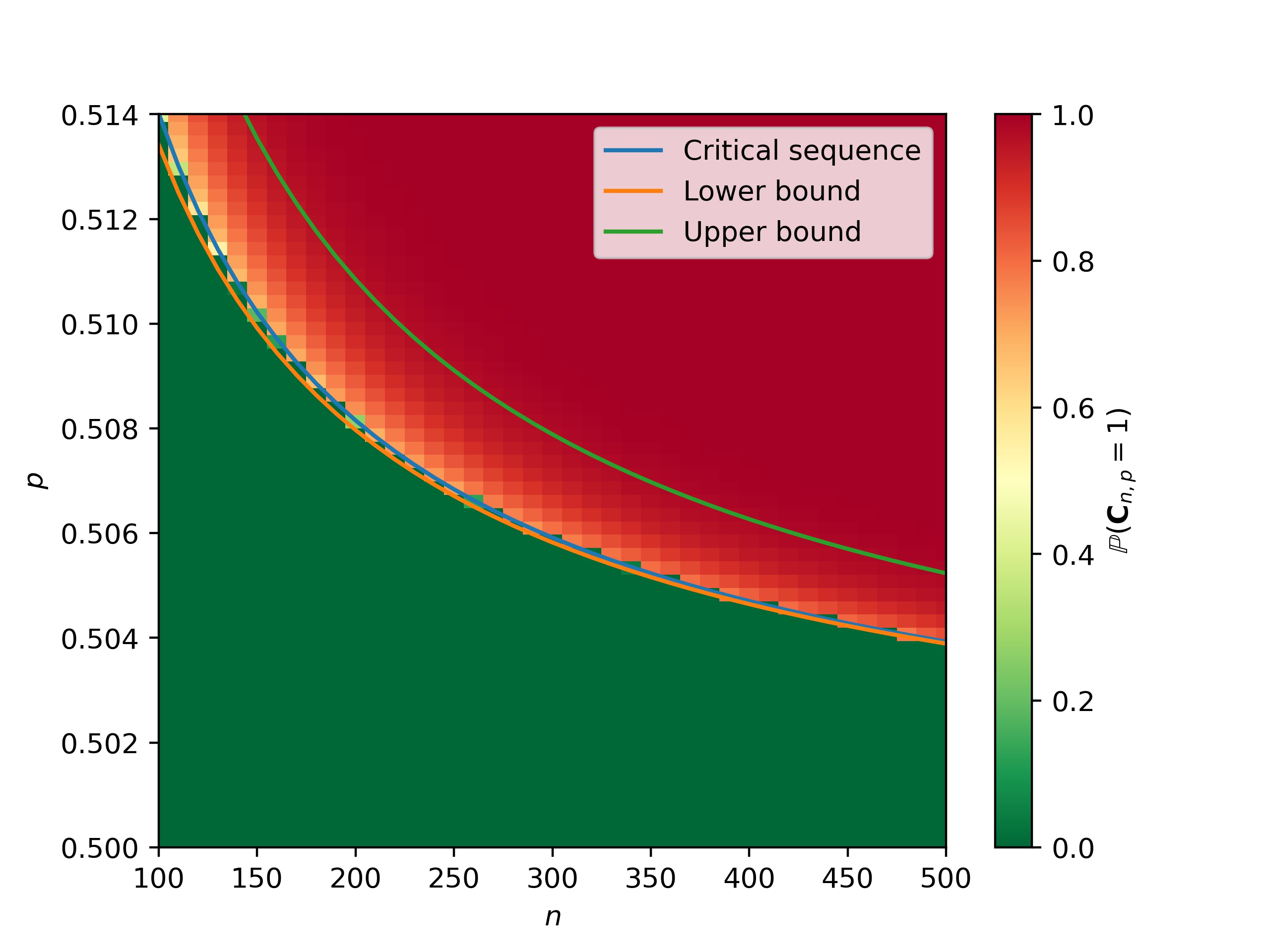}
    \caption{We plot the critical sequence $p_n(\nicefrac{1}{2})$ together with the upper and lower bounds from Lemmas~\ref{lem:critical-lower} and~\ref{lem:critical-upper}. The background's colour is based on $\proba(\rv{C}_{n,p}=1)$, underlining the narrowness of the critical window.}
    \label{fig:phase-transition}
\end{figure}

\medskip

\begin{theorem}[Near-critical regime]\label{thm:mainWindow}
	Consider \(\rv{CG}_{n,p}\). Let \(q\in(0,1)\) and \(p_n(q)\) be a sequence satisfying \(\P(\rv{C}_{n,p_n(q)}=1)=q\). Then
	\begin{equation*}
		p_n(q) =\frac{1}{2}+\frac{\log n}{2n}+\bigO\left(\frac{\log\log n}{n}\right).
	\end{equation*}
\end{theorem}

\medskip

We give the proof of Theorem~\ref{thm:mainWindow} in \Cref{sec:window} by proving asymptotic upper- and lower bounds with the same asymptotics in \Cref{lem:critical-lower} and \Cref{lem:critical-upper}. \Cref{fig:phase-transition} shows the phase transition, together with the critical sequence and the aforementioned bounds for $n\in[100,500]$. Further note that the value $q$ does not appear in the main asymptotics of $p_n(q)$. Hence, the dependence on $q$ will only appear in lower-order terms, which shows that this critical window is quite narrow.

\paragraph{The sparse regime.}
We close this section with results about the sparse regime. Classically, a graph is called sparse if the number of vertices and edges are of the same order of magnitude. For the standard ER graph this is the case for \(p\sim \nicefrac{\lambda}{n}\), which is arguably the most studied regime. Theorem~\ref{thm:mainEdges} already shows that for fixed \(p\leq\nicefrac{1}{2}\) the number of edges in \(\rv{CG}_{n,p}\) is of much smaller order than in the standard case. The following theorem shows that this is no longer true in the sparse regime, where (in expectation) a linear proportion of edges exists in the graph, albeit with much smaller constant. More precisely, we show that for \(\rv{CG}_{n,\lambda/n}\), only a vanishing fraction of vertices have a degree higher than one, so that it almost exclusively consists of isolated vertices and isolated edges. Intuitively, adding an additional edge connecting two isolated vertices comes at the probabilistic costs of \(n^{-1}\). If an additional edge however joined a non-isolated vertex, this would require the addition of at least one more edge to form a clique, leading to a total cost of at least \(n^{-2}\). Hence, to reduce the costs either no edges or edges between isolated vertices are formed.

\medskip

\begin{theorem}[Degrees in the sparse regime]\label{thm:mainDegreeSparse}
    Let $\lambda>0$ and consider \(\rv{CG}_{n,p_n}\) and its degrees \(\rv{D}_{n,p_n}\) where the underlying edge probability is given by a sequence satisfying $p_n\sim\nicefrac{\lambda}{n}$. Then
    \[
    	\lim_{n\to\infty}\proba(\rv{D}_{n,p_n}=0) = \frac{\sqrt{4\lambda+1}-1}{2\lambda},\quad
    	\lim_{n\to\infty}\proba(\rv{D}_{n,p_n}=1)= 1-\frac{\sqrt{4\lambda+1}-1}{2\lambda},
    \]
    In particular, for $p_n\sim 1/n$ we have $\proba(\rv{D}_n=0)\rightarrow\rho^{-1},$
    where $\rho=\frac{\sqrt{5}+1}{2}$ is the golden ratio.
\end{theorem}

\medskip

\begin{figure}
    \centering
    \includegraphics[width=0.8\linewidth]{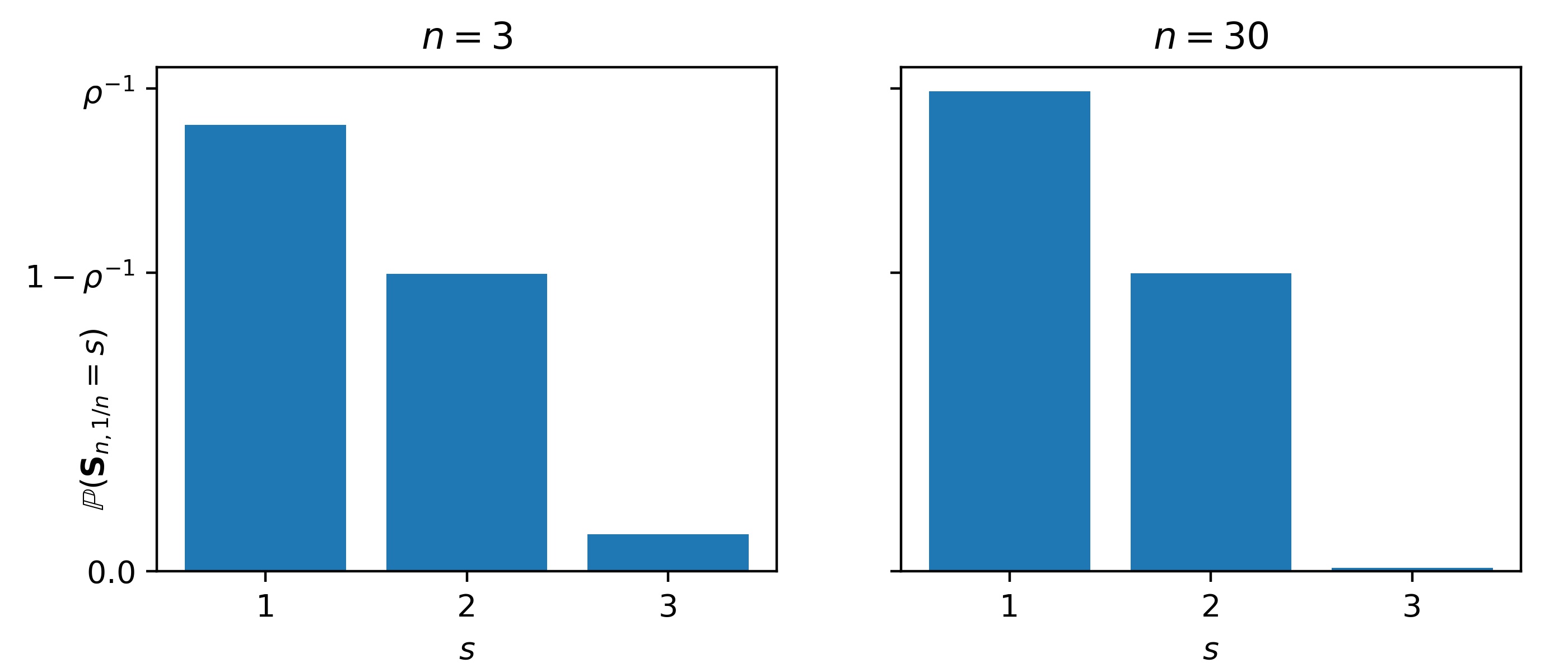}
    \caption{For $n\in\{3,30\}$, we show the distribution of $\rv{S}_{n,1/n}$ to demonstrate the convergence proven in \cref{thm:mainDegreeSparse}. $\rho=\frac{1+\sqrt{5}}{2}$ is the golden ratio.}
    \label{fig:golden-ratio}
\end{figure}
\cref{fig:golden-ratio} shows the distribution of $\rv{S}_{n,1/n}$ for $n\in\{3,30\}$. Even for these small $n$, the convergence $\proba(\rv{S}_{n,1/n}=1)\rightarrow\rho^{-1}$ is already evident in the plots. 
Let us give a more intuitive explanation for the appearance of the golden ratio. We observe from our proof in \Cref{sec:vanishing} for $p_n\sim n^{-1}$,
\[
1+o(1)=\proba(\rv{S}_{n,p_n}=1)+\proba(\rv{S}_{n,p_n}=2)=\frac{B_{n-1}(e^{t_n})}{B_n(e^{t_n})}+\frac{n-1}{n}(1+o(1))\frac{B_{n-2}(e^{t_n})}{B_n(e^{t_n})},
\]
where \(B_n(w)\) are generalised Bell-numbers introduced in \Cref{sec:genFuncCluster}. After multiplying with $B_n(e^{t_n})$, we infer
\[
B_n(e^{t_n})\sim B_{n-1}(e^{t_n})+B_{n-2}(e^{t_n}),
\]
in which we recognise the recurrence relation of the Fibonacci sequence.

We finally consider more general `\(p=o(1)\)'-regimes. We have seen in Theorem~\ref{thm:mainDegree} that for fixed \(p<\nicefrac{1}{2}\) the degrees strongly concentrate around their expectations. This raises the question whether it is possible to only observe clusters (resp.\ degrees) of a given size when choosing an appropriate \(p=o(1)\). Since still \(p\to 0\) in this situation, we also refer to this as a sparse regime. Our last theorem of this sections states the requirements for such sequences. Both, \Cref{thm:mainDegreeSparse} and the following~\Cref{thm:mainFixedDegSparse} are proven in \Cref{sec:vanishing}. 

\medskip

\begin{theorem}[Fixed degrees in sparse regimes]\label{thm:mainFixedDegSparse}
    Define for any $d\in\mathbb{N}\cup\{0\}$ the sequence
    \[
    	p_n=n^{-\frac{2}{(d+1)^2}+o(1)}, 
    \]
    and consider the RCG \(\rv{CG}_{n,p_n}\) and its degrees \(\rv{D}_{n,p_n}\), then
    \[
    	\lim_{n\to\infty}\proba(\rv{D}_n=d)= 1.
    \]
    Furthermore, for any $d'\neq d$,
    \begin{equation}\label{eq:vanishing-asymptotics}
        \proba(\rv{D}_n=d')=n^{-\left(\frac{d'-d}{d+1}\right)^2+o(1)}.
    \end{equation}
\end{theorem}

\medskip

Let us further compare the `sparse cluster graph' with the standard sparse \ER graph. For $p_n=n^{-2/(d+1)^2}$ and $d>1$, the expected degree of an \ER graph grows to infinity, while it converges to $d$ in $\rv{CG}_{n,p_n}$ by \Cref{thm:mainFixedDegSparse} showing that the conditioning reduces the number of edges by an order of magnitude when $n\cdot p_n\rightarrow\infty$.

For $p_n=\lambda n^{-1}$, the degree distribution of an \ER graph converges to a Poisson distribution with parameter $\lambda$ while it converges to a Bernoulli distribution by \Cref{thm:mainDegreeSparse}. In particular, the expected degree in $\rv{CG}_{n,p_n}$ converges to $\proba(\rv{D}_{n,p_n}=1)<1$ for any  fixed value of $\lambda$. Thus, the number of edges is of the same order as in an \ER graph, but the expected degree is smaller.

Let us additionally consider $p_n=2\lambda n^{-2}$. Here, the number of edges in an \ER graph converges in probability to a Poisson distribution with parameter $\lambda$. Moreover, the expected number of \emph{wedges} (connected subgraphs of size three that contain two edges) is given by $3{n\choose 3}p_n^2(1-p_n)=\bigO(n^{-1})$. Hence, the probability of the graph having \emph{any} wedge vanishes for large $n$. However, a graph that does not contain any wedge is a cluster graph. Thus, for $p_n=\Theta(n^{-2})$, also the classical \ER graph is a cluster graph with high probability. Hence, for this choice of \(p_n\) or even faster decreasing sequences. the effect of the conditioning becomes insignificant. 

Finally, let us consider the case when $p_n=n^{-\alpha}$ for $\alpha<1$ being small. Because the decay of $p_n$ is much slower than linear, one may expect \(\rv{CG}_{n,p_n}\) to behave similarly to the constant \(p<\nicefrac{1}{2}\) case. Let us consider for illustration $p_n\sim n^{-2/(d+1)^2}$.  Then $\log(\nicefrac{1-p_n}{p_n})=2\log(n)/(d+1)^2+\bigO(p_n)$. Treating this expression as being constant and substituting this into expression in \cref{thm:mainDegree}~(iii) yields \(\E[\rv{D}_{n,p_n}]\approx d\), which coincides with \Cref{thm:mainFixedDegSparse}. Here, to derive this approximation, we applied results of \cref{prop:subcritDegree} and \cref{lem:subcritical-tau-asymp}.

\subsection{Application to community detection} \label{sec:Community:detection}
In this section, we discuss the implications of the results presented in \cref{sec:main:results} to the field of community detection, and modularity maximisation in particular.
Community detection is one of the main tasks in modern network science~\cite{fortunato2007}, where one aims to partition the network nodes into well-connected \emph{communities}.
Because community detection is an ill-defined problem~\cite{fortunato2016community}, a large number of different community detection methods have been proposed, each of them based on different ideas of what a community ought to look like. This makes it challenging to decide which of these many methods is best for a given network.
One of the most widely-used community detection methods is to optimise a quantity known as \emph{modularity} over the set of partitions~\cite{newman2004finding}. 
\emph{Modularity} measures the excess of edges inside the communities compared to the expected number of edges under a \emph{null model}. A null model is a random graph model without community structure. The most widely-used null models are the Chung-Lu (CL) and the Erd\H{o}s-R\'enyi (ER) random graph model. Modularity comes with a \emph{resolution parameter} that controls the granularity of the obtained clustering~\cite{reichardt2006statistical,traag2011narrow}. That is, by increasing this parameter value, one obtains a larger number of smaller communities. 
Let $G$ be the graph whose communities are to be detected and let $G_C\in\mathcal{CG}_n$ be a cluster graph on the same set of vertices, which represents the partition into communities.
The ER-modularity of the partition $G_C\in\mathcal{CG}_n$ for the resolution parameter $\gamma$ is given by 
\[
\textrm{ERM}(G_C;G,\gamma)=\frac{1}{m(G)}\left(m(G\cap G_C)-\gamma\cdot m(G_C)\right),
\]
where $m(G)$ denotes the number of edges in $G$, $m(G_C)$ denotes the number of edges in $G_C$ (the number of intra-community pairs), and the number of intra-community edges is denoted by $m(G\cap G_C)$ (the intersection between the edge sets of $G$ and $G_C$). ERM increases when either two sets of nodes with edge density higher than $\gamma$ between them are merged, or when two sets of nodes with edge density lower than $\gamma$ between them are separated.
An advantage of modularity maximisation is that one does not have to specify the number of communities or their sizes.
For a particular value of the resolution parameter (corresponding to $\gamma(p_{\text{in}},p_{\text{out}},1/2)$ in~\eqref{eq:bayesian-modularity}), maximising ER-modularity is equivalent to maximising the likelihood for a Planted Partition Model (PPM) random graph~\cite{newman2016equivalence}. Here, likelihood refers to the probability that the model with given parameters yields the observed graph.
However, this equivalence only provides an interpretation of ER-modularity for one particular value of the resolution parameter. In this section, we show that maximising ER-modularity for \emph{any} resolution parameter is equivalent to maximising the Bayesian posterior probability of a PPM for a prior distribution corresponding to $\rv{CG}_{n,p}$ for some value $p$. 
Importantly, we do \emph{not} consider detecting communities \emph{in} cluster graphs (which would be extremely trivial), but instead, we use the distribution over cluster graphs as a prior for the community structure. The network that we detect communities in can be any undirected simple graph.

\paragraph{Planted Partition Model.}
The Planted Partition Model (PPM) is the simplest random graph model with community structure~\cite{holland1983stochastic}. In this model, two vertices are independently connected with probability $p_{\text{in}}$ if they belong to the same community and with probability $p_{\text{out}}$ if they belong to different communities. 

\begin{theorem}\label{thm:bayesian-modularity}
    For a Planted Partition Model with parameters $p_{\text{in}},p_{\text{out}}\in(0,1)$ and a community partition drawn from $\rv{CG}_{n,p}$, maximising the Bayesian posterior probability is equivalent to maximising ER-modularity with resolution parameter
    \begin{equation}\label{eq:bayesian-modularity}
        \gamma(p_{\text{in}},p_{\text{out}},p) = \frac{\log\left(\frac{1-p_{\text{out}}}{1-p_{\text{in}}}\frac{1-p}{p}\right)}{\log\left(\frac{p_{\text{in}}(1-p_{\text{out}})}{p_{\text{out}}(1-p_{\text{in}})}\right)}.
    \end{equation}
\end{theorem}
\begin{proof}
    Consider a graph $G$ with $m(G)$ edges and a cluster graph $G_C$ with $m(G_C)$ edges. Let $m(G\cap G_C)$ be the number of edges that are shared by $G$ and $G_C$ (the intersection of their edges sets). We assume the parameters $p_{\text{in}}$, $p_{\text{out}}$ to be known. The likelihood of a partition $G_C$ after observing a graph $G$ is
    \[
        \textrm{Likelihood}(G_C;G)=p_{\text{in}}^{m(G\cap G_C)}(1-p_{\text{in}})^{m(G_C)-m(G\cap G_C)}p_{\text{out}}^{m_G-m(G\cap G_C)}(1-p_{\text{out}})^{\binom{n}{2}-m(G_C)-m_G+m(G\cap G_C)}.
    \]
    The prior probability of a cluster graph $G_C$ is
    \[
        \textrm{Prior}(G_C)=\frac{\left(\frac{p}{1-p}\right)^{m(G_C)}}{B_n(\nicefrac{p}{1-p})},
    \]
    where $B_n(\nicefrac{p}{1-p})$ is a normalising constant that does not depend on $G_C$ and is introduced in \cref{sec:distributionsOfInterest}. 
    The Bayesian posterior probability is
    \begin{equation}\label{eq:bayes-general}
        \textrm{BayesPosterior}(G_C;G)=\textrm{Likelihood}(G_C;G)\cdot \textrm{Prior}(G_C).
    \end{equation}
    When maximising $\textrm{BayesPosterior}(G_C;G)$ over the set of cluster graphs $G_C$ for fixed $G$, one can omit all factors that do not depend on \(G_C\). This results in
    \[
        \textrm{BayesPosterior}(G_C;G)\propto \left(\frac{p_{\text{in}}}{1-p_{\text{in}}}\frac{1-p_{\text{out}}}{p_{\text{out}}}\right)^{m(G\cap G_C)} \left(\frac{1-p_{\text{in}}}{1-p_{\text{out}}}\frac{p}{1-p}\right).^{m(G_C)}
    \]
    To show the equivalence to maximising ER-modularity, we take the log of the right-hand-side and divide it by the constant (\wrt $G_C$) $m_G\log\left(\frac{p_{\text{in}}}{1-p_{\text{in}}}\frac{1-p_{\text{out}}}{p_{\text{out}}}\right)$ to obtain
    \begin{align*}
    \frac{1}{m_G}\left(m(G\cap G_C)+\left(\frac{t+\log\frac{1-p_{\text{in}}}{1-p_{\text{out}}}}{\log\frac{p_{\text{in}}(1-p_{\text{out}})}{p_{\text{out}}(1-p_{\text{in}})}}\right)\cdot m(G_C)\right)
    &=\frac{1}{m_G}\left(m(G\cap G_C)-\left(\frac{\log\frac{1-p_{\text{out}}}{1-p_{\text{in}}}\frac{1-p}{p}}{\log\frac{p_{\text{in}}(1-p_{\text{out}})}{p_{\text{out}}(1-p_{\text{in}})}}\right)\cdot m(G_C)\right)\\
    &=\textrm{ERM}(G_C;G,\gamma).
    \end{align*}
\end{proof}

Therefore, when maximising ER-modularity, one is implicitly assuming a prior distribution over the partitions into communities. 
In order to understand the resolution parameter value's impact on the obtained communities, we investigate the prior distribution. 
One nice property of this prior distribution is that it belongs to the exponential family of probability distributions, so that it maximises entropy among the partition distributions with expected number of edges (in the cluster graph) equal to $\E[\rv{M}_{n,p}]$. In a way, assuming this prior distribution is thus equivalent to imposing the \emph{soft constraint} that the expected number of intra-community pairs must be equal to some value. This motivates studying $\E[\rv{M}_{n,p}]$. In addition, it suggests that this community detection method should work well for values of $p$ for which $\E[\rv{M}_{n,p}]$ is close to the number of edges in the cluster graph corresponding to the planted communities.

The choice $p=1/2$ is equivalent to assuming a uniform prior ($\textrm{Prior}(G_C)\propto1$), so that by~\eqref{eq:bayes-general}, maximising this posterior is equivalent to likelihood maximisation. It is known that maximising this likelihood is biased towards detecting communities of sizes roughly $\log n$~\cite{zhang2020statistical}, which indeed corresponds to the typical clique size in $\rv{CG}_{n,1/2}$, cf.~ \cref{thm:mainDegree}~(ii).
Therefore, whenever the communities are of sizes close to $\log n$, the choice $p=\tfrac{1}{2}$ likely yields communities of the desired granularity. 

In the above, we assume that the parameters $p_{\text{in}},p_{\text{out}}$ are known, which is common in the theoretical analysis of such community detection methods~\cite{newman2016equivalence,prokhorenkova2019community}. The estimation of these parameters is beyond the scope of this article.

\paragraph{Detecting small communities.}
If the communities are of significantly smaller size than $\log n$, then \(p<\nicefrac{1}{2}\) will be a good choice to detect communities of the desired granularity. \cref{thm:mainFixedDegSparse} suggests that if the number of communities grows linearly with $n$, we should choose $p=n^{-2/s^2+o(1)}$, where $s$ is the size of a typical community. 
Note that the factor $n^{-2/s^2}$ from \cref{thm:mainDegreeSparse} vanishes extremely slowly. To illustrate, $s=10$ and $n=10^{50}$ results in $n^{-2/s^2}=0.1$, while $0.0009=n^{-2/s^2}/\log n=n^{-2/s^2+o(1)}$. Hence, the $n^{o(1)}$ factor in $p=n^{-2/s^2+o(1)}$ has a big impact on the detected community sizes for fixed $n$. 
To find more precise values for $p$, one can use the asymptotics given in \cref{thm:mainEdges}~(iii). 

\paragraph{Detecting large communities.} In many practical applications, the community sizes are typically larger than $\log n$. For example, if the number of communities remains bounded as $n\rightarrow\infty$, then the number of edges in the cluster graph is $\Theta(n^2)$. The phase transition that we describe in Theorem~\ref{thm:mainWindow} suggests the choice $p_n=\tfrac{1}{2}+\tfrac{\log n}{2n}+\bigO(\tfrac{\log\log n}{n})$ for this setting.
In the next experiment, we illustrate that this choice of $p_n$ leads to significantly better community-detection performance than $p=\tfrac{1}{2}$, despite the fact that they only differ by a vanishingly small term.

In this experiment, we assess the performance of maximisation ERM with resolution parameter $\gamma(p_{in},p_{out},p)$ as given in~\eqref{eq:bayesian-modularity}, for different values of $p$. The graphs on which we test these detection methods are generated from a PPM where the planted partition consists of five communities of 200 vertices each. We choose $p_{\text{in}}$ and $p_{\text{out}}$ so that every vertex has in expectation $10$ neighbours inside its community and $10$ neighbours outside its community.
We maximise ERM using the Louvain algorithm~\cite{blondel2008fast}, a standard modularity maximisation algorithm. We generate $20$ PPM graphs and run the Louvain algorithm on each of them for several values of $p$ in the range $[0.5,0.51]$. 
In particular, we focus on the critical window between our asymptotical lower bound $p^{L}_{1000}$ and our asymptotical upper bound $p^U_{1000}$ for the critical value, as given in~\eqref{eq:critical-bounds}.
In Figure~\ref{fig:community-detection}, we assess the average performances of ERM maximisation on this PPM model for different values of $p$. Figure~\ref{fig:community-detection-granularity} shows the number of edges in the detected community partition. We see that this (average) number of edges increases with $p$.
For $p\in[0.5,p^{L}_{1000}]$, this leads to a partition consisting of smaller communities than the planted partition, while $p>p^{U}_{1000}$ leads to a partition that consists of larger communities.
In Figure~\ref{fig:community-detection-performance}, we assess the performance of the detection method by measuring the similarity between the detected and the planted partitions. We quantify their similarity by the correlation coefficient between these partitions~\cite{gosgens2021systematic}. The values shown are averaged over the $20$ different graphs. Again, we see that the best performance is found for $p$ inside the critical window.
These results demonstrate that whenever communities are larger than $\log n$, the resolution parameter $\gamma(p_{in},p_{out},p_n)$ for $p_n=\tfrac{1}{2}+\tfrac{\log n}{2n}+\bigO(\tfrac{\log\log n}{n})$ may lead to significantly better performance than the resolution parameter $\gamma(p_{in},p_{out},\tfrac{1}{2})$, despite the fact that these resolution parameter values are extremely close to each other. The fact that the performance is so sensitive to small changes in $p$ can be explained by the phase transition of the RCG.

\begin{figure}[t!]
\centering
\begin{subfigure}{.5\textwidth}
  \centering
  \includegraphics[width=\linewidth]{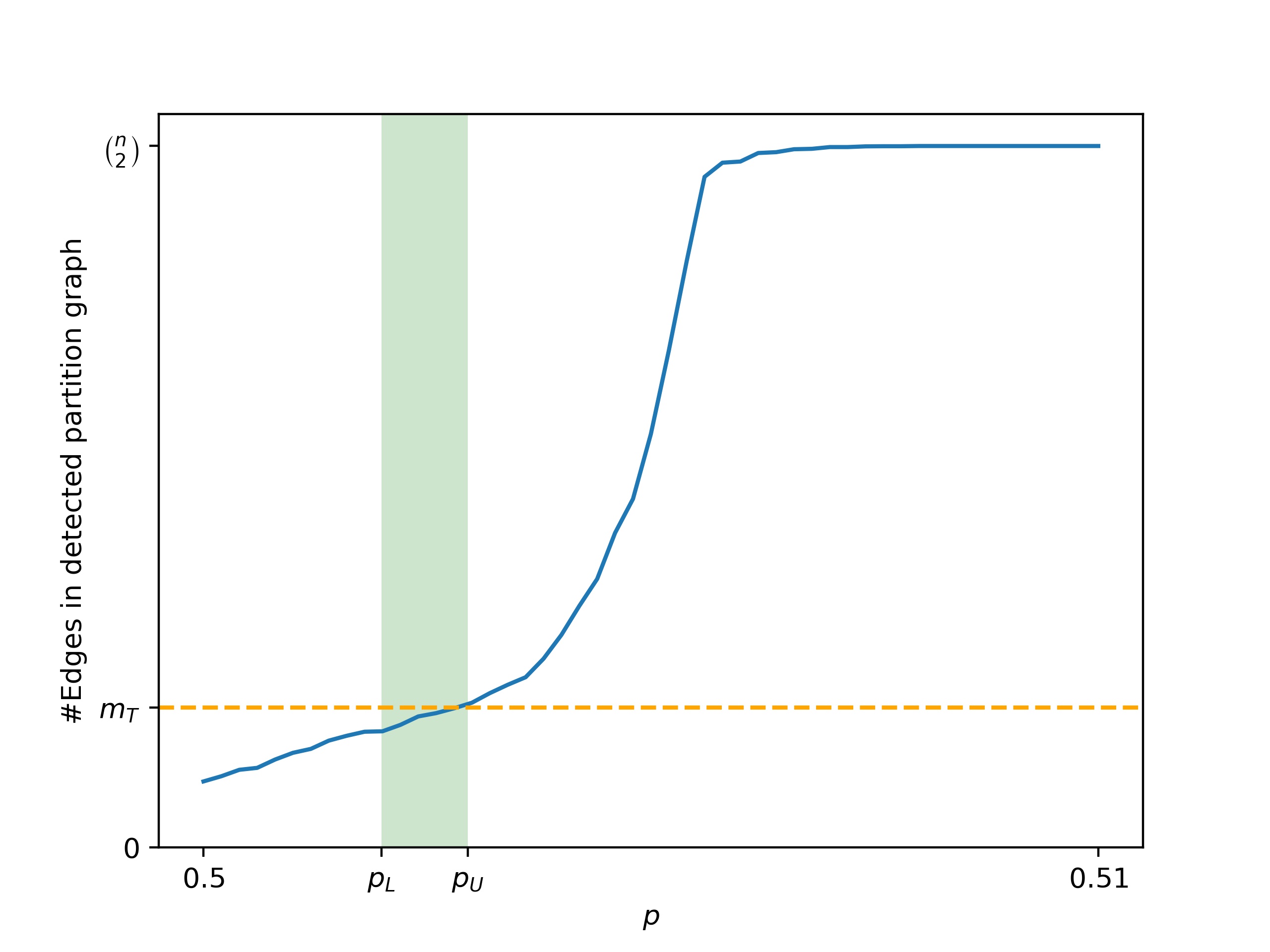}
  \caption{The number of edges in the detected cluster graph.}
  \label{fig:community-detection-granularity}
\end{subfigure}%
\begin{subfigure}{.5\textwidth}
  \centering
  \includegraphics[width=\linewidth]{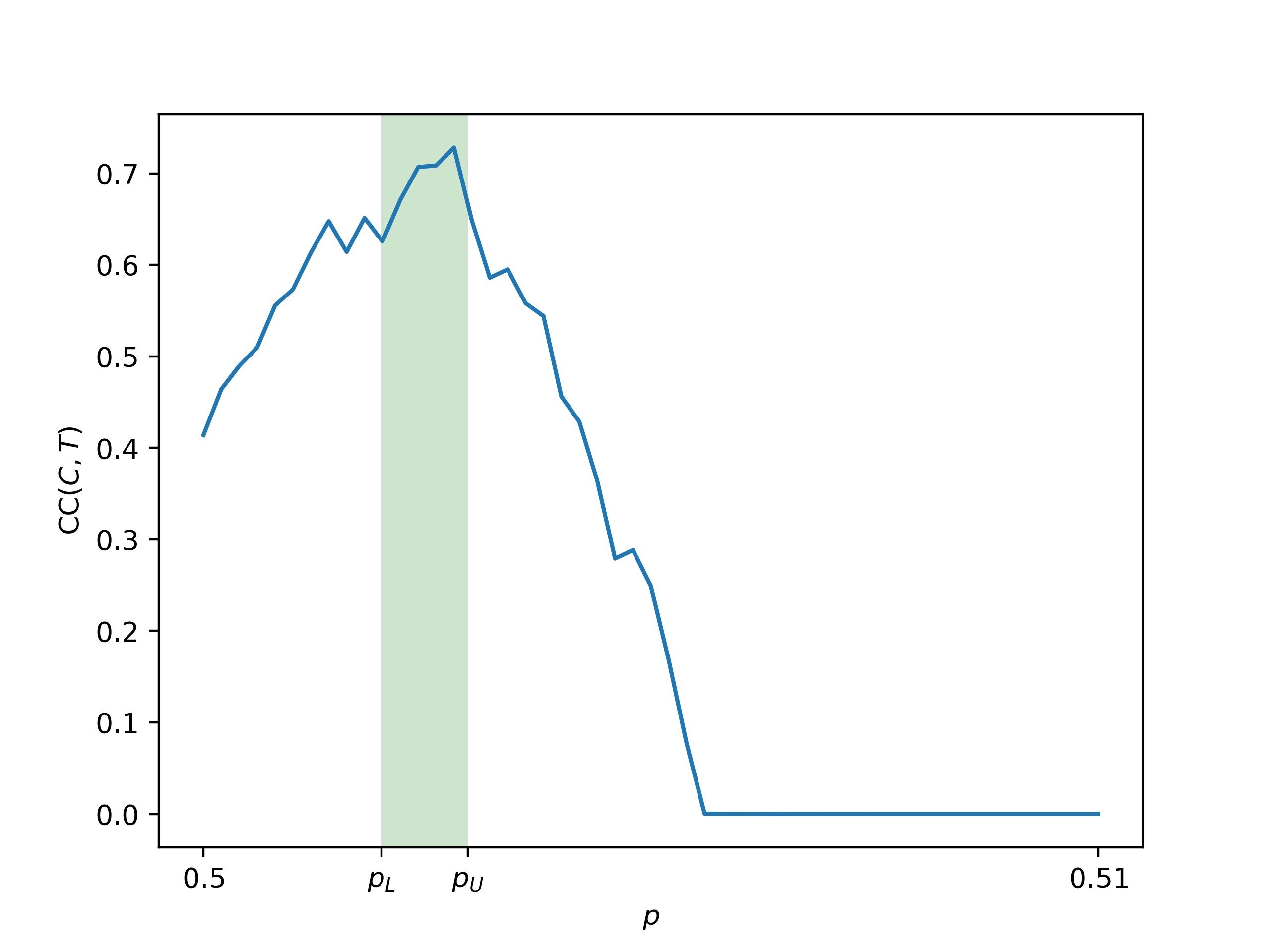}
  \caption{The similarity between the detected planted partition.}
  \label{fig:community-detection-performance}
\end{subfigure}
\caption{We generate 20 PPMs with $n=1000$ vertices divided into 5 communities of size $200$ each. The parameters $p_{\text{in}}$ and $p_{\text{out}}$ are chosen such that each vertex has (in expectation) 10 neighbors inside its community and 10 neighbors outside its community. We detect communities by maximising ERM with resolution parameter $\gamma(p_{\text{in}},p_{\text{out}},p)$ as given in~\eqref{eq:bayesian-modularity}, for various values of $p$. Figure~\ref{fig:community-detection-granularity} shows the number of edges in the detected cluster graph, while Figure~\ref{fig:community-detection-performance} shows the correlation coefficient~\cite{gosgens2021systematic} between the detected and true partitions, which is a measure of similarity between these partitions.}
\label{fig:community-detection}
\end{figure}

\paragraph{Other Bayesian approaches.} In the previous paragraphs, we took modularity maximisation and interpreted it in terms of Bayesian statistics. In other works~\cite{peixoto2019bayesian}, a more general Bayesian framework for detecting stochastic block models 
has been developed using a different prior distribution over the set of partitions. However, the disadvantage of these (more sophisticated) Bayesian approaches, is that the maximisation of the corresponding Bayesian posterior is challenging and one has to resort to Markov Chain Monte Carlo methods to find near-optimal partitions~\cite{peixoto2014efficient,peixoto2020merge}. The advantage of modularity is that the function that is to be optimised is quite simple and well-studied~\cite{newman2006finding,blondel2008fast,traag2019louvain}. There even exist exact optimisation algorithms~\cite{aref2022bayan}.
Another disadvantage of the more advanced Bayesian approaches is that they do not allow for the type of theoretical analysis that we perform in this work, and that the only feasible way of studying them is through experiments and simulations.
We emphasise that the aim of this paper is to better understand modularity maximisation, and not to advocate the usage of modularity maximisation over other Bayesian methods.

\paragraph{Degree-Corrected PPM and Chung-Lu modularity.}
For the most popular form of modularity, Chung-Lu (CL) modularity, there is a similar equivalence~\cite{newman2016equivalence} to the likelihood of a Degree-Corrected Planted Partition Model (DCPPM). However, a major caveat of this equivalence is that it only holds whenever each community has the same total degree in expectation~\cite{peixoto2023descriptive}, which is a strong assumption that rarely holds in practice.
In this work, we show that the RCG relates ER-modularity to the Bayesian posterior of a PPM. One may hope to obtain a similar relation between CL-modularity and the DCPPM.
However, unfortunately this is not the case: the posterior probability of a DCPPM with RCG prior cannot be rewritten to a form similar to CL-modularity. Even when one takes as prior a CL random graph conditioned on being a cluster graph, the corresponding posterior still cannot be rewritten to CL-modularity.
For further information about Bayesian methods for the DCPPM, we refer the reader to~\cite{newman2016equivalence,peixoto2023descriptive}.

        \section{Exact expressions for the statistics of random cluster graphs}\label{sec:distributionsOfInterest}

We obtained the results listed in \cref{sec:main:results}
by applying techniques from analytic combinatorics \cite{FS09},
which is an approach based on generating function manipulations.
This section gives a brief overview of this field,
derives exact expressions for various statistics of random cluster graphs,
and presents techniques to extract limit laws from generating functions.

    \subsection{Generating functions}

We consider graphs with labelled vertices
and unlabelled undirected edges,
where loops and multiple edges are forbidden.
The numbers of vertices and edges of a graph $G$
are denoted by $n(G)$ and $m(G)$.
Consider a graph family $\mA$,
with $A_{n,m}$ denoting the number of graphs
with $n$ vertices and $m$ edges.
We associate to $\mA$ the generating function
\[
    A(w,z) =
    \sum_{G \in \mA}
    w^{m(G)}
    \frac{z^{n(G)}}{n(G)!}
    =
    \sum_{n,m \geq 0}
    A_{n,m}
    w^m
    \frac{z^n}{n!}.
\]
Thus, the number of graphs with $n$ vertices and $m$ edges in $\mA$
is obtained by extracting the coefficient $w^m z^n$ and multiplying by $n!$,
which is denoted by
\[
    A_{n,m} = n! [w^m z^n] A(w,z).
\]
For example, the generating function of nonempty cliques (complete graphs) is
\[
    C(w,z) =
    \sum_{n \geq 1}
    w^{\binom{n}{2}}
    \frac{z^n}{n!}.
\]

    \subsection{Symbolic method}\label{sec:symb}

The \emph{symbolic method} is a dictionary
that translates operations on combinatorial families
into analytic operations on their generating functions.
It is exposed in detail in \cite{BLL97,FS09}. We provide a brief overview in this section.

\paragraph{Disjoint union.}
Consider two disjoint graph families $\mA$ and $\mB$,
with generating functions $A(w,z)$ and $B(w,z)$.
Then the disjoint union $\mD = \mA \uplus \mB$ has generating function
\[
    D(w,z) =
    \sum_{G \in \mA \uplus \mB}
    w^{m(G)}
    \frac{z^{n(G)}}{n(G)!}
    =
    A(w,z) + B(w,z).
\]
That is, the disjoint union $\uplus$ of families
translates into the sum of their generating functions.

\paragraph{Relabeled product.}
Given $\mA$ and $\mB$ as before,
consider the family $\mD$ of all pairs $(G_A', G_B')$
obtained by
\begin{itemize}
\item
taking graphs $G_A$ and $G_B$ respectively in $\mA$ and $\mB$,
\item
choosing an arbitrary subset $S_A$ of $\{1, \ldots, n(G_A) + n(G_B)\}$
and its complement $S_B$,
\item
constructing a graph $G_A'$ (\resp $G_B'$) from $G_A$ (\resp $G_B$)
by replacing the labels of the vertices with $S_A$ (\resp $S_B$)
while keeping their respective order.
\end{itemize}
This construction is called a \emph{relabeled product}
and it ensures that no two vertices in $(G_A', G_B')$
share the same label.
Observe that each couple $(G_A, G_B)$
corresponds to exactly $\binom{n(G_A) + n(G_B)}{n(G_A)}$
couples $(G_A', G_B')$
(number of ways to choose the labels $S_A$),
so that the generating function of $\mD$ reads
\begin{align*}
    D(w,z) &=
    \sum_{\substack{G_A \in \mA\\ G_B \in \mB}}
    \binom{n(G_A) + n(G_B)}{n(G_A)}
    w^{m(G_A) + m(G_B)}
    \frac{z^{n(G_A) + n(G_B)}}{(n(G_A) + n(G_B))!}
    \\&=
    \sum_{\substack{G_A \in \mA\\ G_B \in \mB}}
    w^{m(G_A)} w^{m(G_B)}
    \frac{z^{n(G_A)}}{n(G_A)!}
    \frac{z^{n(G_B)}}{n(G_B)!}
    \\&=
    A(w,z) B(w,z).
\end{align*}
Thus, the relabeled product of combinatorial families
corresponds to the product of their generating functions.

\paragraph{Sequence of length $k$.}
The previous operation extends to sequences of length $k$.
Let $\mD$ denote the family of sequences of $k$ graphs from $\mA$,
relabeled as previously.
Then by induction on $k$,
\[
    D(w,z) = A(w,z)^k.
\]

\paragraph{Set of size $k$.}
Consider the family $\mD$ obtained
by taking an (unordered) set of $k$ relabeled graphs from $\mA$.
Each such set corresponds to exactly $k!$
sequences of length $k$, so
\[
    k! D(w,z) = A(w,z)^k
    \qquad \text{and} \qquad
    D(w,z) = \frac{A(w,z)^k}{k!}.
\]

\paragraph{Set of arbitrary size.}
Assume $\mA$ does not contain the empty graph
(whose generating function is $w^0 \frac{z^0}{0!} = 1$),
so $A(w,0) = 0$ and let us define $\mD$
as the family of all finite subsets of relabeled graphs from $\mA$.
Then $\mD$ is the disjoint union
of the sets of size $k$ from $\mA$, for all $k$.
Hence,
\[
    D(w,z) =
    \sum_{k \geq 0}
    \frac{A(w,z)^k}{k!}
    =
    e^{A(w,z)},
\]
where the exponential is considered as a formal power series
(we do not assume that the series has a non-zero radius of convergence).

\paragraph{Pointing.}
Consider the family $\mD$ obtained
from $\mA$ by distinguishing
a vertex in all possible ways in each graph.
Each graph $G$ from $\mA$ has then $n(G)$ copies in $\mD$
(each with a different distinguished vertex),
so
\[
    D(w,z) =
    \sum_{G \in \mA} n(G) w^{m(G)} \frac{z^{n(G)}}{n(G)!}
    = z \partial_z A(w,z),
\]
where $\partial_z$ denotes the derivative \wrt $z$.

By definition a cluster graph is a graph
where each connected component is a clique. We denote the family of cliques by $\mC$, and we denote the family of cluster graphs by $\mathcal{CG}$.
Applying the symbolic method,
we deduce the following elementary lemma.

\begin{lemma}
\label{th:Gnm:cliques}
Let $c(G)$ denote
the number of components
in the cluster graph $G$,
and define
\[
    CG(w,z,u) =
    \sum_{G \in \mathcal{CG}}
    u^{c(G)}
    w^{m(G)}
    \frac{z^{n(G)}}{n(G)!},
\]
then
\[
    CG(w,z,u) =
    e^{u C(w,z)},
\]
where
\[
    C(w,z) =
    \sum_{n \geq 1}
    w^{\binom{n}{2}}
    \frac{z^n}{n!},
\]
is the generating function of non-empty cliques.
\end{lemma}

\begin{proof}
A cluster graph is a set of cliques,
with relabeled vertices.
The generating function of cliques is
\[
    C(w,z) =\sum_{G\in\mC}w^{m(G)}\frac{z^{n(G)}}{n(G)!}=
    \sum_{n \geq 1}
    w^{\binom{n}{2}}
    \frac{z^n}{n!}.
\]
The result is obtained by marking each clique
with the variable $u$, noticing that $C(w,z,u)= u C(w,z)$ 
and applying the set construction from the symbolic method (see \textquotedblleft set of arbitrary size").
\end{proof}

For any nonnegative integer $r$, we define $C_r(w,z)$ as the generating function that results from applying $r$ times the pointing operation  to $\mC$. That is,
\[
    C_r(w,z) = (z\partial_z)^rC(w,z)=
    \sum_{n \geq 1}
    n^r
    w^{\binom{n}{2}}
    \frac{z^n}{n!}.
\]
This allows us to construct the following generating function:

\begin{lemma}
\label{th:Gnm:degree}
Consider the family $\mathcal{CG}^{\star}$ of cluster graphs $G$
with a distinguished vertex,
which's degree is denoted by $d(G)$,
and define its generating function by
\[
    CG^{\star}(w,z,u) =
    \sum_{G \in \mathcal{CG}^{\star}}
    u^{d(G)}
    w^{m(G)}
    \frac{z^{n(G)}}{n(G)!},
\]
then
\[
    CG^{\star}(w,z,u) =
    \frac{1}{u}
    C_1(w, u\, z)
    e^{C(w,z)}.
\]
\end{lemma}

\begin{proof}
Using the pointing operation from the symbolic method,
the generating function of cliques
with one distinguished vertex is
\[
    z \partial_z C(w,z) =
    C_1(w,z).
\]
A cluster graph with a distinguished vertex
has a unique decomposition as
the relabelled product of
a clique with a distinguished vertex
and a cluster graph containing the other cliques.
If we add a vertex $u$
to mark the degree of the marked vertex, which is equal
to the number of edges of the clique containing it
minus $1$, we obtain $\frac{C_1(w, u z)}{u}$.
The result of the lemma follows from the symbolic method (see \textquotedblleft relabelled product").

\end{proof}

The last two lemmas provide a description
of the parameters of interest
(number of cliques, degree of a random vertex)
that would be suitable if we were considering
random cluster graphs with a given number of vertices and edges.
The next subsection builds a bridge
for the analysis of random \ER graphs
conditioned to be cluster graphs.

    \subsection{Generating functions for random cluster graphs} \label{sec:genFuncCluster}

The \emph{probability generating function} $\PGF_X(u)$
of a random variable $X$ taking non-negative integer values
is defined as
\[
    \PGF_X(u) = \E[u^X]=
    \sum_{k \geq 0}
    \proba(X = k) u^k.
\]
In this section, we will derive the probability generating functions of the number of cliques $\rv{C}_{n,p}$, the number of edges $\rv{M}_{n,p}$ and the degree of a uniformly chosen vertex $\rv{D}_{n,p}$ of a random cluster graph \(\rv{CG}_{n,p}\). First, we will use the methods from \cref{sec:symb} to derive the probability mass function of $\rv{CG}_{n,p}$. We denote by $\mathcal{CG}$ the family of cluster graphs and by $\mathcal{CG}_n$ the set of cluster graphs of size $n$. In order to derive the desired probability generating functions, we rely crucially on the \(n\)-th coefficient of the generating function \(e^{C(w,z)}\). We define the function
\begin{equation}\label{eq:partFunction}
    B_n(w)=n! [z^n] e^{C(w,z)},
\end{equation}
which will play the role of a normalising factor in the following probability distributions. For \(w=1\), the functions reduces to \(B_n(1)=|\mathcal{CG}|_n=B_n\), where \(B_n\) is the \(n\)-th \emph{Bell number}. The bell numbers play a crucial role in the analysis of uniform set partitions. In this light, \(B_n(w)\) can be seen as a generalisation of the Bell numbers. 

\begin{lemma}\label{lem:CG:pmf}
    Let $G\in\mathcal{CG}_n$ be a cluster graph of size $n$ and let $\rv{CG}_{n,p}$ be an \ER graph on $n$ vertices with connection probability $p$, conditioned on being a cluster graph. Then
    \[
    \P(\rv{CG}_{n,p}=G)=\frac{\left(\tfrac{p}{1-p}\right)^{m(G)}}{B_n(\nicefrac{p}{(1-p)})},
    \]
    where the \emph{partition function} $B_n(w)$ is given in~\eqref{eq:partFunction}.
\end{lemma}
\begin{proof}
    Let $\rv{ER}_{n,p}$ denote a random \ER graph consisting of $n$ vertices with edge probability $p$. Its probability mass function is given by
\begin{equation*}
\P(\rv{ER}_{n,p}=G)=(1-p)^{\binom{n}{2}-m(G)}p^{m(G)}=(1-p)^{\binom{n}{2}}w^{m(G)},
\end{equation*}
for $w=\nicefrac{p}{(1-p)}$ and a graph $G$. The probability that $\rv{ER}_{n,p}$ is a cluster graph is hence given by
\begin{equation}\label{PER_clust}
\P(\rv{ER}_{n,p}\in\mathcal{CG}_n)=(1-p)^{\binom{n}{2}}\sum_{G\in\mathcal{CG}}w^{m(G)}.
\end{equation}
From \cref{th:Gnm:cliques}, we derive 
\[
\sum_{G\in \mathcal{CG}_n}w^{m(G)}=n![z^n]CG(w,z,1)=n!e^{C(w,z)}=B_n(w),
\]
implying
\[
\P(\rv{ER}_{n,p}\in\mathcal{CG}_n)=(1-p)^{\binom{n}{2}}B_n(w).
\]
Since the distribution of $\rv{CG}_{n,p}$ is defined as the law of $\rv{ER}_{n,p}$ conditioned on the event $\{\rv{ER}_{n,p}\in\mathcal{CG}_n\}$, we conclude for $G\in\mathcal{CG}_n$, 
\begin{equation}\label{Pcluster}
\P(\rv{CG}_{n,p}=G)=\frac{w^{m(G)}}{B_n(w)}.
\end{equation}
\end{proof}

We now combine the last lemma
with the \cref{th:Gnm:cliques,th:Gnm:degree}
to describe the statistics
of $\rv{CG}_{n,p}$.

\begin{proposition}
\label{th:exact:pgf}
Let $\rv{M}_{n,p}$ and $\rv{C}_{n,p}$
denote the number of edges and number of cliques
of a random cluster graph $\PG_{n,p}$,
and let $\rv{D}_{n,p}$ denote the degree of a uniformly chosen vertex
in $\PG_{n,p}$. Recall \(B_n(w)\) from~\eqref{eq:partFunction} and set \(w=\nicefrac{p}{(1-p)}\).
The probability generating functions of these random variables are then
\begin{align}
    \PGF_{\rv{M}_{n,p}}(u) &=
    \frac{B_n(u\, w)}{B_n(w)}\label{eq:edges-pgf-exact},
    \\
    \PGF_{\rv{C}_{n,p}}(u) &=
    \frac{[z^n] e^{u C(w,z)}}{[z^n] e^{C(w,z)}}\label{eq:cliques-pgf-exact},
    \\
    \PGF_{\rv{D}_{n,p}}(u) &=
    \frac{[z^n] C_1(w, u\, z) e^{C(w, z)}}{u [z^n] C_1(w,z)e^{C(w,z)}}\label{eq:degree-pgf-exact}.
\end{align}
\end{proposition}

\begin{proof}
For the edges, we use \cref{lem:CG:pmf} to write
\begin{align*}
\PGF_{\rv{M}_{n,p}}(u)
&=\E[u^{\rv{M}_{n,p}}]
=\sum_{G\in\mathcal{CG}_n}\P(\rv{CG}_{n,p}=G)u^{m(G)}
=\sum_{G\in\mathcal{CG}_n}\frac{(uw)^{m(G)}}{B_n(w)}
=\frac{B_n(uw)}{B_n(w)}.
\end{align*}

Similarly, using \cref{th:Gnm:cliques}, we obtain
\[
    \PGF_{\rv{C}_{n,p}}(u) =\E[u^{\rv{C}_{n,p}}]=\frac{n![z^n]CG(w,z,u)}{B_n(w)}
    =
    \frac{[z^n] e^{u C(w,z)}}
    {[z^n] e^{C(w,z)}}.
\]
Further we deduce from \cref{th:Gnm:degree}
\[
n![z^n]\tfrac{1}{u}C_1(w,uz)e^{C(w,z)}=\sum_{G\in\mathcal{CG}_n}\sum_{i=1}^n w^{m(G)}u^{d_i(G)},
\]
where $d_i(G)$ denotes the degree of node $i$ in $G$. We use this to write
\begin{align*}
    \PGF_{\rv{C}_{n,p}}(u) 
    &=\E[u^{\rv{D}_{n,p}}]
    =\sum_{G\in\mathcal{CG}_n}\P(\rv{CG}_{n,p}=G)\sum_{i=1}^n\frac{1}{n}u^{d_i(G)}
    =\frac{n![z^n]\tfrac{1}{u}C_1(w,uz)e^{C(w,z)}}{nB_n(w)}.
\end{align*}
Finally, note that $C_1(w,z)e^{C(w,z)}=z\partial_ze^{C(w,z)}$, so that this generating function corresponds to the family of cluster graphs with a distinguished vertex. Therefore, $[z^n]C_1(w,z)e^{C(w,z)}=n[z^n]e^{C(w,z)}$. We conclude
\[
\frac{n![z^n]\tfrac{1}{u}C_1(w,uz)e^{C(w,z)}}{nB_n(w)}=\frac{n![z^n]\tfrac{1}{u}C_1(w,uz)e^{C(w,z)}}{n![z^n]C_1(w,z)e^{C(w,z)}}=\frac{[z^n]C_1(w,uz)e^{C(w,z)}}{u[z^n]C_1(w,z)e^{C(w,z)}},
\]
as claimed.
\end{proof}
\cref{th:exact:pgf} provides an exact description
of the laws of $\rv{C}_{n,p}$,
$\rv{M}_{n,p}$, and $\rv{D}_{n,p}$
for any fixed number $n$ of vertices
and value $p \in [0,1]$. We will use \cref{th:exact:pgf} to derive an exact expression for the degree distribution.

\begin{corollary}\label{cor:exact:degree-pmf}
    The degree $\rv{D}_{n,p}$ of a uniformly chosen vertex in $\rv{CG}_{n,p}$ has distribution
    \[
    \P(\rv{D}_{n,p}=d)={\binom{n-1}{d}}w^{{\binom{d+1}{2}}}\frac{B_{n-d-1}(w)}{B_n(w)}, \qquad \text{ for }d\in\{0,\dots,n-1\}.
    \]
\end{corollary}
\begin{proof}
    The probability mass function is obtained by extracting coefficients from the PGF. We have
    \begin{align*}
        \P(\rv{D}_{n,p}=d)
        &=[u^d]\PGF_{\rv{D}_{n,p}}(u) 
        =
    \frac{[u^dz^n] \tfrac{1}{u}C_1(w, u\, z) e^{C(w, z)}}{[z^n] C_1(w,z)e^{C(w,z)}}
        =
    \frac{[u^{d+1}z^n]C_1(w, u\, z) e^{C(w, z)}}{[z^n] C_1(w,z)e^{C(w,z)}}
\end{align*}
We now introduce the variable $x=uz$ to write
\begin{align*}
    [u^{d+1}z^n]C_1(w, u\, z) e^{C(w, z)}
    &=[x^{d+1}z^{n-d-1}]C_1(w, x) e^{C(w, z)}
    =\left([x^{d+1}]C_1(w, x)\right)\cdot \left([z^{n-d-1}]e^{C(w, z)}\right)\\
    &=\frac{(d+1)w^{{\binom{d+1}{2}}}}{(d+1)!}\cdot\frac{B_{n-d-1}(w)}{(n-d-1)!}.
    \end{align*}
    For the denominator, we write
    \[
    [z^n] C_1(w,z)e^{C(w,z)}=[z^n](z\partial_z)e^{C(w,z)}=\frac{nB_n(w)}{n!}=\frac{B_n(w)}{(n-1)!},
    \]
    yielding the desired expression.
\end{proof}
\cref{cor:exact:degree-pmf} allows us to find a recursion formula for $B_n(w)$, which generalises a well-known recursion formula of the Bell numbers (the case $w=1$).
\begin{corollary}\label{cor:Bn:recursion}
    The sequence $B_n(w)$ satisfies the following recursion formula:
    \[
    B_n(w)=\sum_{s=1}^{n}{\binom{n-1}{s-1}}w^{{\binom{s}{2}}}B_{n-s}(w).
    \]
\end{corollary}
\begin{proof}
    This follows directly from \cref{cor:exact:degree-pmf} and the law of total probability as
    \begin{align*}
        1=\sum_{s=1}^n\P(\rv{D}_{n,p}=s-1)=\sum_{s=1}^n{\binom{n-1}{s-1}}w^{{\binom{s}{2}}}\frac{B_{n-s}(w)}{B_n(w)}.
    \end{align*}
    Multiplying both sides by $B_n(w)$ gives the desired result.
\end{proof}

\subsection{From probability generating functions to limit laws}

The \emph{moment generating function} of the random variable $X$ can be expressed in terms of its probability generating function as
\[
    \mean(e^{t X}) =
    \PGF_X(e^t).
\]
Mean $\mu$ and variance $\sigma^2$ are recovered
from the probability generating function
or the moment generating function by the formulae
\begin{align*}
    \mu &=
    \PGF_X'(1)
    = \partial_{t = 0} \mean(e^{t X}),
    \\
    \sigma^2 &=
    \PGF_X''(1) + \PGF_X'(1) - \PGF_X'(1)^2
    =
    \partial_{t = 0}^2 \log \left(\mean(e^{t X}) \right).
\end{align*}
To prove Gaussian limit laws, we will apply
the following variant of L\'evy’s continuity theorem.

\begin{theorem}[\cite{billingsley1986probability}, Section 30]
\label{th:levy}
Consider a sequence of real-valued random variables $X_n$
and a real-valued random variable $Y$.
If the moment generating function $\mean(e^{t X_n})$
converges pointwise for $t$ in a neighborhood of $0$
to $\mean(e^{t Y})$,
then $X_n$ converges in law to $Y$.
\end{theorem}

In particular, the following classic corollary is designed
to prove a normal limit law.

\begin{corollary}
\label{th:normallimitlaw}
Consider a sequence of random variables $X_n$
with probability generating functions $\PGF_{X_n}(u)$.
    If there exists $\mu_n$ and $\sigma_n$ such that,
    pointwise for $s$ in a neighborhood of $0$,
    \[
        \PGF_{X_n}(e^{s / \sigma_n})
        \underset{n \to +\infty}{\sim}
        e^{s \mu_n / \sigma_n}
        e^{s^2/2},
    \]
    then the renormalized random variable
    $X_n^{\star} = \frac{X_n - \mu_n}{\sigma_n}$
    converges to the standard normal distribution.
\end{corollary}

\begin{proof}
The moment generating function of $X_n^{\star}$ is
\begin{align*}
    \mean(e^{s X_n^{\star}}) &=
    \mean(e^{s X_n / \sigma_n})
    e^{- s \mu_n / \sigma_n}
    =
    \PGF_{X_n}(e^{s / \sigma_n})
    e^{- s \mu_n / \sigma_n}
    \sim
    e^{s^2/2},
\end{align*}
so by L\'evy's continuity Theorem,
$X_n^{\star}$ converges to the standard normal distribution.
\end{proof}

The rest of the paper is devoted to the study
of statistics of $\PG(n,p)$ random graphs
in the limit when the number of vertices $n$ tends to infinity,
for various values of $p$,
or with $p$ varying with $n$.

\section{Asymptotics of the functions $C_r(w,z)$}\label{sec:c-asymp}
The asymptotics of the RCG are largely determined by the functions $C_r(w,z)$. In particular, the series $n^rw^{\binom{n}{2}}\tfrac{z^n}{n!}$ has zero radius of convergence precisely if $|w|>1$ explaining the phase transition at $w=1$ (i.e., $p=\nicefrac{1}{2}$).
In this section, we derive technical properties of $C_r(w,z)$ needed in \cref{sec:critical,sec:subcritical}. For the critical regime, we consider $w=e^{is}$, for $s\in\reals$. For the subcritical regime, we consider $w\in(0,1)$.
We can equivalently define \(C_r(w,z)\) as
\[
C_r(w,z)=(z\partial_z)^r C(w,z)=\sum_{n\geq 1} n^rw^{{\binom{n}{2}}}\frac{z^n}{n!},
\]
where $\partial_z$ denotes the partial derivative \wrt $z$.

\subsection{The critical case \(|w|=1\)}
In this section, we study $C(w,z)$ for $w$ on the unit circle, i.e., for $w=e^{is}$ for $s\in\mathbb{R}$. These results will be used to study the critical regime ($p=\nicefrac{1}{2}$) in \cref{sec:critical}.
When $s=0$, we have $w=1$ and the function reads
\[
C(1,z)=\sum_{n\geq 1} 1^{\binom{n}{2}}\frac{z^n}{n!}=e^z-1.
\]
Hence, for $r\in\mathbb{N}$, we also have $C_r(1,z)=(z\partial_z)^re^z$. \cref{sec:critical} will make use of the values of $z$ satisfying $C_1(1,z)=n$ or equivalently $ze^z=n$. Let $W(n)$ denote the unique positive solution of this equation. This $W(n)$ is known as the \emph{Lambert $W$-function}, and sometimes also referred to as the \emph{product logarithm}.
We define $c_r=C_r(1,W(n))$. 
For the asymptotics of $c_r$, it is sufficient to focus on asymptotics up to polynomial orders of $W(n)$. To that end, we write $f_n=\tilde{\bigO}(g_n)$ whenever $f_n/g_n$ is bounded by a polynomial in $W(n)$. For example, we may write $n=W(n)e^{W(n)}=\tilde{\bigO}(e^{W(n)})$.

\begin{lemma}\label{lem:Crs-leading}
    The functions $C_r(1,z)$ for $r\in\mathbb{N}$ are of the form $C_r(1,z)=P_r(z)e^{z}$, where $P_r(z)$ are $r$-th order polynomials with leading coefficient $1$. That is, $C_r(1,z)\sim z^re^{z}$ as $z\rightarrow\infty$. In particular, $c_r\sim W(n)^r\cdot e^{W(n)}=\tilde{\bigO}(e^{W(n)})$ for any $r$.
\end{lemma}
\begin{proof}
    Recall $C(1,z)=e^z-1$ and $C_r(1,z)=(z\partial_z)^r(e^z-1)$. We prove the claim by induction. For $r=1$, we have $(z\partial_z)(e^z-1)=ze^z$, proving the claim with $P_1(z)=z$. For the induction step, we observe 
    \[
    C_{r+1}(1,z)=(z\partial_z)C_r(1,z)=(z\partial_z)P_r(z)e^z=zP_r(z)e^z+z(\partial_zP_r(z))e^z,
    \]
    and set $P_{r+1}(z)=zP_r(z)+z\partial_z P_r(z)$. By the induction hypothesis, $z\partial_z P_r(z)$ is a polynomial of degree $r$, while $zP_r(z)$ is a polynomial of degree $r+1$ with leading coefficient $1$. This completes the proof. 
\end{proof}
\begin{lemma}\label{lem:Crs-asymp}
    Let $\zeta_s=W(n)\cdot \exp(a_1s+\tfrac{a_2}{2}s^2)$ with $a_1,a_2=\tilde{\bigO}(1)$. Then
    \begin{equation}\label{eq:Crs-asymp}
    	\begin{aligned}
        	C_r(e^{is},\zeta_s)&=c_r+s\cdot\left(a_1c_{r+1}+\frac{i}{2}c_{r+2}-\frac{i}{2}c_{r+1}\right) \\
			&\quad+s^2\cdot\left(\frac{a_2}{2}c_{r+1}+\frac{a_1^2}{2}c_{r+2}+a_1i\left(c_{r+3}-c_{r+2}\right)-\frac{c_{r+4}-2c_{r+3}+c_{r+2}}{8}\right)\\
			&\quad+\tilde{\bigO}\left(s^3e^{W(n)}\right).
    	\end{aligned}
    \end{equation}
    In particular, $C_r(e^{is},\zeta_s)\sim c_r$ for $s=\tilde{\bigO}(e^{-W(n)/2})$.
\end{lemma}
\begin{proof}
    Note that $(w\partial_w)C_r(w,z)=\frac{1}{2}\left(C_{r+2}(w,z)-C_{r+1}(w,z)\right)$. The derivative of $\zeta_s$ is given by $\partial_s\zeta_s=(a_1+a_2s)\zeta_s$, while for $w = e^{i s}$, we have $\partial_sw=iw$.
    Hence,
    \begin{equation}\label{expans_Cr}
    	\begin{aligned}
    		\frac{d}{ds}C_r(e^{is},\zeta_s)
    		&=(w\partial_w+(a_1+a_2s)z\partial_z)C_r(e^{is},\zeta_s) \\
    		&=\frac{i}{2}\left(C_{r+2}(e^{is},\zeta_s)-C_{r+1}(e^{is},\zeta_s)\right)+(a_1+a_2s)C_{r+1}(e^{is},\zeta_s).
    	\end{aligned}
    \end{equation}
    Abbreviating $C_r(e^{is},\zeta_s)$ by $C_r$, the second derivative is
    \begin{equation*}
    	\begin{aligned}
        	\frac{d^2}{ds^2}C_r(e^{is},\zeta_s)
        	&=\left(-(w\partial_w)^2+2i(a_1+a_2s)(w\partial_w)(z\partial_z)+a_2z\partial_z+(a_1+a_2s)^2(z\partial_z)^2\right)C_r\\
        	&=-\frac{C_{r+4}-2C_{r+3}+C_{r+2}}{4}+2i(a_1+a_2s)(C_{r+3}-C_{r+2})+a_2C_{r+1}+(a_1+a_2s)^2C_{r+2}.
    	\end{aligned}
    \end{equation*}
    Similarly, a crude bound on the third derivative follows from the previous lemma as $\frac{d^3}{d s^3} C_r(e^{i s}, \zeta_s) = \tilde{\bigO}(e^{W(n)})$.
    The desired result is then obtained by taking the Taylor expansion around $s=0$ using these derivatives.
\end{proof}
\begin{lemma}\label{lem:Crs-saddle}
    Let the \emph{approximate saddle point} $\zeta_s$ be given by
    \begin{equation}\label{eq:saddle-approx}
        \zeta_s=W(n)\exp\left(-\frac{i}{2}\frac{c_3-c_2}{c_2}s+\left(\frac{c_5-2c_4+c_3}{4c_2}+\frac{(c_3-c_2)^2c_3}{4c_2^3}-\frac{(c_3-c_2)(c_4-c_3)}{2c_2^2}\right)\frac{s^2}{2}\right).
    \end{equation}
    Then $C_1(e^{is},\zeta_s)-n=\tilde{\bigO}(s^3e^{W(n)})$.
\end{lemma}
\begin{proof}
    We apply \cref{lem:Crs-asymp} with $r=1$, $a_1=-\frac{i}{2}\frac{c_3-c_2}{c_2}$ and
    \[
    a_2=\frac{c_5-2c_4+c_3}{4c_2}+\frac{(c_3-c_2)^2c_3}{4c_2^3}-\frac{(c_3-c_2)(c_4-c_3)}{2c_2^2}.
    \]
    Note that $c_1=W(n)e^{W(n)}=n$. What remains to show is that the linear and quadratic terms in $s$ vanish. Indeed, we have
    \[
    a_1c_{2}+\frac{i}{2}c_3-\frac{i}{2}c_2=0,
    \]
    and
    \[
    a_2c_2+a_1^2c_3+2a_1i\left(c_4-c_3\right)-\frac{c_5-2c_4+c_3}{4}=0,
    \]
    as required.
\end{proof}

\subsection{The subcritical case $w<1$}

\begin{lemma}
\label{lem:C-asymptotics}
For $w\in(0,1)$, $x>0$, $\theta\in\reals$ and $r\in\mathbb{N}\cup\{0\}$, $C_r(w,xe^{i\theta})$ is given by
\[
    C_r(w,x e^{i \theta}) =
    w^{-\tau^2/2}
    \frac{e^{\tau} \tau^r}{\sqrt{2 \pi \tau}}
    e^{i \tau \theta}
    E_{w,r}(\tau, \theta),
\]
where
\[
    E_{w,r}(\tau,\theta) =
    \sum_{\tau + t \in \mathbb{N}}
    w^{t^2/2}
    e^{i t \theta}
    e^t
    \left( 1 + t/\tau \right)^{-\tau - t + r - 1/2}
    \frac{(\tau + t)^{\tau + t} e^{-\tau - t} \sqrt{2 \pi (\tau + t)}}
        {(\tau + t)!}
\]
is a bounded function, and $\tau$ is defined implicitly as the positive solution of
\[
    \tau \, (1/w)^{\tau - 1/2} = x,
\]
or explicitly,
\[
\tau=\frac{W\left( \frac{\log(1/w)}{\sqrt{w}}x \right)}{\log(1/w)}.
\]
\end{lemma}
\begin{proof}
    This proof is inspired by the Laplace method for sums
(see \cite[page 761]{FS09}).
We first separate in the summands
the polynomial contribution in $n$
from the exponential one.
To do so, we multiply and divide by the Stirling approximation, i.e.,
\[
    C_r(w,x e^{i \theta})
    =
    \sum_{n \geq 1}
    n^r w^{\binom{n}{2}}
    \frac{x^n e^{i n \theta}}{n!}
    =
    \sum_{n \geq 1}
    \frac{n^{n+r} e^{-n}}{n!} e^{i n \theta} e^{\phi(n)}
\]
where
\[
    \phi(n) =
    n \log(x / \sqrt{w})
    - n \log(n)
    + n
    - \frac{n^2}{2} \log(1/w).
\]
and $n^{n+r} e^{-n} / n!$
grows only polynomially in $n$.
The dominant contribution to the sum
comes from the values $n$ such that $\phi(n)$
is close to its maximum.
The derivative of $\phi$ is given by
\[
\phi'(n) =
    \log(x / \sqrt{w})
    - \log(n)
    - n \log(1/w),
\]
and solving $\phi'(\tau)=0$ yields
\begin{equation}\label{eq:E-fun-logx}
\log x=\log(\tau)+\tau\log(1/w)-\frac{1}{2}\log(1/w).
\end{equation}
We exponentiate both sides and obtain
\[
    \tau (1/w)^{\tau - 1/2} = x.
\]
Note that $\phi''(\tau)=-\tau^{-1}-\log(1/w)<0$ so that this $\tau$ indeed corresponds to a maximum.
After multiplying by $\log(1/w)w^{-\nicefrac{1}{2}}$, we obtain
\[
\frac{\log(1/w)}{\sqrt{w}}x=\log(1/w)\tau (1/w)^{\tau}=\log(1/w)\tau e^{\log(1/w)\tau},
\]
so that indeed
\[
\tau=\frac{W\left( \frac{\log(1/w)}{\sqrt{w}}x \right)}{\log(1/w)}.
\]

Instead of working with $\tau$
as an implicit function of $x$,
it is more convenient to think of $x$
as an explicit function of $\tau$.
After substituting $\log x$ from~\eqref{eq:E-fun-logx} and the variable change $t = n - \tau$,
we have
\[
    C_r(w,\tau (1/w)^{\tau-1/2} e^{i \theta})
    =
    w^{-\tau^2/2}
    \sum_{\tau + t \in \mathbb{N}}
    (\tau + t)^r
    e^{i (\tau + t) \theta}
    w^{t^2/2}
    \frac{\tau^{\tau + t}}{(\tau + t)!}
    =
    w^{-\tau^2/2}
    \frac{e^{\tau} \tau^r}{\sqrt{2 \pi \tau}}
    e^{i \tau \theta}
    E_{w,r}(\tau,\theta),
\]
where
	\[
		E_{w,r}(\tau,\theta)=e^{-\tau}\sum_{\tau + t \in \mathbb{N}}
		(1 + \tfrac{t}{\tau})^r
		e^{i t \theta}
		w^{t^2/2}
		\frac{\tau^{\tau + t}\sqrt{2 \pi \tau}}{(\tau + t)!}.
	\]
	By multiplication and division by $(1 + \tfrac{t}{\tau})^{\tau+t+\nicefrac{1}{2}}e^{-\tau-t}$, we rewrite $E_{w,r}(\tau,\theta)$ to
	\[
	E_{w,r}(\tau,\theta) =
	\sum_{\tau + t \in \mathbb{N}}
	w^{t^2/2}
	e^{i t \theta}
	e^t
	\left( 1 + \tfrac{t}{\tau} \right)^{-\tau - t + r - \nicefrac{1}{2}}
	\frac{(\tau + t)^{\tau + t} e^{-\tau - t} \sqrt{2 \pi (\tau + t)}}
	{(\tau + t)!},
	\]
	so that the fraction on the right will be close to $1$ due to the Stirling approximation.
Let us now prove that $|E_{w,r}(\tau,\theta)|$
is a bounded function of $\tau \geq 1$
for any fixed $w \in (0,1)$ and nonnegative integer $r$.
We bound
\[
    |E_{w,r}(\tau,\theta)|
    \leq
    \sum_{\tau + t \in \mathbb{N}}
    w^{t^2/2}
    e^t
    \left( 1 + t/\tau \right)^{-\tau - t + r - 1/2}
    \frac{(\tau + t)^{\tau + t} e^{-\tau - t} \sqrt{2 \pi (\tau + t)}}
        {(\tau + t)!}.
\]
The sum is decomposed into three parts,
corresponding respectively to
$t \geq 0$,
$t \in (-\tau^{1-\epsilon}, 0)$
and $t \in [-\tau, -\tau^{1-\epsilon}]$.
For $t \geq 0$, we have
$\left( 1 + t/\tau \right)^{-\tau - t - 1/2} \leq 1$
so, applying Stirling bound,
\[
    \sum_{\substack{\tau + t \in \mathbb{N}\\ t \geq 0}}
    w^{t^2/2}
    e^t
    \left( 1 + t/\tau \right)^{-\tau - t + r - 1/2}
    \frac{(\tau + t)^{\tau + t} e^{-\tau - t} \sqrt{2 \pi (\tau + t)}}
        {(\tau + t)!}
    \leq
    \sum_{\substack{\tau + t \in \mathbb{N}\\ t \geq 0}}
    w^{t^2/2}
    e^t
    \left( 1 + t \right)^r,
\]
which is a convergent sum since $w<1$.
Consider a positive small $\epsilon$, to be fixed later.
For $t \in (-\tau^{1-\epsilon}, 0)$, we have
$1 + t/\tau \geq 1 - \tau^{-\epsilon}$,
so
\[
    (1 + t/\tau)^{-\tau - t}
    \leq
    (1 + t/\tau)^{-\tau}
    \leq
    e^{-\tau \log(1 + t/\tau)}.
\]
We expand the logarithm up to an order $K$
large enough to ensure $K \epsilon > 1$ and obtain
\begin{align*}
    (1 + t/\tau)^{-\tau - t}
    &\leq
    \exp \bigg(
        \tau
        \sum_{k=1}^{K-1} \frac{1}{k} \frac{(-t)^{k}}{\tau^{k}}
        + \bigO(\tau^{1- K \epsilon})
    \bigg)
    \leq
    \exp \bigg(
        |t|
        \sum_{k=1}^{K-1} \frac{1}{k} \frac{|t|^{k-1}}{\tau^{k-1}}
        + \bigO(\tau^{1 - K \epsilon})
    \bigg)
    \\&\leq
    \exp \bigg(
        |t|
        \sum_{k=1}^{K-1} \frac{\tau^{-k \epsilon}}{k}
        + \bigO(\tau^{1 - K \epsilon})
    \bigg)
    \leq
    e^{\bigO(t)}.
\end{align*}
Therefore, by applying Stirling bound again,
\[
    \sum_{\substack{\tau + t \in \mathbb{N}\\ t \in (-\tau^{1-\epsilon}, 0)}}
    w^{t^2/2}
    e^t
    \left( 1 + t/\tau \right)^{-\tau - t + r - 1/2}
    \frac{(\tau + t)^{\tau + t} e^{-\tau - t} \sqrt{2 \pi (\tau + t)}}
        {(\tau + t)!}
    \leq
    \sum_{\substack{\tau + t \in \mathbb{N}\\ t \in (-\tau^{1-\epsilon}, 0)}}
    w^{t^2/2}
    e^{\bigO(t)},
\]
which is a bounded sum.
Finally, for $t \in [-\tau, -\tau^{1-\epsilon}]$,
we have
\begin{align*}
    &
    \sum_{\substack{\tau + t \in \mathbb{N}\\ t \in [-\tau, -\tau^{1-\epsilon}]}}
    w^{t^2/2}
    e^t
    \left( 1 + t/\tau \right)^{-\tau - t + r - 1/2}
    \frac{(\tau + t)^{\tau + t} e^{-\tau - t} \sqrt{2 \pi (\tau + t)}}
        {(\tau + t)!}
    \\&=
    \sum_{\substack{\tau + t \in \mathbb{N}\\ t \in [-\tau, -\tau^{1-\epsilon}]}}
    w^{t^2/2}
    e^t
    \tau^{\tau + t - r + 1/2}
    (\tau + t)^r
    \frac{e^{-\tau-t} \sqrt{2 \pi}}{(\tau + t)!}
    \leq
    \exp \left(
        - \log(1/w) \frac{\tau^{2 - 2 \epsilon}}{2}
        + \bigO(\tau \log(\tau))
    \right).
\end{align*}
Choosing $\epsilon < 1/2$ ensures that this bound tends to $0$.
This concludes the proof that $|E_{w,r}(\tau,\theta)|$ is bounded.
\end{proof}
Most occurrences of the functions $E_{w,r}$ in this paper are with $\theta=0$. In these cases, we will simply write $E_{w,r}(\tau):=E_{w,r}(\tau,\theta)$ to keep notation concise.

\begin{lemma}\label{lem:E-asymptotics}
The bounded functions $E_{w,r}(\tau,\theta)$ allow for the following asymptotic expansion in terms of $\tau^{-1}$, valid for any $K \geq 0$:
$$
E_{w,r}(\tau,\theta)=\sum_{k=0}^{K-1} \frac{1}{\tau^k}\sum_{\ell=0}^{2k}a_{k,\ell}(r)\cdot e_{w,\ell}(\tau,\theta)+O(\tau^{-K}),
$$
where the $e_{w,\ell}(\tau,\theta)$ are periodic functions defined by
$$
e_{w,\ell}(\tau,\theta)=\sum_{t+\tau\in\mathbb{Z}}t^\ell w^{t^2/2}e^{it\theta},
$$
which have Fourier series given by
\begin{equation}\label{eq:e-fourier}
e_{w,\ell}(\tau,\theta)=\sqrt{\frac{2\pi}{\log(1/w)^{\ell+1}}}\cdot\left[\sum_{s\in\mathbb{Z}}e^{-\frac{(2\pi s+\theta)^2}{2\log(1/w)}}\sum_{j=0}^{\lfloor \ell/2\rfloor}\frac{(2j)!}{2^jj!}\left(i\frac{2\pi s+\theta}{\sqrt{\log(1/w)}}\right)^{\ell-2j}e^{2\pi s i\tau}\right],
\end{equation}
and $a_{k,\ell}(r)$ are computable coefficients. The first few coefficients are shown in \cref{tab:coefficients}. In particular, the leading asymptotics of $E_{w,r}$ and its derivatives do not depend on $r$. For $\theta=0$ and $r\in\mathbb{N}$, 
\begin{equation}
	E_{w,r}(\tau,0)\sim
    e_{w,0}(\tau,0)=\sqrt{\frac{2\pi}{\log(1/w)}}\cdot\left[1+2\sum_{s=1}^\infty e^{-\frac{2\pi^2}{\log(1/w)}s^2}\cos(2\pi s\tau)\right],\label{eq:Ewr-asymp}
\end{equation}
\begin{equation}
	\frac{d}{d\tau}E_{w,r}(\tau,0)\sim\log(1/w)e_{w,1}(\tau,0)=4\pi\sqrt{\frac{2\pi}{\log(1/w)}}\sum_{s=1}^\infty s\cdot e^{-\frac{2\pi^2}{\log(1/w)}s^2}\sin(2\pi s\tau),\label{eq:Ewr-d1-asymp}
\end{equation}
\begin{equation}\label{eq:Ewr-d2-asymp}.
	\begin{aligned}
    	\frac{d^2}{d\tau^2}E_{w,r}(\tau,0) 
    		& \sim\log(1/w)^2e_{w,2}(\tau,0)-\log(1/w)e_{w,0}(\tau,0) \\
    		&=-8\pi^2\sqrt{\frac{2\pi}{\log(1/w)}}\sum_{s=1}^\infty  s^2\cdot e^{-\frac{2\pi^2}{\log(1/w)}s^2}\cos(2\pi s\tau).
	\end{aligned}
\end{equation}
\end{lemma}

%%%% Table of coefficients
\begin{table}[t!]
    \centering
    \begin{tabular}{c|ccccc}
         $a_{k,\ell}(r)$ & $\ell=0$        & $\ell=1$       & $\ell=2$ & $\ell=3$ & $\ell=4$ \\
         \hline
         $k=0$        & $1$             &&&&\\
         $k=1$        & $-\frac{1}{12}$ & $r-\frac{1}{2}$ & $-\frac{1}{2}$&&\\
         $k=2$        & $\frac{1}{288}$ & $\frac{r}{12}-\frac{1}{8}$ & $\tfrac{r^2}{2}-r+\tfrac{5}{12}$ & $\frac{5}{12}-\frac{r}{2}$ & $\frac{1}{8}$\\
    \end{tabular}
    \caption{The first $a_{k,\ell}$ coefficients of \cref{lem:E-asymptotics}.}
    \label{tab:coefficients}
\end{table}
%%%%%%
\begin{proof}
\proofparagraph{Removing the tails.}
Define
\[
    a_{w,r,\theta,\tau}(t) =
    w^{t^2/2}
    e^{i t \theta}
    e^t
    \left( 1 + t/\tau \right)^{-\tau - t + r - 1/2}
    \frac{(\tau + t)^{\tau + t} e^{-\tau - t} \sqrt{2 \pi (\tau + t)}}
        {(\tau + t)!},
\]
so $E_{w,r}(\tau,\theta) = \sum_{\tau + t \in \mathbb{N}} a_{w,r,\theta,\tau}(t)$.
Let us first prove that for any $\epsilon \in (0,1/2)$,
there exists $\delta > 0$ such that
\[
    E_{w,r}(\tau,\theta) =
    \sum_{\substack{\tau + t \in \mathbb{N}\\ |t| \leq \tau^{\epsilon}}}
    a_{w,r,\theta,\tau}(t)
    + \bigO(e^{- \tau^{\delta}}).
\]
We have already seen at the end of the proof of \cref{lem:C-asymptotics}
that for any $\epsilon \in (0,1/2)$, we have
\[
    \sum_{\substack{\tau + t \in \mathbb{N}\\ t \in [-\tau, -\tau^{1-\epsilon}]}}
    |a_{w,r,\theta,\tau}(t)|
\]
is exponentially small.
We now consider $t \geq \tau^{\epsilon}$,
applying the Stirling bound and $(1+t/\tau)^{-\tau -t + r - 1/2} < 1$
yields
\[
    |a_{w,r,\theta,\tau}(t)|
    \leq
    e^{-\log(1/w) t^2/2 + t}.
\]
The quadratic term in the exponent dominates the linear one.
By comparison with the complementary error function
\[
    \operatorname{erfc}(x) =
    \frac{1}{\sqrt{\pi}}
    \int_x^{+\infty}
    e^{-y^2} dy,
\]
whose asymptotics is $\bigO(e^{-x^2})$,
we deduce that the sum
\[
    \sum_{\substack{\tau + t \in \mathbb{N}\\ t \geq \tau^{\epsilon}}}
    e^{-\log(1/w) t^2/2 + t}
\]
converges to $0$ as $\bigO(e^{-\tau^{\delta}})$
for some $\delta > 0$.
The last case to consider is $t \in [-\tau^{1-\epsilon}, -\tau^{\epsilon}]$.
We rewrite
\[
    \log(|a_{w,r,\theta,\tau}(t)|)
    =
    - \log(1/w) \frac{t^2}{2}
    + t
    + (\tau + t - r + 1/2)
    \log \left( \frac{1}{1 + t/\tau} \right)
    \leq
    - \log(1/w) \frac{t^2}{2}
    + \tau
    \log \left( \frac{1}{1 + t / \tau} \right).
\]
Let us denote this expression as a function $f(t)$.
Then $f'(t) = - \log(1/w) \frac{t^2}{2} - \frac{1}{1 + t/\tau}$
and $f''(t) = - \log(1/w) + \frac{1}{\tau} \frac{1}{(1+t/\tau)^2}$.
With $t \in [-\tau^{1-\epsilon}, -\tau^{\epsilon}]$,
we have $t/\tau$ tending to $0$, so $f''(t)$ converges, as $\tau$ tends to infinity,
to the negative $-\log(1/w)$.
It is thus negative for any large enough $\tau$
and the function $f'(t)$ is decreasing on
$t \in [-\tau^{1-\epsilon}, -\tau^{\epsilon}]$.
Its zero is located at $- 1 / \log(1/w) + \bigO(\tau^{-1})$,
which is eventually greater than $-\tau^{\epsilon}$.
This implies that $f'(t)$ stays positive on the interval of interest,
so $f(t)$ is increasing.
Its maximum must be reached at $t = -\tau^{\epsilon}$ and
\[
    \log(|a_{w,r,\theta,\tau}(t)|)
    \leq
    - \log(1/w) \frac{\tau^{2\epsilon}}{2}
    + \tau \log \left(
    \frac{1}{1 - \tau^{\epsilon - 1}}
    \right)
    =
    - \log(1/w) \frac{\tau^{2\epsilon}}{2}
    + \tau^{\epsilon}
    + \smallo(1).
\]
Again, by comparison with the complementary error function,
we deduce that the sum
\[
    \sum_{\substack{\tau + t \in \mathbb{N}\\ t \in [-\tau^{1-\epsilon}, -\tau^{\epsilon}]}}
    |a_{w,r,\theta,\tau}(t)|
\]
is $\bigO(e^{-\tau^{\delta}})$ for some $\delta > 0$.
We conclude that the tails are negligible, and there exists $\delta > 0$
such that
\[
    E_{w,r}(\tau,\theta) =
    \sum_{\substack{\tau + t \in \mathbb{N}\\ |t| \leq \tau^{\epsilon}}}
    a_{w,r,\theta,\tau}(t)
    + \bigO(e^{- \tau^{\delta}}).
\]

\proofparagraph{Asymptotic expansion in the central part.}
Consider the ring $\rationals[t][[\tau^{-1}]]$
of formal powers in $\tau^{-1}$ with coefficients
that are polynomials in $t$ over the rationals.
If $F(x)$ denotes a formal power series with rational coefficients
and $G(t,\tau^{-1}) \in \rationals[t][[\tau^{-1}]]$,
then the composition $F(G(t,\tau^{-1}))$
also belongs to $\rationals[t][[\tau^{-1}]]$.
Since
\begin{align*}
    e^t (1 + t/\tau)^{-\tau - t + r - 1/2}
    &=
    \exp \left( t - (\tau + t - r + 1/2) \log(1 + t/\tau) \right)
    \\&=
    \exp \bigg(
        - (t - r + 1/2) t \tau^{-1}
        - (1 + (t - r + 1/2) \tau^{-1})
        \sum_{k \geq 2}
        \frac{1}{k} t^k \tau^{-k+1}
    \bigg),
\end{align*}
this series belongs to $\rationals[t][[\tau^{-1}]]$.
For any $K$, its Taylor expansion of order $K$
has error term $\bigO(t^K \tau^{-K + 1})$.
The Stirling approximation has an associated asymptotic expansion
(see \eg \cite{namias1986simple}),
of the form
\[
    \frac{n!}{n^n e^{-n} \sqrt{2 \pi n}}
    =
    S(n^{-1})
\]
for some formal power series $S(x)$
whose first few coefficients are
\[
    S(x) =
    1 + \frac{1}{12} x + \frac{1}{288} x^2 - \frac{139}{51840} x^3
    + \bigO(x^4).
\]
We deduce that
\[
    \frac{(\tau + t)!}
    {(\tau + t)^{\tau + t} e^{-\tau - t} \sqrt{2 \pi (\tau + t)}}
    =
    S \left( \frac{1}{\tau + t} \right)
    =
    S \bigg(
    \sum_{k \geq 0} \tau^{-k-1} t^k
    \bigg)
\]
belongs to $\rationals[t][[\tau^{-1}]]$.
Its constant term is $1$,
so it has a formal multiplicative inverse,
belonging to the same ring.
Again, the error term of the Taylor expansion
is $\bigO(t^K \tau^{-K + 1})$.
This implies the existence of polynomials
$P_{r,k}(t)=\sum_{\ell=0}^{2k}a_{k,\ell}(r)t^{\ell}$
such that
\begin{equation}
\label{eq:expansion:tau:t}
    e^t (1 + t/\tau)^{-\tau - t + r - 1/2}
    \frac{(\tau + t)^{\tau + t} e^{-\tau - t} \sqrt{2 \pi (\tau + t)}}
        {(\tau + t)!}
    =
    \sum_{k=0}^{K-1}
    P_{r,k}(t) \tau^{-k}
    + \bigO(t^K \tau^{-K+1}).
\end{equation}
Multiplying by $w^{t^2/2} e^{i t \theta}$ and summing over $t$, we obtain,
for some $\delta > 0$,
\[
    E_{w,r}(\tau, \theta) =
    \sum_{k = 0}^{K-1}
    \tau^{-k}
    \sum_{\substack{\tau + t \in \mathbb{N}\\ |t| \leq \tau^{\epsilon}}}
    w^{t^2/2}
    e^{i t \theta}
    P_{r,k}(t)
    + \bigO \bigg(
        \sum_{\tau + t \in \integers}
        w^{t^2/2} |t|^{K+1} \tau^{-K}
    \bigg)
    + \bigO(e^{-\tau^{\delta}}).
\]
The sum in the first error term is bounded,
so this error term is simply $\bigO(\tau^{-K})$
and the second error term is negligible in comparison.

\proofparagraph{Adding the tails.}
Finally, we observe as before that the tails are negligible,
meaning that for each $k \in [0,K-1]$,
there exists some $\delta > 0$ such that
\[
    \sum_{\substack{\tau + t \in \mathbb{N}\\ |t| \geq \tau^{\epsilon}}}
    w^{t^2/2}
    P_{r,k}(t)
    =
    \bigO(e^{-\tau^{\delta}}),
\]
so adding them introduces only a negligible error term,
contained in $\bigO(\tau^{-K})$.

\proofparagraph{First polynomials.}
The first three polynomials $(P_{r,k}(t))_{0 \leq k \leq 2}$
are computed using the computer algebra system Sage~\cite{sage}, and the corresponding coefficients $a_{k,\ell}(r)$ are collected in \cref{tab:coefficients}.
We rewrite
\begin{align*}
    E_{w,r}(\tau,\theta)
    &=
    \sum_{k=0}^{K-1}
    \tau^{-k}
    \sum_{\tau + t \in \integers}
    w^{t^2/2}
    e^{i t \theta}
    P_{r,k}(t)
    + \bigO(\tau^{-K})
    =\sum_{k=0}^{K-1}
    \tau^{-k}
    \sum_{\tau + t \in \integers}
    w^{t^2/2}
    e^{i t \theta}
    \sum_{\ell=0}^{2k}a_{k,\ell}(r)t^\ell
    + \bigO(\tau^{-K})\\
    &=\sum_{k=0}^{K-1}
    \tau^{-k}
    \sum_{\ell=0}^{2k}
    a_{k,\ell}(r)
    \sum_{\tau + t \in \integers}
    w^{t^2/2}
    e^{i t \theta}t^\ell
    + \bigO(\tau^{-K})
    =\sum_{k=0}^{K-1}
    \tau^{-k}
    \sum_{\ell=0}^{2k}
    a_{k,\ell}(r)
    e_{w,\ell}(\tau,\theta)
    + \bigO(\tau^{-K}).
\end{align*}
We now derive the Fourier series of $e_{w,\ell}(\tau,\theta)$. Firstly, note that the functions $e_{w,\ell}(\tau,\theta)$ are indeed $1$-periodic in $\tau$, because $t+\tau\in\mathbb{Z}$ implies $t+(\tau+1)\in\mathbb{Z}$. We will write
\[
e_{w,\ell}(\tau,\theta)=\sum_{s\in\mathbb{Z}}f_{\ell,s}(\theta)e^{2\pi si\tau},
\]
We will compute Fourier coefficients $f_{\ell,s}(\theta)$ as
\begin{align*}
    f_{\ell,s}(\theta)
    &=\int_0^1 e_{w,\ell}(x,\theta)e^{-2\pi i sx}dx
    = \int_0^1\sum_{k\in\mathbb{Z}}(k-x)^\ell w^{(k-x)^2/2}e^{i(k-x)\theta}e^{-2\pi i sx}dx.
\end{align*}
Note that $\bigcup_{k\in\mathbb{Z}}[k-1,k]=\mathbb{R}$. We can thus combine this sum of integrals into one single integral:
\begin{align*}
    f_{\ell,s}(\theta)
    &=\int_{-\infty}^\infty x^\ell w^{x^2/2}e^{ix\theta}e^{-2\pi i sx}dx
    =\int_{-\infty}^\infty x^\ell\exp\left(-\frac{\log(1/w)}{2}x^2+ix\theta+2\pi six\right)dx\\
    &=e^{-\frac{(2\pi s+\theta)^2}{2\log(1/w)}}\int_{-\infty}^\infty x^\ell\exp\left(-\frac{1}{2}\left(x\sqrt{\log(1/w)}-i\frac{2\pi s+\theta}{\sqrt{\log(1/w)}}\right)^2\right)dx.
\end{align*}
We now perform the substitution $y=x\sqrt{\log(1/w)}-i\frac{2\pi s+\theta}{\sqrt{\log(1/w)}}$ with
$$
x=\frac{y+i\frac{2\pi s+\theta}{\sqrt{\log(1/w)}}}{\sqrt{\log(1/w)}},\quad\text{and}\quad \frac{dx}{dy}=\frac{1}{\sqrt{\log(1/w)}}.
$$
We obtain
\begin{align*}
    f_{\ell,s}(\theta)
    &=\frac{e^{-\frac{(2\pi s+\theta)^2}{2\log(1/w)}}}{\log(1/w)^{\frac{\ell+1}{2}}}\int_{-\infty}^\infty \left(y+i\frac{2\pi s+\theta}{\sqrt{\log(1/w)}}\right)^\ell e^{-\frac{1}{2}y^2}dy\\
    &=\frac{e^{-\frac{(2\pi s+\theta)^2}{2\log(1/w)}}}{\log(1/w)^{\frac{\ell+1}{2}}}\sum_{k=0}^{\ell}\left(i\frac{2\pi s+\theta}{\sqrt{\log(1/w)}}\right)^{\ell-k}\int_{-\infty}^\infty y^k e^{-\frac{1}{2}y^2}dy.
\end{align*}
We use
\[
\int_{-\infty}^\infty y^{2k}e^{-y^2/2}dy=\frac{(2k)!}{2^kk!}\sqrt{2\pi},\quad\text{and}\quad\int_{-\infty}^\infty y^{2k+1}e^{-y^2/2}dy=0,
\]
to find
\[
f_{\ell,s}(\theta)=\sqrt{\frac{2\pi}{\log(1/w)}}\frac{e^{-\frac{(2\pi s+\theta)^2}{2\log(1/w)}}}{\log(1/w)^{\frac{\ell}{2}}}\sum_{j=0}^{\lfloor \ell/2\rfloor}\frac{(2j)!}{2^jj!}\left(i\frac{2\pi s+\theta}{\sqrt{\log(1/w)}}\right)^{\ell-2j},
\]
as claimed. Note that $f_{\ell,s}(\theta)$ is real for even $\ell$ and imaginary for odd $\ell$. In addition, $f_{\ell,-s}(\theta)=f_{\ell,s}(\theta)$ for even $\ell$ and $f_{\ell,-s}(\theta)=-f_{\ell,s}(\theta)$ for odd $\ell$. Therefore, we can write
\[
e_{w,2\ell}(\tau,0)=f_{2\ell,0}(0)+2\sum_{s\in\mathbb{N}}f_{2\ell,s}(0)\cos(2\pi s\tau),
\]
and
\[
e_{w,2\ell+1}(\tau,0)=f_{2\ell+1,0}(0)+2\sum_{s\in\mathbb{N}}f_{2\ell+1,s}(0)\sin(2\pi s\tau).
\]
which yield the asymptotics from~\eqref{eq:Ewr-asymp}-\eqref{eq:Ewr-d2-asymp}. Finally, the derivatives follow from $\tfrac{d^k}{d\tau^k}E_{w,r}(\tau,0)\sim \tfrac{d^k}{d\tau^k}e_{w,0}(\tau,0)$, $\tfrac{d}{d\tau}e_{w,0}(\tau,0)=\log(1/w)e_{w,1}(\tau,0)$ and 
\[
\tfrac{d}{d\tau}e_{w,1}(\tau,0)=\log(1/w)e_{w,2}(\tau,0)-e_{w,0}(\tau,0).
\]

\end{proof}

The coefficients of the Fourier series in \eqref{eq:Ewr-asymp}-\eqref{eq:Ewr-d2-asymp} decrease very rapidly due to the quadratic exponents of $\exp\left(-\frac{2\pi^2}{\log(1/w)}\right)=\varepsilon$. This $\varepsilon$ tends to be quite small: for $w=\tfrac{1}{2}$, we already have $\varepsilon\approx 4.3\cdot 10^{-13}$.
Therefore, these Fourier series can be well approximated by omitting higher powers of $\varepsilon$. In particular:

\begin{align*}
e_{w,0}(\tau,0)&\approx\sqrt{\frac{2\pi}{\log(1/w)}}\cdot\left(1+2\varepsilon\cos(2\pi\tau)\right),\\
e_{w,1}(\tau,0)&\approx2\left(\frac{2\pi}{\log(1/w)}\right)^{3/2}\varepsilon\sin(2\pi \tau),\\
e_{w,2}(\tau,0)&\approx\frac{1}{\log(1/w)}\sqrt{\frac{2\pi}{\log(1/w)}}\left[1+\left(2-\frac{8\pi^2}{\log(1/w)}\right)\varepsilon\cos(2\pi\tau)\right].
\end{align*}

Finally, the following lemma gives the asymptotics of $\tau$ from \cref{lem:C-asymptotics}:

\begin{lemma}\label{lem:subcritical-tau-asymp}
    Let $\tau$ be implicitly defined by $C_1 \left(w, \tau (\nicefrac{1}{w})^{\tau - 1/2}\right)=n e^{-s}$ for fixed $s$ and $w$. Then $\tau$ has the asymptotics
    \begin{equation}
        \tau=\sqrt{\frac{2\log n}{\log(1/w)}}-\frac{1}{\log(1/w)}+\bigO\left(\frac{\log \log n}{\sqrt{\log n}}\right).
    \end{equation}
\end{lemma}
\begin{proof}
    We take the logarithm and rewrite it using \cref{lem:C-asymptotics} as
    \begin{equation}\label{eq:subcritical-tau-asymp-1}
    \log n -s=\log C_1 \left(w, \tau (\nicefrac{1}{w})^{\tau - 1/2}\right)=\frac{\log(1/w)}{2}\tau^2+\tau+\frac{1}{2}\log\tau+\log\frac{E_{w,1}(\tau,0)}{\sqrt{2\pi}}.
    \end{equation}
    Since the left-hand side goes to infinity, it must hold that $\tau\rightarrow\infty$. Therefore, the right-hand side is dominated by the quadratic term.
    Hence, $\log n\sim \tfrac{1}{2}\log(1/w)\tau^2$, so that
    \[
    \tau\sim\sqrt{\frac{2\log n}{\log(1/w)}}.
    \]
    This tells us that $\log\tau=\bigO(\log\log n)$, so that we can rewrite~\eqref{eq:subcritical-tau-asymp-1} to
    \[
    0=\frac{\log(1/w)}{2}\tau^2+\tau-\log n+\bigO(\log\log n).
    \]
    The positive solution of this quadratic formula is
    \begin{align*}
    \tau
    &=\frac{\sqrt{1+2\log(1/w)(\log n+\bigO(\log\log n))}-1}{\log(1/w)}
    =\frac{\sqrt{2\log(1/w)\log n\cdot \left(1+\bigO\left(\frac{\log\log n}{\log n}\right)\right)}-1}{\log(1/w)}\\
    &=\sqrt{\frac{2\log n}{\log(1/w)}}\cdot\left(1+\bigO\left(\frac{\log\log n}{\log n}\right)\right)-\frac{1}{\log(1/w)}
    =\sqrt{\frac{2\log n}{\log(1/w)}}-\frac{1}{\log(1/w)}+\bigO\left(\frac{\log\log n}{\sqrt{\log n}}\right),
    \end{align*}
    as claimed.
\end{proof}

\section{The critical regime $p=\nicefrac{1}{2}$} \label{sec:critical}
We fix \(p=\nicefrac{1}{2}\) during the whole section and set $w=\nicefrac{p}{(1-p)}=1$. Hence, the distribution of $\rv{CG}_{n,\nicefrac{1}{2}}$ is given by
\[
\P(\rv{CG}_{n,\nicefrac{1}{2}}=G)=\frac{1^{m(G)}}{B_n(1)}=\frac{1}{B_n},
\]
for any cluster graph $G\in\mathcal{CG}_n$, where $B_n=|\mathcal{CG}_n|$ is the $n$-th \emph{Bell number}. The distribution of $\rv{CG}_{n,\nicefrac{1}{2}}$ therefore coincides with the uniform distribution over set partition of $n$ elements.

The uniform distribution over set partitions has been studied since the 1960's. Harper was the first who showed that the number of blocks in a random uniform set partition of $n$ points has a normal distribution, as $n\to\infty$, with expected value of order $n/\log(n)$ \cite{harper1967stirling}. The exact value of this expectation can be expressed~\cite{sachkov1997probabilistic} in terms of the Bell numbers as
\[
\E[\rv{C}_{n,\nicefrac{1}{2}}]=\frac{B_{n+1}}{B_n}-1.
\]
The asymptotic growth of the Bell numbers is given by
    \begin{equation}\label{Bn}
    \frac{B_n}{n!}=\frac{e^{e^{r_n}-1}}{(r_n)^n\sqrt{2\pi r_n(r_n +1)e^{r_n}}}\left(1+\bigO(e^{-r_n/5})\right),
    \end{equation}
    where $r_n$ is defined implicitly as the positive solution of $r_ne^{r_n}=n+1$, or in terms of the \emph{Lambert $W$-function}, $r_n=W(n+1)$~\cite[Proposition~VIII.3]{FS09}.
    
While there are numerous articles studying the uniform distribution over partitions, there are (to the best of our knowledge) none where this distribution is studied as a random graph distribution. This change of perspective allows us to investigate the asymptotic behaviour of graph observables such as the \emph{degree distribution} or the \emph{number of edges} by means of analytic combinatorics techniques.   
The core of this section is the proof of \Cref{thm:mainEdges}~(ii) and~\Cref{thm:mainDegree}~(ii). We prove that for large $n$, the degree $\rv{D}_{n,1/2}$ of a random vertex in $\rv{CG}_{n,1/2}$ is nearly indistinguishable from a Poisson distribution with parameter $W(n)$, where $W$ denotes the Lambert $W$-function. For the random total number of edges $\rv{M}_{n,1/2}$, we provide exact expressions for the mean and variance, as well as a Gaussian limit law. First, we introduce some preliminary results that are preparatory to the main proofs. 

\subsection{Asymptotics of Bell numbers} \label{sec:critical:Bell}
First we derive the asymptotics of $B_{n-1}/B_n$ and relate this to the expected number of edges in the cluster graph. Denote by \(\rv{C}_n^{(j)}\) the number of clusters of size \(j\). Particularly, \(\rv{C}_n^{(1)}\) then denotes the number of isolated vertices. Recall that the Lambert \(W\)-function is the unique positive solution of the equation \(x=W(x)e^{W(x)}\) for \(x\in[0,\infty)\). Put differently, \(W\) is the inverse of the function \(x\mapsto xe^x\). 

\begin{lemma}\label{lem:bell-frac-asymptotics}
   Let $W$ be the Lambert $W$-function. The edge-density of a uniform cluster graph is given by
    \begin{equation}\label{eq:bell-frac-density}
        \frac{\E\left[\rv{M}_{n,1/2}\right]}{{n\choose 2}}=\frac{B_{n-1}}{B_n}=\frac{W(n)}{n}+\bigO\left(\left(\frac{\log n}{n}\right)^{6/5}\right).
    \end{equation}
    \begin{equation}\label{eq:uniform-kappa}  \E[\rv{C}^{(1)}_n]=n\frac{B_{n-1}}{B_n}=W(n)+\bigO\left(\frac{(\log n)^{6/5}}{n^{1/5}}\right).
    \end{equation}
\end{lemma}
\begin{proof}
    We start by proving $\E[\rv{C}^{(1)}_n]=n\frac{B_{n-1}}{B_n}$. If we pick an arbitrary vertex, then each partition in which this vertex is isolated corresponds to a partition of the remaining $n-1$ vertices. Hence, there are $B_{n-1}$ partitions in which this vertex is isolated, leading to a probability $\tfrac{B_{n-1}}{B_n}$ of being isolated, and an expected number of isolated vertices equal to $\E[\rv{C}^{(1)}_n]=n\frac{B_{n-1}}{B_n}$. 

    We prove that $\E\left[\rv{M}_{n,1/2}\right]={n\choose 2}\frac{B_{n-1}}{B_n}$ by similar reasoning: for a given vertex-pair $(i,j)$, each partition where $i$ and $j$ are in the same set, can be considered as a partition of $\left([n]\setminus\{i,j\}\right)\cup \{(i,j)\}$. That is, we consider the pair $(i,j)$ as a single element, so that a partition where $i$ and $j$ are in the same set corresponds to a partition of the $n-1$ resulting elements.
    Hence, for each vertex-pair $(i,j)$, there are $B_{n-1}$ cluster graphs where $i$ and $j$ are connected, so that the connection probability is $\frac{B_{n-1}}{B_n}$. Summing over all vertex-pairs yields $\E\left[\rv{M}_{n,1/2}\right]={n\choose 2}\frac{B_{n-1}}{B_n}$, as required.
    Dividing by ${n\choose 2}$ yields the first equality of $\eqref{eq:bell-frac-density}$.
    
We now derive the asymptotics of $\frac{B_{n-1}}{B_n}$.
We write
    \[
    n\frac{B_{n-1}}{B_n}=\frac{B_{n-1}/(n-1)!}{B_n/n!}.
    \]
    We rewrite~\eqref{Bn} to
    \[
    \frac{B_n}{n!}=e^{f(r_n)+\bigO(e^{-r_n/5})},
    \]
    for $r_n=W(n+1)$, 
    where $f$ is a continuously differentiable function defined as
    \[
        f(r):=e^r-1-(re^r-1)\log r-\frac{1}{2}\left[\log(2\pi)+\log r+\log(r+1)+r\right].
    \]
In the following, we will use whenever convenient the classic asymptotics
$W(n) = \log(n) - \log\log(n) + \smallo(1)$.
    Since, $e^{-r_n/5}=\left(\frac{W(n+1)}{n+1}\right)^{1/5}\sim\left(\frac{W(n)}{n}\right)^{1/5}$, we write
    \begin{equation} \label{eq:bell-frac-partial}
    n\frac{B_{n-1}}{B_n}=e^{f(r_{n-1})-f(r_n)+\bigO({\left(\nicefrac{W(n)}{n}\right)^{1/5}})}.
    \end{equation}
By Taylor's Theorem with the Lagrange form of the remainder,
for any $n$, there exists $h \in [n, n+1]$ such that
\[
    W(n+1) =
    W(n) + W'(n) + \frac{W''(h)}{2}
\]
where
\[
    W'(n) = \frac{W(n)}{n (1 + W(n))} = \bigO(n^{-1})
    \quad \text{and} \quad
    W''(h) = - \frac{W(h)^2 (W(h) + 2)}{h^2 (W(h) + 1)^3} = \bigO(n^{-2}).
\]
Applying Taylor's Theorem again,
there exists $t \in [W(n), W(n) + W'(n) + W''(h) / 2]$
such that
\begin{equation*}
	\begin{aligned}
    	f(W(n+1))
    		&=
    		f \Big( W(n) + W'(n) + \frac{W''(h)}{2} \Big)
			\\&=
    		f(W(n))
    		+ f'(W(n)) (W'(n) + W''(h)/2)
    		+ \frac{f''(t)}{2} (W'(n) + W''(h)/2)^2
		\\&=
    		f(W(n))
    		+ f'(W(n)) W'(n) + f'(W(n)) \bigO(n^{-2})
    		+ f''(t) \bigO(n^{-2}).
	\end{aligned}
\end{equation*}
As $r$ tends to infinity, we have
\begin{equation*}
	\begin{aligned}
    	f'(r) &=
    	e^r
    	- (e^r + r e^r) \log(r)
    	- (e^r - 1/r)
    	- \frac{1}{2} \left[ \frac{1}{r} + \frac{1}{r+1} + 1 \right]
    	=
    	- (1 + r) e^r \log(r) + \bigO(1)
	\\
    	f''(r) &=
    	- e^r \log(r) - (1 + r) e^r \log(r) - \frac{1+r}{r} e^r + \bigO(1)
    	\sim
    	- r e^r \log(r),
	\end{aligned}
\end{equation*}
so
\begin{equation*}
\begin{aligned}
    f'(W(n)) \bigO(n^{-2}) &=
    - (1 + W(n)) e^{W(n)} \log(W(n)) \bigO(n^{-2}) + \bigO(n^{-2})
    =
    \bigO \left( \frac{\log \log(n)}{n} \right),
\\
    f''(t) \bigO(n^{-2}) &=
    \bigO \left( \frac{\log \log(n)}{n} \right),
\\
    f'(W(n)) W'(n) &=
    - (1 + W(n)) e^{W(n)} \log(W(n))
    \frac{W(n)}{n (1 + W(n))}
    + \bigO(n^{-1})
\\
    &=
    - \log(W(n)) + \bigO(n^{-1}).
\end{aligned}
\end{equation*}
Injecting those relations in the Taylor expansion of $f(W(n+1))$ yields
\[
    f(W(n+1)) =
    f(W(n))
    - \log(W(n))
    + \bigO \left( \frac{\log \log(n)}{n} \right).
\]
This is injected in \cref{eq:bell-frac-partial}
\[
    \frac{n B_{n-1}}{B_n}
    = e^{f(W(n) - f(W(n+1)) + \bigO((\nicefrac{W(n)}{n} )^{1/5})}
    = e^{\log(W(n)) + \bigO((\nicefrac{W(n)}{n} )^{1/5})}
    = W(n) + \bigO \left( \frac{W(n)^{6/5}}{n^{1/5}} \right).
\]
    Dividing by $n$ gives~\eqref{eq:bell-frac-density}.
\end{proof}
Lemma~\ref{lem:bell-frac-asymptotics} already tells us that $\frac{B_{n-1}}{B_n}\rightarrow0$ and that $\frac{B_{n-2}}{B_{n-1}}-\frac{B_{n-1}}{B_n}=\bigO\left(\left(\frac{\log}{n}\right)^{6/5}\right)$. However, we need tighter bounds on these decrements, which are provided in the following lemma:
\begin{lemma}\label{lem:bell-frac-bounds}
    For $w\leq 1$, $B_n(w)$ is log-convex in $n$ while $B_n(w)/n!$ is log-concave in $n$. Furthermore, the following bounds hold for all $w\leq 1$ and $n\geq 2$:
    \begin{equation}
        0\leq \frac{B_{n-2}(w)}{B_{n-1}(w)}-\frac{B_{n-1}(w)}{B_{n}(w)}\leq\frac{1}{n-1}\frac{B_{n-1}(w)}{B_{n}(w)}.
    \end{equation}
\end{lemma}
\begin{proof}
    We use the following result from \cite{bender1996log}: Let $x_n$ be a log-concave sequence and let $a_n,p_n$ be sequences such that
$$
\sum_{n=0}^\infty a_nu^n=\sum_{n=0}^\infty \frac{p_n}{n!}u^n=\exp\left(\sum_{j=1}^\infty \frac{x_j u^j}{j}\right).
$$
 Then $a_n$ is a log-concave sequence while $p_n$ is log-convex. 
For our generating function
$$
e^{C(w,z)}=\exp\left(\sum_{s=1}^\infty\frac{w^{s\choose 2}z^s}{s!}\right),
$$
this gives
$$
x_s=\frac{w^{s\choose 2}}{(s-1)!}.
$$
We compute
$$
\frac{x_s^2}{x_{s+1}x_{s-1}}=w^{2{s\choose 2}-{s+1\choose 2}-{s-1\choose 2}}\frac{s!(s-2)!}{(s-1)!^2}=w^{-s+s-1}\frac{s}{s-1}=w^{-1}\frac{s}{s-1}\geq1,
$$
for $w\leq1$, so that $x_s$ is indeed log-concave. This tells us that $B_n(w)=n![z^n]e^{C(w,z)}$ is log-convex, while $B_n(w)/n!$ is log-concave.
To prove the bounds, we follow an approach similar to~\cite{asai2000bell} to prove
\[
1\leq \frac{B_{n}(w)B_{n-2}(w)}{B_{n-1}(w)^2}\leq\frac{n}{n-1}.
\]
The lower bound follows from the log-convexity of $B_n(w)$, while the upper bound follows from the log-concavity of $B_n(w)/n!$.
Subtracting $1$ and multiplying the resulting inequalities by $B_{n-1}(w)/B_n(w)$ leads to the desired bounds.
\end{proof}

\begin{corollary}\label{cor:bell-frac-bounds}
    For $w\leq1$, Lemma~\ref{lem:bell-frac-bounds} implies
    \[
    \frac{B_{n-s}(w)}{B_n(w)}=\left(\frac{B_{n-1}(w)}{B_n(w)}\right)^s\left(1+\bigO\left(\frac{s^2}{n}\right)\right).
    \]
\end{corollary}
\begin{proof}
    Lemma~\ref{lem:bell-frac-bounds} implies that
    \[
    \frac{B_{n-2}(w)}{B_{n-1}(w)}=\frac{B_{n-1}(w)}{B_n(w)}\left(1+\bigO(n^{-1})\right).
    \]
    Repeating this relation several times, gives
    \[
    \frac{B_{n-s-1}(w)}{B_{n-s}(w)}=\frac{B_{n-1}(w)}{B_n(w)}\left(1+\bigO(n^{-1})\right)^{s}=\frac{B_{n-1}(w)}{B_n(w)}\left(1+\bigO\left(\frac{s}{n}\right)\right).
    \]
    We use this to write
    \begin{align*}
        \frac{B_{n-s}(w)}{B_{n}(w)}
        &=
        \prod_{i=0}^{s-1}\frac{B_{n-1-i}(w)}{B_{n-i}(w)}
        =
        \prod_{i=0}^{s-1}\frac{B_{n-1}(w)}{B_{n}(w)}\left(1+\bigO\left(\frac{i}{n}\right)\right)
        =\left(\frac{B_{n-1}(w)}{B_{n}(w)}\right)^s\left(1+\bigO\left(\frac{1}{n}\sum_{i=0}^{s-1}i\right)\right) \\
        & =\left(\frac{B_{n-1}(w)}{B_{n}(w)}\right)^s\left(1+\bigO\left(\frac{s^2}{n}\right)\right).
    \end{align*}
\end{proof}
\cref{cor:bell-frac-bounds} allows us to derive the expected number of cliques of a given size.

\begin{corollary}\label{cor:uniform-expected-cliques}
    The expected number of cliques of size $s$ is given by
    \[
    \E[\rv{C}_n^{(s)}]=\frac{n}{s}\proba(\rv{S}_n=s)={n\choose s}\frac{B_{n-s}}{B_n}\sim \frac{\log(n)^s}{s!}.
    \]
\end{corollary}
\begin{proof}
    For each vertex, the probability that it is in a clique of size $s$, is equal to $\proba(\rv{S}_n=s)$. Therefore, the expected number of vertices that reside in cliques of size $s$ is given by $n\cdot\proba(\rv{S}_n=s)$. Dividing by $s$ gives the expected number of cliques of that size. The asymptotics are obtained by applying Corollary~\ref{cor:bell-frac-bounds}:
 $$   
{n\choose s}\frac{B_{n-s}}{B_n}=  \frac{n!}{s!(n-s)!}\left(\frac{B_{n-1}}{B_n}\right)^{s}(1+o(1))=\frac{n!}{s!(n-s)!}\left(\frac{W(n)}{n} +o(1)\right)^{s}(1+o(1))\sim \frac{\log(n)^s}{s!}, 
$$
where we have used Stirling's approximation for deducing the last asymptotics.   
\end{proof}
The asymptotics of \cref{cor:uniform-expected-cliques} were already known in the literature: the asymptotics for fixed $s$ follows from \cite{sachkov1997probabilistic}, while the results in \cite{pittel1997random} imply the same asymptotics for $s=W(n)-\Omega(W(n)^{1/2})$. To the best of our knowledge, the exact expressions in terms of Bell numbers were not stated before.
\cref{lem:bell-frac-bounds} additionally allows us to express and bound the variance of the number of edges:

\begin{corollary}\label{cor:critical-edges-var}
    The variance of the number of edges $\rv{M}_{n,1/2}$ is given by
    \[
    \Var(\rv{M}_{n,1/2})={n\choose 2}^2\frac{B_{n-1}}{B_n}\left[\frac{B_{n-2}}{B_{n-1}}-\frac{B_{n-1}}{B_n}\right]+{n\choose 2}\left[\frac{B_{n-1}}{B_n}-\frac{B_{n-2}}{B_n}\right],
    \]
    and $\Var(\rv{M}_{n,1/2})=\Omega(n\log n)$.
\end{corollary}
\begin{proof}
    Let us write $\rv{M}_{n,1/2}=\sum_{1\leq i<j\leq n}\rv{I}_{ij}$ by means of the (random) adjacency matrix of the graph, where $\rv{I}_{ij}=1$ if $i$ and $j$ are connected in $\rv{CG}_{n,1/2}$. Hence,
    \[
        \Var(\rv{M}_n)=\sum_{1\leq i_1<j_1\leq n}\sum_{1\leq i_2<j_2\leq n}\text{Cov}(\rv{I}_{i_1j_1},\rv{I}_{i_2j_2}).
    \]
    Note that the covariance $\text{Cov}(\rv{I}_{i_1j_1},\rv{I}_{i_2j_2})$ only depends on the overlap between $\{i_1,j_1\}$ and $\{i_2,j_2\}$: if $|\{i_1,j_1\}\cap\{i_2,j_2\}|=2$, i.e., $\{i_1,j_1\}=\{i_2,j_2\}$, then
    \[
        \text{Cov}(\rv{I}_{i_1j_1},\rv{I}_{i_2j_2})=\Var(\rv{I}_{i_1j_1})=\frac{B_{n-1}}{B_n}\left(1-\frac{B_{n-1}}{B_n}\right).
    \]
    For the other cases, we write $\text{Cov}(\rv{I}_{i_1j_1},\rv{I}_{i_2j_2})=\E[\rv{I}_{i_1j_1}\rv{I}_{i_2j_2}]-\E[\rv{I}_{i_1j_1}]^2$ and will argue that $\E[\rv{I}_{i_1j_1}\rv{I}_{i_2j_2}]=\frac{B_{n-2}}{B_n}$, whenever $|\{i_1,j_1\}\cap\{i_2,j_2\}|<2$. 
 \begin{itemize}   
\item When the overlap is $1$, then we can assume w.l.o.g. that $j_1=i_2$. Then $\rv{I}_{i_1j_1}\rv{I}_{j_1j_2}=1$ whenever the vertex-triplet $\{i_1,j_1,j_2\}$ is in the same set (and, therefore, clique). Each partition of $[n]$ where the vertices $i_1,j_1,j_2$ are in the same set, can equivalently be considered as a partition of the set $\left([n]\setminus\{i_1,j_1,j_2\}\right)\cup\{(i_1,j_1,j_2)\}$, which has $n-2$ elements.
\item Similarly, when the overlap is $0$, then $\rv{I}_{i_1j_1}\rv{I}_{i_2j_2}=1$ whenever $i_1,j_1$ are in the same set and $i_2,j_2$ are in the same set. Each such partition can equivalently be considered as a partition of the set $\left([n]\setminus\{i_1,j_1,j_1,j_2\}\right)\cup\{(i_1,j_1),(i_2,j_2)\}$. Again, this set has $n-2$ elements.
\end{itemize}
Thus, by similar reasoning as in \cref{lem:bell-frac-asymptotics}, both cases lead to a probability $\frac{B_{n-2}}{B_n}$.
    We have obtained
    \[
        \Var(\rv{M}_n)={n\choose 2}\cdot\left[\frac{B_{n-1}}{B_n}-\left(\frac{B_{n-1}}{B_n}\right)^2\right]+\left[{n\choose 2}^2-{n\choose 2}\right]\cdot \left[\frac{B_{n-2}}{B_n}-\left(\frac{B_{n-1}}{B_n}\right)^2\right],
    \]
    which can be rewritten to the desired expression.
    Using the lower bound from \cref{lem:bell-frac-bounds}, we get $\Var(\rv{M}_n)\ge {n\choose 2}\left[\frac{B_{n-1}}{B_n}-\frac{B_{n-2}}{B_n}\right]\sim {n\choose 2}\frac{B_{n-1}}{B_n}\sim \frac{n}{2}\log n$.
\end{proof}
Later, \cref{prop:critical-edges} will provide the more precise asymptotics $\Var(\rv{M}_{n,1/2})\sim\tfrac{1}{4}n\log(n)^2$, but the established lower bounds will show to be helpful in the proof.

\subsection{Degree distribution}\label{sec:critical:Degree}
Corollary~\ref{cor:bell-frac-bounds} allows us to prove several properties of the degree $\rv{D}_{n,1/2}$ of a random vertex. Note that the degree of a randomly chosen vertex is equal to its clique size minus one.
\begin{proof}[Proof of Theorem~\ref{thm:mainDegree}~(ii), first statement.]
We start with proving the first statement in Part~(ii) and derive the asymptotic equality of the PGF of \(\rv{S}_{n}\) and the one of an appropriate Poisson random variable. We derive for the PGF of $\rv{S}_n$,
\begin{equation}
\begin{aligned}
    \PGF_{\rv{S}_n}(z)
    &=\sum_{s=1}^n\proba(\rv{S}_n=s)z^s
    	=\sum_{s=1}^n{n-1\choose s-1}\frac{B_{n-s}}{B_n}z^s
		=\frac{1}{n}\sum_{s=1}^n \frac{1}{(s-1)!}\left(\frac{nB_{n-1}}{B_n}z\right)^s\left(1+\bigO(s^2/n)\right)\\
	&=\frac{1}{n}\sum_{s=1}^n \frac{s}{s!}\left(\frac{nB_{n-1}}{B_n}z\right)^s+\bigO\left(\frac{1}{n^2}\sum_{s=1}^n \frac{s^3}{s!}\left(\frac{nB_{n-1}}{B_n}z\right)^s\right)\\
	&=\frac{1}{n}\sum_{s=1}^n \frac{s}{s!}\left(\kappa_nz\right)^s+\bigO\left(\frac{1}{n^2}\sum_{s=1}^n \frac{s^3}{s!}\left(\kappa_nz\right)^s\right)\label{eq:pgf-uniform1},
\end{aligned}
\end{equation}
where $\kappa_n=\frac{nB_{n-1}}{B_n}$.
We now show that, by introducing a small correction, these sums can be extended to infinity. Let us consider $\frac{s^r}{s!}(\kappa_nz)^s$, $r\in\{1,3\}$, to include both terms in \eqref{eq:pgf-uniform1}. We obtain
\begin{equation*}
    \begin{aligned}
        \left|\frac{s^r}{s!}(\kappa_nz)^s\right|&=\frac{s^r}{s!}(\kappa_n|z|)^s
            =\exp\left(-s\log s+s+\bigO(\log s)+s(\log\kappa_n+\log |z|)\right)\\
            &=\exp\left(s(\log\kappa_n+\log|z|+1-\log s)+\bigO(\log s)\right).
    \end{aligned}
\end{equation*}
Note that $\kappa_n=W(n)+o(1)$ so that for $s>n$, we have 
\[
\log\kappa_n+\log|z|+1-\log s<\log(W(n)+o(1))+\log|z|+1-\log n\rightarrow-\infty.
\] 
Thus, for any $\varepsilon>0$, there is a $N(\varepsilon)<\infty$ such that for all $s>n>N(\varepsilon)$ we have
\[
\left|\frac{s^r}{s!}(\kappa_nz)^s\right|<\varepsilon^s,
\]
so that
$$
\left|\sum_{s=n+1}^\infty\frac{s^r}{s!}(\kappa_nz)^s\right|<\sum_{s=n+1}^\infty\varepsilon^s=\frac{\varepsilon^{n+1}}{1-\varepsilon}=o(1).
$$
This results in
$$
\sum_{s=1}^n \frac{s^r}{s!}\left(\kappa_nz\right)^s=\sum_{s=1}^\infty \frac{s^r}{s!}\left(\kappa_nz\right)^s+o(1).
$$
Thus, we can extend the sums in~\eqref{eq:pgf-uniform1} to infinity:
\[
\PGF_{\rv{S}_n}(z)=\frac{1}{n}\sum_{s=1}^\infty\frac{s}{s!}\left(\kappa_nz\right)^s+\bigO\left(\frac{1}{n^2}\sum_{s=1}^\infty \frac{s^3}{s!}\left(\kappa_nz\right)^s\right)+o(1).
\]
We now rewrite the error term
\begin{align*}
    \frac{1}{n^2}\sum_{s=1}^\infty \frac{s^3}{s!}\left(\kappa_nz\right)^s&=\frac{1}{n^2}(x\partial_x)^3e^x\vline_{x=\kappa_nz}
    =\frac{1}{n^2}x(1+3x+x^2)e^x\vline_{x=\kappa_nz}
    =\bigO\left(\frac{\kappa_n^3}{n^2}e^{\kappa_nz}\right)
\end{align*}
Lemma~\ref{lem:bell-frac-asymptotics} tells us that $\kappa_n=W(n)+o(1)$. We write
$$
\frac{\kappa_n^3}{n^2}e^{\kappa_nz}\sim \frac{W(n)^3}{n^2}e^{W(n)z}=W(n)\frac{W(n)^2e^{2W(n)}}{n^2}e^{(z-2)W(n)}=W(n)e^{(z-2)W(n)}.
$$
Thus, our error term is $\bigO\left(W(n)e^{(z-2)W(n)}\right)$. Similarly, we write
\begin{align*}
\sum_{s=1}^\infty\frac{1}{(s-1)!}\left(\kappa_nz\right)^s
&=(x\partial_x)e^x\vline_{x=\kappa_nz}=xe^x\vline_{x=\kappa_nz}
=\kappa_nze^{\kappa_nz}\sim zne^{(z-1)W(n)},
\end{align*}
where we again used $\kappa_n= W(n)+o(1)$. After dividing this by $n$, we get something of the order $\bigO(e^{(z-1)W(n)})$, which is of larger order than our error term $\bigO\left(W(n)e^{(z-2)W(n)}\right)$. Everything combined, we thus conclude
$$
\PGF_{\rv{S}_n}(z)\sim ze^{(z-1)W(n)}.
$$
The result follows from the fact that $\rv{D}_{n,1/2}=\rv{S}_{n,1/2}-1$, so that $\PGF_{\rv{D}_n}(z)=\PGF_{\rv{S}_n}(z)/z$.
\end{proof}

In order to prove the second statement in Part~(ii) of the theorem, we first require the result of Corollary~\ref{CLT_Deg}, which we prove next. 

\begin{proof}[Proof of Corollary \ref{CLT_Deg}]
    We have to show that
    \[
        \frac{\rv{D}_{n,\nicefrac{1}{2}}-W(n)}{\sqrt{W(n)}}\longrightarrow\mathcal{N}(0,1),
    \]
    in distribution, as \(n\to\infty\).
    We obtain the Moment Generating Function (MGF) of $\rv{D}_n$ by substituting $z=e^\alpha$ in the PGF. That is, $\text{MGF}_{\rv{D}_n}(\alpha)=\PGF_{\rv{D}_n}(e^{\alpha})$. To show that $\sigma^{-1}\cdot (\rv{D}_n-\mu_n)$ converges in distribution to a standard Gaussian, we show that its MGF converges to the MGF of a standard Gaussian. That is, we prove that
    \[
    \text{MGF}_{\rv{D}_n}\left(\frac{\alpha}{\sigma_n}\right)e^{-\frac{\alpha\mu_n}{\sigma_n}}\rightarrow e^{\frac{\alpha^2}{2}},
    \]
    for every $\alpha$.
    We use the asymptotics from the first statement of \Cref{thm:mainDegree}~(ii) and get
    \begin{equation*}
    \begin{aligned}
        \text{MGF}_{\rv{D}_n}\left(\frac{\alpha}{\sigma_n}\right)e^{-\frac{\alpha\mu_n}{\sigma_n}}
        &\sim e^{\left(e^{\alpha/\sqrt{W(n)}}-1\right)\cdot W(n)}e^{-\frac{\alpha W(n)}{\sqrt{W(n)}}}\\
        &=
        \exp\left(\left(1+\alpha W(n)^{-1/2}+\frac{\alpha^2}{2}W(n)^{-1}+\bigO(W(n)^{-3/2})-1\right)\cdot W(n)-\alpha\sqrt{W(n)}\right)\\
        &=
        \exp\left( \frac{\alpha^2}{2}+\bigO(W(n)^{-1/2}) \right)\longrightarrow e^{\frac{\alpha^2}{2}},
    \end{aligned}
    \end{equation*}
    as required.
\end{proof}

\begin{proof}[Proof of \Cref{thm:mainDegree}~(ii), second statement.]
We now finish the proof of Part~(ii) of \Cref{thm:mainDegree}. Recall that the total variation (TV) distance is given by
    \begin{align*}
        d_{TV}(\rv{D}_n,X_n)=\frac{1}{2}\sum_{d=0}^\infty|\proba(\rv{D}_n=d)-\proba(X_n=d)|.
    \end{align*}
    From \cref{CLT_Deg}, it follows that
    $\frac{\rv{D}_n-W(n)}{\sqrt{W(n)}}\rightarrow\mathcal{N}(0,1)$ in distribution, so that
    \[
        \proba\left(\left|\rv{D}_n-W(n)\right|>W(n)^{3/4}\right)=\proba\left(\left|\frac{\rv{D}_n-W(n)}{\sqrt{W(n)}}\right|>W(n)^{1/4}\right)\longrightarrow0.
    \]
    The same holds true for our Poisson random variable, i.e.\
    \[
        \proba(|X_n-W(n)|>W(n)^{3/4})\rightarrow0.
    \]
    Let us define the set $\Delta_n$, given by
    \[
        \Delta_n=\left[W(n)-W(n)^{3/4},W(n)+W(n)^{3/4}\right]\cap\mathbb{Z},
    \]
    and define
    \[
        d_n=\arg\max_{d\in\Delta_n}|\proba(\rv{D}_n=d)-\proba(X_n=d)|.
    \]
    We can bound
    \begin{equation}\label{eq:degree-tv-uniform}
    \begin{aligned}
        d_{TV}(\rv{D}_n,X_n)&=\frac{1}{2}\sum_{d=0}^\infty|\proba(\rv{D}_n=d)-\proba(X_n=d)| \\
        &<\frac{1}{2}\sum_{d\in\Delta_n}^\infty|\proba(\rv{D}_n=d)-\proba(X_n=d)|+\frac{1}{2}\left(\proba(\rv{D}_n\not\in\Delta_n)+\proba(X_n\not\in\Delta_n)\right) \\
        &<\frac{|\Delta_n|}{2}|\proba(\rv{D}_n=d_n)-\proba(X_n=d_n)|+\frac{1}{2}\left(\proba(\rv{D}_n\not\in\Delta_n)+\proba(X_n\not\in\Delta_n)\right).    
        \end{aligned}
    \end{equation}
    In the remainder of the proof, we prove that for $d=\bigO(\log n)$,
    \[
        \proba(\rv{D}_n=d)=\proba(X_n=d)\left(1+\bigO\left(\frac{\log\log n}{\log n}\right)\right).
    \]
    We inspect the fraction:
    \begin{equation*}
    \begin{aligned}
        \frac{\proba(\rv{D}_n=d)}{\proba(X_n=d)}
        &= \frac{\proba(\rv{S}_n=d+1)}{\proba(X_n=d)} 
        = \frac{{n-1\choose d}\frac{B_{n-d-1}}{B_n}}{e^{-W(n)}\frac{W(n)^d}{d!}} 
        = \frac{\frac{(n-1)!}{(n-d-1)!d!}\frac{B_{n-d-1}}{B_n}}{\frac{W(n)^{d+1}}{n\cdot d!}} \\
        &= W(n)^{-d-1}\frac{n!}{B_n}\frac{B_{n-d-1}}{(n-d-1)!}
        = W(n)^{-d-1}\prod_{i=0}^d(n-i)\frac{B_{n-i-1}}{B_{n-i}}.
    \end{aligned}
    \end{equation*}
    We take $d=\bigO(\log n)$ and use Lemma~\ref{lem:bell-frac-bounds} to write $(n-i)\frac{B_{n-i-1}}{B_{n-i}}=W(n-i)+\bigO\left(\frac{\log\log n}{\log n}\right)$:
    \begin{equation*}
    \begin{aligned}
        \frac{\proba(\rv{D}_n=d)}{\proba(X_n=d)}
        &= \prod_{i=0}^d\frac{W(n-i)+\bigO\left(\frac{\log\log (n-i)}{\log (n-i)}\right)}{W(n)} 
        = \prod_{i=0}^d\left(1+\bigO\left(\frac{i}{n\log n}\right)+\bigO\left(\frac{\log\log n}{(\log n)^2}\right)\right)\\
        &= 1+\bigO\left(\frac{d\log\log n}{(\log n)^2}\right)
        =1+\bigO\left(\frac{\log\log n}{\log n}\right),
    \end{aligned}
    \end{equation*}
    where we used $d=\bigO(\log n)$ in the last step.
    Note that $d_n<W(n)+W(n)^{3/4}=\bigO(\log n)$, so that the bound in~\eqref{eq:degree-tv-uniform} is
    \begin{align*}
        d_{TV}(\rv{D}_n,X_n)<&\frac{|\Delta_n|}{2}|\proba(\rv{D}_n=d_n)-\proba(X_n=d_n)|+\frac{1}{2}\left(\proba(\rv{D}_n\not\in\Delta_n)+\proba(X_n\not\in\Delta_n)\right)\\
        =&\bigO\left((\log n)^{3/4}\right)\cdot \bigO\left(\frac{\log\log n}{\log n}\right)+o(1)
        =o(1).
    \end{align*}
\end{proof}

\subsection{Number of edges}\label{sec:critical:Edges}
We now use the saddle point method to prove that the edges satisfy the Gaussian limit law that is claimed in \cref{thm:mainEdges}~(ii). Unfortunately, it is not possible to apply \cref{th:normallimitlaw} to $\rv{M}_{n,\nicefrac{1}{2}}$ as its MGF is given by $B_n(e^t)/B_n$, which diverges for every positive $t$. This is a consequence of the fact that $B_n(w)=[z^n]e^{C(w,z)}$, while $C(w,z)$ diverges for all $|w|>1$, reflecting the phase transition at \(w=1\). To circumvent this divergence, we move along the imaginary axis and derive the characteristic function of $\rv{M}_{n,\nicefrac{1}{2}}$, which requires us to perform the saddle point analysis from scratch. Because we are working with characteristic functions, we have to consider series with complex coefficients instead of restricting to real coefficients.

\begin{proposition}\label{prop:critical-edges}
Let \(\mu_{n}=\E\rv{M}_{n,1/2}\) and \(\sigma_{n}^2=\Var(\rv{M}_{n,1/2})\). The number of edges satisfies a Gaussian limit law with $\mu_n\sim \tfrac{1}{2}nW(n)$ and $\sigma_n^2\sim \tfrac{1}{4}nW(n)^2$. That is,
\[
    \frac{\rv{M}_{n,1/2}-\tfrac{1}{2}nW(n)}{\tfrac{1}{2}W(n)\sqrt{n}}\stackrel{}{\longrightarrow}\mathcal{N}(0,1),
\]
in distribution, as \(n\to\infty\).
\end{proposition}
\begin{proof}

To prove this result, it is sufficient to focus on asymptotics up to polynomial orders of $W(n)$. To that end, we write $f_n=\tilde{\bigO}(g_n)$ whenever $f_n/g_n$ is bounded by a polynomial in $W(n)$. For example, we may write $n=\tilde{\bigO}(e^{W(n)})$.
To prove \cref{prop:critical-edges}, we make use of ~\cref{lem:Crs-asymp,lem:Crs-leading,lem:Crs-saddle}.
    We use the saddle point method to prove that the characteristic function of this rescaled random variable converges to the characteristic function of a Gaussian. That is, for all $\alpha\in\mathbb{R}$, we prove
$$
\E\left[ \exp\left( i\alpha\frac{\rv{M}_n-\mu_n}{\sigma_n} \right) \right]=\frac{B_n(e^{i\alpha/\sigma_n})}{B_n}e^{-i\alpha\frac{\mu_n}{\sigma_n}}\rightarrow e^{-\alpha^2/2}.
$$
We are thus interested in deriving the asymptotics of 
$$
\E[e^{is\rv{M}_n}]=\frac{B_n(e^{is})}{B_n}=\frac{n![z^n]e^{C(e^{is},z)}}{n![z^n]e^{C(1,z)}}.
$$
We use Cauchy's integration formula to rewrite
\begin{equation}
    \begin{aligned}
        B_n(e^{is})
            &=n![z^n]e^{C(e^{is},z)}
            =\frac{n!}{2\pi i}\oint\frac{e^{C(e^{is},z)}}{z^{n+1}}dz
            =\frac{n!}{2\pi}\int_{-\pi}^\pi \frac{e^{C(e^{is},\zeta e^{i\theta})}}{\zeta^{n}e^{ni\theta}}d\theta \\
            & =\frac{n!}{2\pi}\int_{-\pi}^\pi e^{C(e^{is},\zeta e^{i\theta})-ni\theta-n\log\zeta}d\theta,\label{eq:critical-edges-cauchy1}
    \end{aligned}
\end{equation}
where we substituted $z=\zeta e^{i\theta}$ with $dz=i\zeta e^{i\theta}d\theta$. We find a saddle point by substituting $\theta=0$ and taking the derivative of the log of the integrand \wrt $\zeta$:
\[
\partial_\zeta\left[C(e^{is},\zeta)-n\log\zeta\right]=\partial_\zeta C(e^{is},\zeta)-\frac{n}{\zeta}=0.
\]
Multiplying by $\zeta$ we obtain
\begin{equation}\label{eq:critical-edges-saddle}
\zeta\partial_\zeta C(e^{is},\zeta)=C_1(e^{is},\zeta)=n.
\end{equation}
We do not aim to find the saddle point exactly. Instead, we derive an \emph{approximate} saddle point $\zeta_s$ as an explicit function of $W(n)$ and $s$. 
We only need asymptotics for $s=\bigO(\sigma_n^{-1})=\tilde{\bigO}(e^{-W(n)/2})$, since \cref{cor:critical-edges-var} tells us that $\sigma_n^2=\Omega(n\log n)$.
We use the saddle point from \cref{lem:Crs-saddle}
\[
    \zeta_s =
    W(n)
    \exp \left( a_1 s + \frac{a_2}{2} s^2 \right)
\]
which tells us that the $a_1,a_2$ for which the linear and quadratic terms of \cref{lem:Crs-asymp} in $s$ cancel, are given by
\[
a_1=-\frac{i}{2}\left(\frac{c_3}{c_2}-1\right),\quad\text{and}\quad a_2=\frac{1}{c_2}\left(\frac{1}{4}(c_5-c_4+c_3)-a_1^2c_3-a_1i(c_4-c_3)\right),
\]
where $c_r=C_r(1,W(n))$.
Note that $\text{Re}(a_1)=0$ and $\text{Im}(a_2)=0$, so that $|\zeta_s|=W(n)\cdot e^{\tfrac{a_2}{2}s^2}=W(n)+\tilde{\bigO}(s^2)$ and $\text{Arg}(\zeta_s)=\text{Im}(a_1)s$.
We now split the integration of~\eqref{eq:critical-edges-cauchy1} into three parts: Firstly, there is the \emph{central} part, consisting of $\theta$ with $|\theta|\le\varepsilon=e^{-\tfrac{2}{5}W(n)}$. Secondly, the \emph{short tails} are defined by $|\theta|\in (\varepsilon,\delta)$, for $\delta=e^{-W(n)/5}$. And finally, the \emph{long tails} are defined by $|\theta|\in(\delta,\pi]$. 
That is,
\begin{equation}\label{eq:critical-edges-integral-split}
\begin{aligned}
B_n & (e^{is}) \\
 & =\frac{n!}{2\pi\zeta_s^n}\left[\int_{|\theta|\le\varepsilon} e^{C(e^{is},\zeta_s e^{i\theta})-ni\theta}d\theta+\int_{\varepsilon<|\theta|\le\delta} e^{C(e^{is},\zeta_s e^{i\theta})-ni\theta}d\theta+\int_{\delta<|\theta|\le\pi} e^{C(e^{is},\zeta_s e^{i\theta})-ni\theta}d\theta\right].
\end{aligned}
\end{equation}
We show that the contribution of the central part will give us the desired asymptotics, while the contributions of the short and long tails will be negligible.

\paragraph{The central part.}
To approximate the integral in the range $|\theta|\le\varepsilon$, we take the Taylor expansion of $C(e^{is},\zeta_se^{i\theta})$ around $\theta=0$. Recall that $C_r(w,z)=(z\partial_z)^r C(w,z)$. We obtain
\[
C(e^{is},\zeta_s e^{i\theta})=C(e^{is},\zeta_s)+\theta i C_1(e^{is},\zeta_s)-\frac{\theta^2}{2}C_2(e^{is},\zeta_s)+\int_0^\theta\frac{(\theta-t)^2}{2}C_3(e^{is},\zeta_s e^{it})dt.
\]
Note that the first term is constant \wrt $\theta$, so that it can be taken out of the integral. The part of the exponent of~\eqref{eq:critical-edges-cauchy1} that does depend on $\theta$ is given by
\[
\theta i C_1(e^{is},\zeta_s)-ni\theta-\frac{\theta^2}{2}C_2(e^{is},\zeta_s e^{i\theta})+\int_0^\theta\frac{(\theta-t)^2}{2}C_3(e^{is},\zeta_s e^{it})dt.
\]
Notice that the linear term in $\theta$ is $\tilde{\bigO}(s^3e^{W(n)})=\tilde{\bigO}(e^{-W(n)/2})$ by our choice of $\zeta_s$ from \cref{lem:Crs-saddle} and because $s=\tilde{\bigO}(e^{-W(n)/2})$.
To bound the remainder integral, we take its absolute value:
\begin{align*}
    \left|\int_0^\theta\frac{(\theta-t)^2}{2}C_3(e^{is},\zeta_s e^{it})dt\right|
    &\leq\int_0^{|\theta|}\frac{(|\theta|-t)^2}{2}\left|C_3(e^{is},\zeta_s e^{it})\right|dt\\
&\leq\int_0^{|\theta|}\frac{(|\theta|-t)^2}{2}C_3(1,|\zeta_s|)dt=C_3(1,|\zeta_s|)\int_0^{|\theta|}\frac{t^2}{2}dt\\
&=\frac{|\theta|^3}{6}C_3(1,|\zeta_s|)\sim\tilde{\bigO}(\theta^3e^{W(n)}),
\end{align*}
where in the last passage we used Lemma \ref{lem:Crs-leading} and \ref{lem:Crs-asymp}.
Thus, for $|\theta|\le\varepsilon$, the integral is of the order $\tilde{\bigO}(\varepsilon^3e^{W(n)})=\tilde{\bigO}(e^{-W(n)/5})$, being $\varepsilon=e^{-\tfrac{2}{5}W(n)}$.
Notice that this is of a higher order than the $\tilde{\bigO}(e^{-W(n)/2})$ term from the linear term in $\theta$.
We have thus obtained
\begin{equation}\label{splitting}
\int_{|\theta|\le\varepsilon} e^{C(e^{is},\zeta_s e^{i\theta})-ni\theta}d\theta=e^{C(e^{is},\zeta_s)}\int_{|\theta|\le\varepsilon} \exp\left(-\frac{\theta^2}{2}C_2(e^{is},\zeta_s)+\tilde{\bigO}(e^{-W(n)/5})\right)d\theta.
\end{equation}
To deal with the error term in the exponent, we use $\exp\left(\tilde{\bigO}(e^{-W(n)/5})\right)=1+\tilde{\bigO}(e^{-W(n)/5})$ to write
\begin{equation}\label{eq:critical-edges-central-split}
\int_{|\theta|\le\varepsilon} \exp\left(-\frac{\theta^2}{2}C_2(e^{is},\zeta_s)+\tilde{\bigO}(e^{-W(n)/5})\right)d\theta=\int_{|\theta|\le\varepsilon} e^{-\frac{\theta^2}{2}C_2(e^{is},\zeta_s)}d\theta+\int_{|\theta|\le\varepsilon} e^{-\frac{\theta^2}{2}C_2(e^{is},\zeta_s)}\cdot\tilde{\bigO}(e^{-W(n)/5})d\theta.
\end{equation}
We bound the second integral by taking its absolute value
\begin{align*}
\left|\int_{|\theta|\le\varepsilon} e^{-\frac{\theta^2}{2}C_2(e^{is},\zeta_s)}\cdot\tilde{\bigO}(e^{-W(n)/5})d\theta\right|
&\le \int_{|\theta|\le\varepsilon} e^{-\frac{\theta^2}{2}\text{Re}(C_2(e^{is},\zeta_s))}\cdot\tilde{\bigO}(e^{-W(n)/5})d\theta\\
&\le \tilde{\bigO}(e^{-W(n)/5})\cdot \int_{-\infty}^\infty e^{-\frac{\theta^2}{2}\text{Re}(C_2(e^{is},\zeta_s))}d\theta.\\
\end{align*}
We then perform the substitution $z=\theta\sqrt{\text{Re}(C_2(e^{is},\zeta_s))}$ to rewrite this to
\begin{equation}\label{eq:critical-edges-central-error}
\tilde{\bigO}(e^{-W(n)/5})\cdot\frac{1}{\sqrt{\text{Re}(C_2(e^{is},\zeta_s))}} \int_{-\infty}^\infty e^{-\frac{z^2}{2}}d z=\tilde{\bigO}(e^{-W(n)/5})\cdot\sqrt{\frac{2\pi}{\text{Re}(C_2(e^{is},\zeta_s))}}=\tilde{\bigO}(e^{-\frac{7}{10}W(n)}),
\end{equation}
since $C_2(e^{is},\zeta_s)\sim c_2 \sim W(n)^2 e^{W(n)}$ by \cref{lem:Crs-asymp,lem:Crs-leading}. We will now inspect the main contribution of \eqref{eq:critical-edges-central-split} and show that it is of a higher order than $\tilde{\bigO}(e^{-\frac{7}{10}W(n)})$, so that this error term is indeed negligible.

We use a similar substitution $z=\theta\sqrt{C_2(e^{is},\zeta_s)}$ to write
\[
\int_{|\theta|\le\varepsilon} e^{-\frac{\theta^2}{2}C_2(e^{is},\zeta_s)}d\theta
=\frac{1}{\sqrt{C_2(e^{is},\zeta_s)}}
\int_{-\varepsilon\sqrt{C_2(e^{is},\zeta_s)}}^{\varepsilon\sqrt{C_2(e^{is},\zeta_s)}}e^{-\frac{z^2}{2}}dz.
\]
Again, $C_2(e^{is},\zeta_s)\sim c_2$ by \cref{lem:Crs-asymp}, so that by \cref{lem:Crs-leading}, we have $\varepsilon\sqrt{C_2(e^{is},\zeta_s)}\sim W(n) e^{W(n)/10}\rightarrow\infty$.
Therefore,
\[
\frac{1}{\sqrt{C_2(e^{is},\zeta_s)}}
\int_{-\varepsilon\sqrt{C_2(e^{is},\zeta_s)}}^{\varepsilon\sqrt{C_2(e^{is},\zeta_s)}}e^{-\frac{z^2}{2}}dz\sim \frac{1}{\sqrt{c_2}}\int_{-\infty}^\infty e^{-\frac{z^2}{2}}dz=\sqrt{\frac{2\pi}{c_2}}.
\]
This is of the order $\tilde{\bigO}(e^{-W(n)/2})$, so that \eqref{eq:critical-edges-central-error} is indeed negligible. In conclusion, going back to \eqref{splitting}, the first integral of \eqref{eq:critical-edges-integral-split} (\emph{central} part) has asymptotics
\[
\int_{|\theta|\le\varepsilon} e^{C(e^{is},\zeta_s e^{i\theta})-ni\theta}d\theta\sim \sqrt{\frac{2\pi}{c_2}}e^{C(e^{is},\zeta_s)}.
\]

\paragraph{The tails.}
Now, for the short and long tails, we will show that the contribution is of a lower order than $e^{C(e^{is},\zeta_s)}c_2^{-1/2}$.
We rewrite the absolute value of the integrand:
\begin{align*}
    \left|e^{C(e^{is},\zeta_s e^{i\theta})-ni\theta}\right|
    =\exp\left(\text{Re}(C(e^{is},\zeta_s e^{i\theta}))\right).
\end{align*}
Thus it is sufficient to show that for all $|\theta|>\varepsilon$, it holds that
\[
\text{Re}(C(e^{is},\zeta_s e^{i\theta}))-\text{Re}(C(e^{is},\zeta_s))+\tfrac{1}{2}\log c_2\rightarrow-\infty.
\]
Note that $\log c_2=W(n)+\bigO(\log W(n))=\tilde{\bigO}(1)$. In addition, \cref{lem:Crs-asymp} tells us that 
\[
C(e^{is},\zeta_s)=c_0+s\cdot\left(a_1c_1+\frac{i}{2}(c_2-c_1)\right)+\tilde{\bigO}(1).
\]
Since $\text{Re}(a_1)=0$, this yields $\text{Re}(C(e^{is},\zeta_s))=c_0+\tilde{\bigO}(1)=e^{W(n)}+\tilde{\bigO}(1)$.
In conclusion, it is sufficient to prove that for all $|\theta|>\varepsilon$, it holds that
\begin{equation}\label{eq:critical-edges-tails-sufficient}
\text{Re}(C(e^{is},\zeta_s e^{i \theta}))-e^{W(n)}+\tilde{\bigO}(1)\rightarrow-\infty,
\end{equation}
where the divergence needs to be faster than any polynomial in $W(n)$.

We take the expansion of $C(e^{is},\zeta_s e^{i\theta})$ \wrt $s$ in the first argument and use the relation $\partial_{w=1} C(w, z) = \frac{1}{2 w} (C_2(w,z) - C_1(w,z)) = z^2 e^z$
\begin{equation} \label{eq:critical-real-monster}
\begin{aligned}
    \text{Re}\left(C(e^{is},\zeta_s e^{i\theta})\right)
    &=\text{Re}\left(C(1,\zeta_s e^{i\theta})+\frac{is}{2}(C_2(1,\zeta_s e^{i\theta})-C_1(1,\zeta_s e^{i\theta}))+\tilde{\bigO}(s^2e^{\zeta_0})\right)\\
    &=\text{Re}\left(e^{\zeta_s e^{i\theta}}\right)+\text{Re}\left(\frac{i s}{2} \zeta_s^2e^{2i\theta}  e^{\zeta_s e^{i\theta}}\right)+\tilde{\bigO}(1).
\end{aligned}
\end{equation}
We compute the real parts of these terms:
\begin{equation*}
    \begin{aligned}
        \text{Re}\left(e^{\zeta_s e^{i\theta}}\right)
            &=\left|e^{\zeta_s e^{i\theta}}\right| \cos\text{Arg}\left(e^{\zeta_s e^{i\theta}}\right)
            =e^{\text{Re}\left(\zeta_s e^{i\theta}\right)} \cos\text{Im}\left(\zeta_s e^{i\theta}\right)\\
            &=e^{|\zeta_s| \cos(\text{Arg}\zeta_s+\theta)} \cos\left(|\zeta_s| \sin(\text{Arg}\zeta_s+\theta)\right).
    \end{aligned}
\end{equation*}
We now get rid of $|\zeta_s|$ and $\text{Arg}\zeta_s$. We write $\cos(\theta+\text{Arg}\zeta_s)=\cos(\theta)\cos\text{Arg}\zeta_s-\sin(\theta)\sin\text{Arg}\zeta_s$ and use $|\zeta_s|=W(n) e^{a_2s^2/2}=W(n)+\tilde{\bigO}(s^2)$, and $\text{Arg}\zeta_s=\text{Im}(a_1)s=-\tfrac{s}{2}\frac{W(n)^2+2W(n)}{W(n)+1}$, so that $\cos\text{Arg}\zeta_s=1+\tilde{\bigO}(s^2)$ and $\sin\text{Arg}\zeta_s=-\tfrac{s}{2}\frac{W(n)^2+2W(n)}{W(n)+1}+\tilde{\bigO}(s^3)$.

This allows us to rewrite the exponent to
$$
|\zeta_s|\cos(\theta+\text{Arg}\zeta_s)=W(n)\cos\theta+\frac{s}{2}\frac{W(n)^3+2W(n)^2}{W(n)+1}\sin\theta+\tilde{\bigO}(s^2),
$$
so that
$$
e^{|\zeta_s|\cos(\theta+\text{Arg}\zeta_s)}=e^{W(n)\cos\theta}\left(1+\frac{s}{2}\frac{W(n)^3+2W(n)^2}{W(n)+1}\sin\theta\right)+\tilde{\bigO}(1).
$$
For the other factor, we write 
\begin{align*}
    \cos(|\zeta_s|\sin(\theta+\text{Arg}\zeta_s))
&=
    \cos \left(
        (W(n)+\tilde{\bigO}(s^2))
        \cdot
        \left(
            \sin \theta
            - \frac{s}{2} \frac{W(n)^2+2W(n)}{W(n)+1}
            \cos \theta
            + \tilde{\bigO}(s^2)
        \right)
    \right)
\\&=
    \cos \left(
        W(n) \sin\theta-\frac{s}{2}\frac{W(n)^3+2W(n)^2}{W(n)+1}\cos\theta
    \right)
    + \tilde{\bigO}(s^2)
\\&=
    \cos(W(n)\sin\theta)
    + \sin(W(n)\sin\theta)
    \frac{s}{2}
    \frac{W(n)^3+2W(n)^2}{W(n)+1}
    \cos\theta
    +\tilde{\bigO}(s^2),
\end{align*}
Combining these two terms and noting that $\cos(W(n)\sin\theta)\sin\theta+\sin(W(n)\sin\theta)\cos\theta=\sin[\theta+W(n)\sin\theta]$,
we obtain for the first term of \eqref{eq:critical-real-monster}
\[
    \real(e^{\zeta_s e^{i \theta}}) =
    e^{W(n) \cos(\theta)}
    \left(
        \cos(W(n) \sin(\theta))
        + \frac{s}{2}
        \frac{W(n)^3 + 2 W(n)^2}{W(n) + 1}
        \sin[\theta + W(n) \sin(\theta)]
    \right)
    + \tilde{\bigO}(1).
\]
For the second term of \eqref{eq:critical-real-monster}, we write
\begin{equation*}
\begin{aligned}
    \text{Re} \left(
        \tfrac{i s}{2}
        \zeta_s^2e^{2i\theta}
        e^{\zeta_s e^{i\theta}}
    \right)
    &=-\frac{s}{2}|\zeta_s|^2e^{|\zeta_s|\cos(\theta+\text{Arg}\zeta_s)}\sin\left[2\text{Arg}\zeta_s+2\theta+|\zeta_s|\sin(\theta+\text{Arg}\zeta_s)\right]+\tilde{\bigO}(1)\\
    &=-\frac{s}{2}W(n)^2\left(e^{W(n)\cos(\theta)}+\tilde{\bigO}(s)\right)\cdot\left(\sin\left[2\theta+W(n)\sin(\theta)\right]+\tilde{\bigO}(s)\right)+\tilde{\bigO}(1)
\\&=
    -\frac{s}{2}
    W(n)^2
    e^{W(n) \cos(\theta)}
    \sin[2 \theta + W(n) \sin(\theta)]
    + \tilde{\bigO}(1).
\end{aligned}
\end{equation*}

Combining those two terms, we obtain for \eqref{eq:critical-real-monster}
\begin{equation*}
\begin{aligned}\label{eq:critical-edges-tails-exponent}
    \text{Re}\left(C(e^{is},\zeta_s e^{i\theta})\right)
     & =e^{W(n)\cos\theta}\cos(W(n)\sin\theta) +\tilde{\bigO}(1) \\
    &\qquad +\frac{s e^{W(n)\cos\theta}}{2}\left(\frac{W(n)^3+2W(n)^2}{W(n)+1}\sin\left[\theta+W(n)\sin\theta\right]
    -
    W(n)^2\sin\left[2\theta+W(n)\sin\theta\right]\right).
\end{aligned}
\end{equation*}

\paragraph{The short tail.} We use \eqref{eq:critical-edges-tails-exponent} to show that \eqref{eq:critical-edges-tails-sufficient} holds for $\varepsilon<|\theta|\le\delta$ where $\varepsilon=e^{-\tfrac{2}{5}W(n)}$ and $\delta=e^{-W(n)/5}$.
We bound $e^{W(n)\cos\theta}\le e^{W(n)}$ and
\[
\cos(W(n)\sin\theta)\le\cos(W(n)\sin\varepsilon)=1-\tfrac{1}{2}W(n)^2\varepsilon^2+\tilde{\bigO}(\varepsilon^3).
\]
Hence, the first term of~\eqref{eq:critical-edges-tails-exponent} is
\[
e^{W(n)\cos\theta}\cos(W(n)\sin\theta)\le e^{W(n)}-\tfrac{1}{2}W(n)^2\varepsilon^2e^{W(n)}+\tilde{\bigO}(\varepsilon^3e^{W(n)}).
\]
What remains to show, is that the term linear in $s$ is of a lower order than $\varepsilon^2e^{W(n)}=\tilde{\bigO}(e^{W(n)/5})$.
Recall that in the current regime $\theta\rightarrow0$. Hence, we have
\[
\sin[\theta+W(n)\sin\theta]=\sin[\theta+W(n)\theta+\tilde{\bigO}(\theta^3)]=(W(n)+1)\theta+\tilde{\bigO}(\theta^3).
\]
And similarly, $\sin[2\theta+W(n)\sin\theta]=(W(n)+2)\theta+\tilde{\bigO}(\theta^3)$. Substituting this yields
\begin{align*}
    &\frac{W(n)^3+2W(n)^2}{W(n)+1}\sin\left(\theta+W(n)\sin\theta\right)-W(n)^2\sin\left(2\theta+W(n)\sin\theta\right)\\
    & \qquad=(W(n)^3+2W(n)^2)\cdot\theta-W(n)^2\cdot(W(n)+2)\cdot\theta+\tilde{\bigO}(\theta^3)\\
    & \qquad =\tilde{\bigO}(\theta^3).
\end{align*}
Therefore, the second term of~\eqref{eq:critical-edges-tails-exponent} has contribution $\tfrac{s}{2}e^{W(n)}\cdot\tilde{\bigO}(\theta^3)=\tilde{\bigO}\left(se^{W(n)}\delta^3\right)=\tilde{\bigO}(e^{(-\tfrac{1}{2}+1-3\cdot\tfrac{1}{5})W(n)})=\tilde{\bigO}(e^{-\tfrac{1}{10}W(n)})$, which is indeed negligible compared to $\tilde{\bigO}(e^{W(n)/5})$. In the first equality we used $s=\tilde{\bigO}(e^{-W(n)/2})$ and $|\theta|\le\delta=\tilde{\bigO}(e^{-W(n)/5})$.

We conclude
\[
    \text{Re}\left(C(e^{is},\zeta_s e^{i\theta})\right)-e^{W(n)}=-\tfrac{1}{2}W(n)^2\varepsilon^2e^{W(n)}+\tilde{\bigO}(1)\sim-\tfrac{1}{2}W(n)^2\varepsilon^2e^{W(n)}\rightarrow-\infty,
\]
as required.

\paragraph{The long tail.}
We now show that \eqref{eq:critical-edges-tails-sufficient} holds for $\delta<|\theta|\le\pi$, where $\delta=e^{-\tfrac{1}{5}W(n)}$.
Recall that $|\zeta_s|=W(n) e^{a_2s^2/2}=W(n)+\tilde{\bigO}(s^2)$, so that $\zeta_0=W(n)$; we simplify \eqref{eq:critical-edges-tails-exponent} to
\[
\text{Re}\left(C(e^{is},\zeta_s e^{i\theta})\right)=e^{\zeta_0\cos\theta}\left(\cos(\zeta_0\sin\theta)+\tilde{\bigO}(s)\right),
\]
and use the bounds $\cos(\zeta_0\sin\theta)\le 1$ and $\cos\theta\le\cos\delta$ to write
\begin{equation*}
    \begin{aligned}
        \text{Re}\left(C(e^{is},\zeta_s e^{i\theta})\right)
            &\le e^{W(n)\cos\delta}\left(1+\tilde{\bigO}(s)\right)
             = e^{W(n)(1-\tfrac{\delta^2}{2}+\bigO(\delta^4))}\left(1+\tilde{\bigO}(s)\right)\\
            &=e^{W(n)}\left(1-\tfrac{1}{2}W(n)e^{-2W(n)/5}+\tilde{O}(e^{-4W(n)/5})\right)(1+\tilde{\bigO}(s))\\
            &=e^{W(n)}-\tfrac{1}{2}W(n)e^{3W(n)/5}+\tilde{\bigO}(e^{W(n)/2}),
    \end{aligned}
\end{equation*}
so that $\text{Re}\left(C(e^{is},\zeta_s e^{i\theta})\right)-e^{W(n)}\sim-\tfrac{1}{2}W(n)e^{3W(n)/5}\rightarrow-\infty$ as required.

We conclude that the contributions of the tails are negligible, and
\eqref{eq:critical-edges-integral-split} is thus dominated by the central integral. This means, going back to \eqref{eq:critical-edges-integral-split}, that
\[
B_n(e^{is})\sim\frac{n!}{\sqrt{2\pi c_2}\zeta_s^n} e^{C(e^{is},\zeta_s)}.
\]
\paragraph{Conclusion.}
Therefore, the asymptotics of the MGF are given by
\[
\E[e^{is\rv{M}_n}]=\frac{[z^n]e^{C(e^{is},z)}}{[z^n]e^{C(1,z)}}\sim \left(\frac{\zeta_0}{\zeta_s}\right)^n\exp\left(C(e^{is},\zeta_s)-C(1,\zeta_0)\right)=\left(\frac{W(n)}{\zeta_s}\right)^{c_1}\exp\left(C(e^{is},\zeta_s)-c_0\right),
\]
where we used $\zeta_0=W(n),$ $C(1,W(n))=c_0$ and $n=c_1$ in the last step.
Next, we substitute $\zeta_s=W(n)\exp(a_1s+a_2s^2/2)$ and use \cref{lem:Crs-asymp} to write the exponent as
\begin{equation*}
    \begin{aligned}
        \log\E[e^{is\rv{M}_n}]
            & = -a_1c_1s-a_2c_1s^2/2+s\cdot\left(a_1c_{1}+\frac{i}{2}c_{2}-\frac{i}{2}c_{1}\right)\\
            &\phantom{\log\E[e^{is}]}+s^2\cdot\left(\frac{a_2}{2}c_{1}+\frac{a_1^2}{2}c_{2}+\frac{a_1i}{2}\left(c_{3}-c_{2}\right)-\frac{c_{4}-2c_{3}+c_{2}}{8}\right)+o(1)\\
            & = s\cdot\left(\frac{i}{2}c_{2}-\frac{i}{2}c_{1}\right)
+s^2\cdot\left(\frac{a_1^2}{2}c_{2}+\frac{a_1i}{2}\left(c_{3}-c_{2}\right)-\frac{c_{4}-2c_{3}+c_{2}}{8}\right)+o(1).
    \end{aligned}
\end{equation*}

In the linear term, we recognize the mean $\mu_n=\tfrac{1}{2}(c_2-c_1)=\tfrac{1}{2}W(n)^2e^{W(n)}=\frac{1}{2}n\cdot W(n)$. The quadratic term corresponds to the variance, which we derive as
\begin{align*}
    \sigma_n^2
    &=\frac{c_{4}-2c_{3}+c_{2}}{4}-a_1^2c_2-a_1i\left(c_{3}-c_{2}\right)
=\frac{c_{4}-2c_{3}+c_{2}}{4}+\frac{(c_3-c_2)^2}{4c_2}-\frac{(c_3-c_2)^2}{2c_2}\\
&=\frac{c_2c_4-2c_2c_3+c_2^2}{4c_2}-\frac{c_3^2-2c_2c_3+c_2^2}{4c_2}
=\frac{c_2c_4-c_3^2}{4c_2}.
\end{align*}
Recall that $c_r=(z\partial_z)^rC(1,z)|_{z=W(n)}$, which allows us to compute
\[
\frac{c_2c_4-c_3^2}{4c_2}=\frac{W(n)(W(n)+1)(W(n)^4+6W(n)^3+7W(n)^2+W(n))-W(n)^2(W(n)^2+3W(n)+1)^2}{4W(n)(W(n)+1)}e^{W(n)}.
\]
Note that the term of order $W(n)^6$ cancels, while the term of order $W(n)^5$ has coefficient $1$. Hence, the numerator is asymptotically equivalent to $W(n)^5$, while the numerator is asymptotically equivalent to $W(n)^2$. This tells us that
$\sigma_n^2\sim\frac{1}{4}W(n)^3e^{W(n)}=\frac{1}{4}W(n)^2\cdot n$.
Finally, we complete the proof by substituting $s=\alpha/\sigma_n=\tilde{\bigO}(e^{-W(n)/2})$ and writing
\[
\E\left[ \exp\left( i\alpha\frac{\rv{M}_n-\mu_n}{\sigma_n} \right) \right]\sim e^{-\frac{1}{2}\alpha^2},
\]
which concludes the proof of the proposition, and thus of \cref{thm:mainEdges}~(ii).
\end{proof}

\section{The subcritical regime ($p<\nicefrac{1}{2}$)}\label{sec:subcritical}

This section investigates the statistical properties
of $\rv{CG}_{n,p}$
when $p$ is fixed in $(0,\nicefrac{1}{2})$.
The corresponding value $w = \tfrac{p}{1-p}$
then belongs to $(0,1)$.
\cref{sec:subcritical:function}
provides analytic properties of the generating function $C(w,z)$ of cliques
and asymptotics for the generating function of cluster graphs of a given size, that are applied
in \cref{sec:subcritical:blocks}
and \cref{sec:subcritical:edges}
to derive central limit laws
for the number of cliques and edges of the graph. We further investigate the degree distribution in \Cref{sec:subcritical_degree}.

\subsection{Hayman admissibility}
\label{sec:subcritical:function}

Hayman admissibility is a classical tool
from analytic combinatorics \cite[Section~VIII.~5.1]{FS09},  \cite{hayman1956, wong1989}
to extract asymptotics from generating functions.
The next result is a small variant,
tailored to our needs.

\begin{lemma}
\label{th:hayman}
Consider a series $F(z)$ with nonnegative coefficients
and radius of convergence $\rho$ (possibly infinite),
and a series $A(z)$ with complex coefficients
and radius of convergence at least $\rho$.
Define
\[
    a(z) = z \partial_z \log(F(z)),
    \qquad
    b(z) = (z \partial_z)^2 \log(F(z)),
\]
and assume
\begin{itemize}
\item \emph{Capture condition.}
$\lim_{x \to \rho} a(x) = +\infty$
and $\lim_{x \to \rho} b(x) = +\infty$,
\item \emph{Locality condition.}
There exists a function $\theta_0(x)$
from $(0, \rho)$ to $(0,\pi)$ such that
\[
    A(x e^{i \theta})
    \underset{x \to \rho}{\sim}
    A(x),
    \qquad
    F(x e^{i \theta})
    \underset{x \to \rho}{\sim}
    F(x) e^{i a(x) \theta - b(x) \theta^2/2}
\]
uniformly for $|\theta| \leq \theta_0(x)$,
\item \emph{Decay condition.}
Uniformly for $\theta_0(x) \leq |\theta| < \pi$, we have
\[
    A(x e^{i \theta}) F(x e^{i \theta})
    =
    \smallo \left( \frac{|A(x)| F(x)}{\sqrt{b(x)}} \right).
\]
\end{itemize}
Let $\zeta = \zeta(n)$ denote the unique solution in $(0, \rho)$ of
\[
    \frac{\zeta F'(\zeta)}{F(\zeta)} = n.
\]
Then
\[
    [z^n] A(z) F(z)
    \underset{n \to +\infty}{\sim}
    \frac{A(\zeta)}{\sqrt{2 \pi b(\zeta)}}
    \frac{F(\zeta)}{\zeta^n}.
\]
\end{lemma}

\begin{proof}
The coefficient extraction is represented by its Cauchy integral
on a circle of radius $\zeta$
\[
    [z^n] A(z) F(z)
    =
    \frac{1}{2 i \pi} \oint
    A(z) F(z)
    \frac{d z}{z^{n+1}}
    =
    \frac{\zeta^{-n}}{2 \pi}
    \int_{-\pi}^{\pi}
    A(\zeta e^{i \theta})
    F(\zeta e^{i \theta})
    e^{- i n \theta}
    d \theta.
\]
The integral is cut in three parts.
The central part and the tail correspond respectively to
\[
    \Icentral =
    \frac{\zeta^{-n}}{2 \pi}
    \int_{-\theta_0(\zeta)}^{\theta_0(\zeta)}
    A(\zeta e^{i \theta})
    F(\zeta e^{i \theta})
    e^{- i n \theta}
    d \theta,
    \qquad
    \Itail =
    \frac{\zeta^{-n}}{2 \pi}
    \int_{\theta_0(\zeta)}^{\pi}
    A(\zeta e^{i \theta})
    F(\zeta e^{i \theta})
    e^{- i n \theta}
    d \theta.
\]
The third part is the negative tail,
which corresponds to the same integral on $[-\pi, -\theta_0(\zeta)]$.
We will now extract the asymptotics of the central part,
showing that it is equal to the asymptotics of the theorem,
then prove that the tail is negligible.
The proof that the negative tail is negligible being identical,
this will conclude the proof of the theorem.

\proofparagraph{Central part.}
Since $a(x)$ is defined as $\frac{x F'(x)}{F(x)}$,
we have $a(\zeta) = n$.
The capture condition ensures
$a(x)$ tends to infinity as $x$ tends to $\rho$.
It is an increasing function (see \cite{FS09})
because $F(z)$ has nonnegative coefficients.
Therefore, $\zeta$ tends to $\rho$ as $n$ tends to infinity.
The locality condition then yields
\[
    \Icentral \sim
    \frac{\zeta^{-n}}{2 \pi}
    \int_{-\theta_0(\zeta)}^{\theta_0(\zeta)}
    A(\zeta)
    F(\zeta) e^{i a(\zeta) \theta - b(\zeta) \theta^2 / 2}
    e^{- i n \theta}
    d \theta.
\]
After injecting $a(\zeta) = n$ and applying
the change of variable $x = \sqrt{b(\zeta)} \theta$,
the expression reads
\[
    \Icentral \sim
    A(\zeta)
    \frac{F(\zeta)}{\zeta^n}
    \frac{1}{2 \pi}
    \int_{-\theta_0(\zeta) \sqrt{b(\zeta)}}^{\theta_0(\zeta) \sqrt{b(\zeta)}}
    e^{- x^2 / 2}
    \frac{d x}{\sqrt{b(\zeta)}}.
\]
The combination of the locality condition
and the decay condition at $\theta = \theta_0(x)$
implies that $\theta_0(\zeta) \sqrt{b(\zeta)}$ tends to infinity.
Thus, the integral converges to the Gaussian integral and
\[
    \Icentral \sim
    \frac{A(\zeta)}{\sqrt{2 \pi b(\zeta)}}
    \frac{F(\zeta)}{\zeta^n}.
\]

\proofparagraph{Tail.}
Let us now prove that the positive tail is negligible
compared to the central part.
The same result holds for the negative tail.
The decay condition implies
\[
    \Itail \leq
    \frac{\zeta^{-n}}{2 \pi}
    \int_{\theta_0(\zeta)}^{\pi}
    |A(\zeta e^{i \theta})
    F(\zeta e^{i \theta})
    e^{- i n \theta}|
    d \theta
    =
    \smallo \left(
        \frac{A(\zeta)}{\sqrt{b(\zeta)}}
        \frac{F(\zeta)}{\zeta^n}
    \right)
\]
and the tail is hence negligible compared to the central part.
\end{proof}

The following result is an application of \cref{th:hayman}
and is our main tool for extracting asymptotics
in the subcritical regime $p < 1/2$. Recall the definition
\[
    C_r(w,z) =
    \sum_{n \geq 1}
    n^r
    w^{\binom{n}{2}}
    \frac{z^n}{n!}
\]
and $C(w,z) = C_0(w,z)$.

\begin{lemma} \label{th:application:hayman}
Consider real values $w \in (0,1)$
and $s$ in a neighborhood of $0$,
a complex value $v$ in small neighborhood of $1$,
and nonnegative integers $r$, $k$.
Then the following asymptotics hold
as $n$ tends to infinity
\[
    [z^n] C_r(w, v z)^k e^{e^s C(w,z)}
    \sim
    \sqrt{\frac{\log(\nicefrac{1}{w})}
        {2 \pi n \gamma}}
    C_r \left(w, v e^{\gamma} \right)^k
    \exp \left(
        e^s C(w, e^{\gamma})
        - n \gamma
    \right)
\]
where $\gamma=\gamma(s)$ is defined implicitly by $C_1(w,e^\gamma) = ne^{-s}$
and goes to infinity with $n$.
\end{lemma}

\begin{proof}
We will prove that the conditions of \cref{th:hayman}
are satisfied for
$A(z) = C_r(w, v z)^k$
and $F(z) = e^{e^s C(w,z)}$.
In the notations of the lemma, we have
\begin{align*}
    a(z) &=
    e^s C_1(w,z),
    \\
    b(z) &=
    e^s C_2(w,z),
    \\
    e^s C_1(w, \zeta) &=
    n,
\end{align*}
and hence $\zeta = e^{\gamma}$.
In order to apply \cref{lem:C-asymptotics},
we introduce the positive value $\tau$,
characterised by
\[
    \tau (\nicefrac{1}{w})^{\tau - 1/2} = e^{\gamma} = \zeta.
\]
Taking the logarithm, we have in particular
\[
    \gamma \sim \tau \log(\nicefrac{1}{w}).
\]
\cref{lem:C-asymptotics} then implies
\[
    e^s C_2(w, \zeta)
    \sim
    e^s \tau C_1(w, \zeta)
    =
    n \tau
    \sim
    n \gamma / \log(\nicefrac{1}{w}).
\]
Injecting this and $\zeta = e^{\gamma}$,
we obtain
\[
    [z^n] C_r(w, v z)^k e^{e^s C(w,z)}
    \sim
    \frac{C_r \left(w, v e^{\gamma} \right)^k}
        {\sqrt{2 \pi n \gamma / \log(\nicefrac{1}{w})}}
    \exp \left(
        e^s C(w, e^{\gamma})
        - n \gamma
    \right).
\]
We now check that the conditions of \cref{th:hayman} are satisfied.

\proofparagraph{Capture condition.}
The radius of convergence $\rho$ of $C(w,z)$
is infinite since $w \in (0,1)$.
For any $k \geq 0$, the function $C_k(w,x)$
tends to infinity with $x$,
implying that $a(x)$ and $b(x)$ both tend to infinity.

\proofparagraph{Locality condition.}
Set
\[
    c(z) = (z \partial_z)^3 \log(F(z)) = e^s C_3(w,z).
\]
A rule of thumb from \cite[Section~VIII.~5.1]{FS09}
is to choose $\theta_0(\zeta)$ such that
\[
    b(\zeta) \theta_0(\zeta)^2 \to +\infty
    \qquad \text{and} \qquad
    c(\zeta) \theta_0(\zeta)^3 \to 0.
\]
Given the asymptotics from \cref{lem:C-asymptotics} given by
\[
    C_k(w,\zeta) =
    w^{-\tau^2/2} e^{\exactbigO(\tau)},
\]
we choose $\theta_0(\zeta) = w^{\tau^2/5}$.

We now check that the first equation
of the locality condition is satisfied,
namely that
\begin{equation}
\label{eq:locality:A}
    C_r(w, v \zeta e^{i \theta})^k
    \sim
    C_r(w, v \zeta)^k
\end{equation}
holds uniformly for $|\theta| \leq \theta_0(\zeta) = w^{\tau^2/5}$.
Taylor's Theorem is applied at $\theta = 0$
with a remainder denoted by $R_0(\theta)$
\[
    C_r(w, v \zeta e^{i \theta})
    =
    C_r(w, v \zeta)
    + R_0(\theta),
\]
where
\[
    |R_0(\theta)|
    =
    \left|
    \int_0^{\theta}
    \partial_t C_r(w, v \zeta e^{i t}) d t
    \right|
    \leq
    C_{r+1}(w, |v| \zeta)
    w^{\tau^2/5}
\]
for any $|\theta| \leq w^{\tau^2/5}$.
\cref{lem:C-asymptotics} yields,
for $v$ in a small enough vicinity of $1$,
\[
    C_{r+1}(w, |v| \zeta)
    \sim
    |C_{r+1}(w, v \zeta)|
    \sim
    \tilde{\tau} |C_r(w, v \zeta)|
\]
where
\[
    \tilde{\tau} =
    \frac{W \left( \frac{\log(\nicefrac{1}{w})}{\sqrt{w}} |v| \zeta \right)}
    {\log(\nicefrac{1}{w})}.
\]
The Lambert function $W(x)$ is characterised, for $x > 0$, by
$W(x)e^{W(x)} = x$, so $W(x)+\log(W(x)) = \log(x)$
and we deduce $W(x) \leq \log(x)$.
This inequality is applied to
\[
    \tilde{\tau}
    \leq
    \frac{\log(\zeta)}{\log(\nicefrac{1}{w})}
    + \bigO(1)
    \leq
    \tau + \frac{\log(\tau)}{\log(\nicefrac{1}{w})} + \bigO(1).
\]
We deduce
\[
    \frac{|R_0(\theta)|}{|C_r(w, v \zeta)|}
    \leq
    \tilde{\tau} w^{\tau^2/5} (1 + \smallo(1))
    \leq
    \tau w^{\tau^2/5} (1 + \smallo(1))
    =
    \smallo(1),
\]
proving \eqref{eq:locality:A}.

We follow the same approach for the second equation
of the locality condition, namely
\begin{equation}
\label{eq:locality:F}
    e^{e^s C(w, \zeta e^{i \theta})}
    \sim
    e^{e^s C(w, \zeta)}
    e^{i e^s C_1(w, \zeta) \theta
        - e^s C_2(w, \zeta) \theta^2/2}
\end{equation}
uniformly for $|\theta| \leq w^{\tau^2/5}$.
Taylor's Theorem is applied to $C(w, \zeta e^{i \theta})$
at $\theta = 0$ with a remainder denoted by $R_2(\theta)$ to obtain
\begin{align*}
    C(w, \zeta e^{i \theta})
    &=
    C(w,\zeta)
    + \partial_{\theta=0} C(w, \zeta e^{i \theta}) \theta
    + \partial_{\theta=0}^2 C(w, \zeta e^{i \theta}) \frac{\theta^2}{2}
    + R_2(\theta)
    \\&=
    C(w, \zeta) + i a(\zeta) \theta - b(\zeta) \frac{\theta^2}{2}
    + R_2(\theta).
\end{align*}
This remainder is expressed in its integral form and bounded
in the case $|\theta| \leq \theta_0(\zeta)$ by
\begin{align*}
    |R_2(\theta)| =
    \left|
    \int_0^{\theta}
    \partial_t^3 C(w, \zeta e^{i t})
    \frac{(\theta - t)^2}{2} d t
    \right|
    \leq
    \sup_{|z| = \zeta} |C_3(w,z)| w^{\tau^2 3 / 5}.
\end{align*}
Since $C_3(w,z)$ is a series with nonnegative coefficients,
we have
\[
    \sup_{|z| = \zeta} |C_3(w,z)|
    \leq
    C_3(w,\zeta)
    = w^{-\tau^2/2} e^{\exactbigO(\tau)},
\]
so the remainder tends to $0$
and \eqref{eq:locality:F} is satisfied.

\proofparagraph{Decay condition.}
In order to prove that
\[
    C_r(w, v \zeta e^{i \theta})^k
    e^{e^s C(w, \zeta e^{i \theta})}
    =
    \smallo \left(
    \frac{
        C_r(w, v \zeta)^k
        e^{e^s C(w, \zeta)}
    }{
        \sqrt{C_2(w, \zeta)}
    }
    \right)
\]
uniformly for $w^{\tau^2/5} \leq |\tau| \leq \pi$,
it suffices to prove (looking at the logarithm)
\[
    e^s \left(
        \real(C(w, \zeta e^{i \theta}))
        - C(w, \zeta)
    \right)
    + \frac{1}{2} \log(C_2(w, \zeta))
    + k \log \left(
        \frac{|C_r(w, v \zeta e^{i \theta}|}
        {|C_r(w, v \zeta)|}
    \right)
    \to
    -\infty.
\]
By \cref{lem:C-asymptotics},
the third summand is bounded
and the second summand tends to infinity
polynomially fast in $\tau$.
Thus, it is sufficient to prove that
$\real(C(w, \zeta e^{i \theta})) - C(w, \zeta)$
tends to $-\infty$ exponentially fast in $\tau$.

Recall $\theta_0(\zeta) = w^{\tau^2/5}$.
Let us assume $\theta \in [\theta_0(\zeta), \pi]$,
the case $\theta \in [-\pi, -\theta_0(\zeta)]$ being symmetrical.
We claim that there exists a real value
$c \in [-1, 2]$ such that $\tau + c$ is an integer
and $(\tau + c) \theta$ is not
in $[-\theta_0(\zeta), \theta_0(\zeta)]$ modulo $2 \pi$.
If $\lfloor \tau \rfloor$ is not
in $[-\theta_0(\zeta), \theta_0(\zeta)]$ modulo $2 \pi$,
we take $c = \lfloor \tau \rfloor - \tau$,
which belongs to $[-1, 0]$.
Otherwise,
\begin{itemize}
\item
if $\theta \in [2 \theta_0(\zeta), \pi]$,
then $\left( \lfloor \tau \rfloor + 1 \right) \theta$
is not in $[-\theta_0(\zeta), \theta_0(\zeta)]$ modulo $2 \pi$,
and we choose $c = \lfloor \tau \rfloor - \tau + 1$,
\item
if $\theta \in [\theta_0(\zeta), 2 \theta_0(\zeta)]$,
then $\left( \lfloor \tau \rfloor + 2 \right) \theta$
is not in $[-\theta_0(\zeta), \theta_0(\zeta)]$ modulo $2 \pi$,
and we choose $c = \lfloor \tau \rfloor - \tau + 2$.
\end{itemize}

Neglecting the negative terms corresponding
to $n \theta \in [-\theta_0(\zeta),\ \theta_0(\zeta)]$ modulo $2 \pi$,
we have
\begin{equation*}
	\begin{aligned}
		\real(C(w,\zeta e^{i \theta})) - C(w,\zeta) &=
    		\sum_{n \geq 1}
    		(\cos(n \theta) - 1)
    		w^{\binom{n}{2}}
    		\frac{\zeta^n}{n!}
    	\\&\leq
    		(\cos(\theta_0(\zeta)) - 1)
    		\sum_{\substack{n \geq 1\\ n \theta \notin [-\theta_0(\zeta),\  \theta_0(\zeta)] \mod 2 \pi}}
    		w^{\binom{n}{2}}
    		\frac{\zeta^n}{n!}
    	\\&\leq
    		-\frac{\theta_0(\zeta)^2}{2}
    		\left( 1 + \bigO(\theta_0(\zeta))^2 \right)
    		w^{\binom{\tau + c}{2}}
    		\frac{\zeta^{\tau + c}}{(\tau + c)!}.
	\end{aligned}
\end{equation*}
We replace $\zeta$ with $\tau (1/w)^{\tau - 1/2}$
and apply Stirling's formula to obtain
\begin{align*}
    \real(C(w,\zeta e^{i \theta})) - C(w,\zeta)
    &\leq
    -\frac{\theta_0(\zeta)^2}{2}
    \left( 1 + \bigO(\theta_0(\zeta))^2 \right)
    w^{\binom{\tau + c}{2}}
    \frac{\tau^{\tau + c} (1/w)^{(\tau + c) (\tau - 1/2)}}
    {(\tau + c)^{\tau + c} e^{-\tau - c} \sqrt{2 \pi \tau}}
    (1 + \smallo(1))
    \\&\leq
    -\frac{\theta_0(\zeta)^2}{2}
    w^{c^2/2}
    w^{-\tau^2/2}
    \frac{e^{\tau}}{\sqrt{2 \pi \tau}}
    (1 + \smallo(1)).
\end{align*}
We replace $\theta_0(\zeta)$ with its expression $w^{\tau^2/5}$
and conclude
\begin{align*}
    \real(C(w,\zeta e^{i \theta})) - C(w,\zeta)
    &\leq
    -\frac{w^{- \tau^2/10}}{2}
    w^{c^2/2}
    \frac{e^{\tau}}{\sqrt{2 \pi \tau}}
    (1 + \smallo(1)).
\end{align*}
which tends to $-\infty$ exponentially fast in $\tau$.
\end{proof}

\subsection{Number of cliques}
\label{sec:subcritical:blocks}
In this section, we study $\rv{C}_{n,p}$, the number of cliques of \(\rv{CG}_{n,p}\). We denote respectively by 
$\mu_{n,p}:=\E(\rv{C}_{n,p})$ and $\sigma_{n,p}^{2}:=\Var(\rv{C}_{n,p})$ the mean and the variance of this random observable throughout this section. 
The aim of this section is to derive a central limit theorem for $\rv{C}_{n,p}$ in the subcritical regime $p<\nicefrac{1}{2}$ or \(w=\tfrac{p}{1-p}<1\) and thus prove \Cref{thm:mainClust}~(iii).

\begin{proposition}[CLT for the number of cliques]\label{prop:CltCluster}
Let \(p<\nicefrac{1}{2}\). Then
\begin{equation*}
	\frac{\rv{C}_{n,p} - \mu_{n,p}}{\sigma_{n,p}} \overset{}{\to} \mathcal{N}(0,1), 
\end{equation*}
in distribution as \(n\to\infty\), where 
\begin{equation*}
	\begin{aligned}
   	\mu_{n,p} 
   		&\sim\frac{n}{\sqrt{\log(n)}}\sqrt{\frac{\log\left(\tfrac{p}{1-p}\right)}{2}} \\
    \sigma_{n,p}^2 
    	&\sim\frac{ n}{\log(n)^{\nicefrac{3}{2}}}\cdot \left(\frac{\log(\tfrac{p}{1-p})}{2}\right)^{\nicefrac{3}{2}}\cdot \frac{e_{w,0}(\tau)e_{w,2}(\tau)-e_{w,1}(\tau)^2}{e_{w,0}(\tau)^2},
\end{aligned}
\end{equation*}
and the functions \(e_{w,j}(\tau)\), \(j=0,1,2\), are the periodic bounded functions, given in \Cref{lem:E-asymptotics}.
\end{proposition}
\begin{proof}
We prove the claim by showing that the moment generating function of the normalised number of cliques converges appropriately in a neighbourhood of zero. By~\Cref{th:exact:pgf} the moment generating function of \(\rv{C}_{n,p}\) is given by
\[
	 \PGF_{\rv{C}_{n,p}}(e^s) = \frac{[z^n] e^{e^s C(w,z)}}{[z^n] e^{C(w,z)}}, \quad s\in\mathbb{R}.
\]
To obtain its main asymptotics, we apply \cref{th:application:hayman} for \(v=k=0\) and consider \(\gamma(s)\) implicitly defined via \(C_1(w,e^{\gamma(s)})=n e^{-s}\). Thus, by \Cref{th:application:hayman} 
\[
	\PGF_{\rv{C}_{n,p}}(e^s)	\sim \exp\big(e^s C(w,e^{\gamma(s)})-C(w,e^{\gamma(0)})-n(\gamma(s)-\gamma(0))\big).
\]
Let us define the function \(H(s)=e^s C(w,e^{\gamma(s)})-n\gamma(s)\). We obtain from Taylor expansion with Lagrange remainder\begin{equation}\label{eq:MGFCliques}
	\PGF_{\rv{C}_{n,p}}(e^s)	\sim \exp\big(H'(0)s + H''(0) \tfrac{s^2}{2}+H'''(t)\tfrac{t^3}{6}\big), \quad \text{ for some } \quad t\in(-s,s).
\end{equation}
Using \(z \partial_z C_{r}(w,z)=C_{r+1}(w,z)\) and \(C_1(w,e^{\gamma(s)})=n e^{-s}\), we obtain the following derivatives:
\begin{equation}\label{eq:cliquesHDerivatives}
	\begin{aligned}
		H'(s) & = e^s C(w, e^{\gamma(s)}), \\
		H''(s) &= e^{s} C(w, e^{\gamma(s)})+n \gamma'(s), \quad \text{ and } \\
		H'''(s) &= e^{s} C(w,e^{\gamma(s)})+n\gamma'(s)+n\gamma''(s).
	\end{aligned}  
\end{equation}
From differentiating both sides of the equation \(C_1(w,e^{\gamma(s)})=n e^{-s}\), we infer
\begin{equation}\label{eq:cliquesGammaDerivatives}
	\begin{aligned}
		\gamma'(s)  & = -\frac{C_1(w,e^{\gamma(s)})}{C_2(w,e^{\gamma(s)})} \quad \text{ and } \\
		\gamma''(s) &=-\gamma'(s) + \gamma'(s) \frac{C_3(w,e^{\gamma(s)})}{C_2(w,e^{\gamma(s)})} = \frac{C_1(w,e^{\gamma(s)})}{C_2(w,e^{\gamma(s)})}-\frac{C_1(w,e^{\gamma(s)})C_3(w,e^{\gamma(s)})}{C_2(w,e^{\gamma(s)})^2}.
	\end{aligned}
\end{equation}
Further, we define \(\tau(s)\) via the equation \(e^{\gamma(s)}=\tau(s)(\nicefrac{1}{w})^{\tau(s)-1/2}\) and write \(\tau=\tau(0)\). Hence, we can imply \Cref{lem:C-asymptotics} and obtain
\[
	C_r(w,e^{\gamma(0)})=C_{r+1}(w,e^{\gamma(0)})\frac{E_{w,r}(\tau)}{\tau E_{w,r+1}(\tau)},
\] 	
where we recall \(E_{w,r}(\tau)=E_{w,r}(\tau,0)\). By \cref{lem:E-asymptotics}, we infer \(E_{w,r}(\tau)\sim E_{w,r+1}(\tau)\). Moreover, \Cref{lem:subcritical-tau-asymp} implies \(\tau\sim \sqrt{2\log (n)/ \log(1/w)}\). We thus infer, 
\begin{equation*}
	\begin{aligned}
		\mu_{n,p} 
			& \sim H'(0) = C(w, e^{\gamma(0)})= \frac{C_1(w,e^{\gamma(0)}) E_{w,0}(\tau)}{\tau E_{w,1}(\tau)}  \sim \frac{n}{\tau}\sim n\sqrt{\frac{\log \nicefrac{1}{w}}{2\log n}},
	\end{aligned}
\end{equation*}
Similarly, 
\begin{equation*}
	\begin{aligned}
		\sigma_{n,p}^2 
			& \sim H''(0) = \frac{n E_{w,0}(\tau)}{\tau E_{w,1}(\tau)} - n\frac{C_1(w,e^{\gamma(0)})}{C_2(w,e^{\gamma(0)})} = \frac{n}{\tau}\big(\frac{E_{w,0}(\tau)}{E_{w,1}(\tau)}-\frac{E_{w,1}(\tau)}{E_{w,2}(\tau)}\big).
	\end{aligned}
\end{equation*}
To derive the order of the difference on the right-hand side, we apply \Cref{lem:E-asymptotics} and infer for \(r\in\mathbb{N}\)
\begin{equation} \label{eq:expandEr}
	E_{w,r}(\tau)= E_{w,0}(\tau) + \sum_{k=0}^2 \tau^{-k}\sum_{\ell=0}^{2k}(a_{k,\ell}(r)-a_{k,\ell}(0))e_{w,\ell}(\tau) + \bigO(\tau^{-3}).
\end{equation}
Using the coefficients given in \Cref{tab:coefficients}, we obtain
\begin{equation*}
	\begin{aligned}
		& a_{0,0}(r)-a_{0,0}(0)=0, \\
		& a_{1,0}(r)-a_{1,0}(0) = 0, \  a_{1,1}(r)-a_{1,1}(0)= r, \  a_{1,2}(r)-a_{1,2}(0) =  0, \\
		& a_{2,0}(r)-a_{2,0}(0) = 0, \  a_{2,1}(r)-a_{2,1}(0)= \tfrac{r}{12}, \  a_{2,2}(r)-a_{2,2}(0) =  \tfrac{r^2}{2}-r, \ a_{2,3}(r)-a_{2,3}(0) = -\tfrac{r}{2}, \\
		& a_{2,4}(r)-a_{2,4}(0) = 0. 
	\end{aligned}
\end{equation*}
Using this, it is straightforward  to deduce
\begin{equation*}
	\begin{aligned}
		\frac{E_{w,0}(\tau)E_{w,2}(\tau)-E_{w,1}(\tau)^2}{E_{w,1}(\tau)E_{w,2}(\tau)} & = \frac{\tau^{-2}\big(e_{w,2}(\tau)E_{w,0}(\tau)-e_{w,1}^2\big)}{E_{w_0}(\tau)^2+\bigO(\tau^{-1})}+\bigO(\tau^{-3}) \\
		& = \tau^{-2}\frac{e_{w,0}(\tau)e_{w,2}(\tau)-e_{w,1}(\tau)^2}{e_{w,0}(\tau)^2+\bigO(\tau^{-1})}+\bigO(\tau^{-3}), 
	\end{aligned}
\end{equation*}
from which we immediately infer
\[
	\sigma_{n,p}^2 \sim H''(0) \sim \frac{n}{\tau^3}\frac{e_{w,0}(\tau)e_{w,2}(\tau)-e_{w,1}(\tau)^2}{e_{w,0}(\tau)^2}.
\]
Finally, we consider the third derivative. It is straightforward to deduce from~\eqref{eq:cliquesHDerivatives} and~\eqref{eq:cliquesGammaDerivatives} that
\begin{equation*}
	\begin{aligned}
		H'''(t) & = e^t C(w,e^{\gamma(t)})+n\gamma'(t) \frac{C_3(w,e^{\gamma(t)})}{C_2(w, e^{\gamma(t)})} = e^t n\frac{E_{w,0}(\tau(t))}{\tau E_{w,1}(\tau(t))}-n\frac{E_{w,1}(\tau(t))E_{w,3}(\tau(t))}{E_{w,2}(\tau(t))}=\bigO(n),
	\end{aligned} 
\end{equation*}
for any \(t\in(-s,s)\). Consider now the sequence \(s_n = s/ \sigma_{n,p}\). Let \(t_n\in(-s_n,s_n)\) be the argument of the third derivative in the Lagrange remainder of the Taylor expansion of \(H(s_n)\) at zero (for a fixed \(n\)). Then \(t_n=\bigO(s_n)\), implying 
\[
	H'''(t_n) t_n^3 = \bigO(n\cdot \tfrac{\tau^{9/2}}{n^{3/2}}) = o(1),
\]
as \(\tau\) only grows logarithmically in \(n\). Combining the above with~\eqref{eq:MGFCliques}, we obtain 
\[
	\lim_{n\to\infty} e^{-\mu_{n,p} s_n}\PGF_{\rv{C}_{n,p}}(e^{s_n}) = e^{s^2/2},
\]
concluding the proof.
\end{proof}

\subsection{Number of edges} \label{sec:subcritical:edges}
In this section, we derive a central limit theorem for the number of edges in a subcritical cluster graph following the same strategy used in the previous section. 
The following proposition is the main result of this section and proves \cref{thm:mainEdges}~(iii). Recall that the umber of edges is denoted by \(\rv{M}_{n,p}\). We further adapt the previously used notation and write \(\E\rv{M}_{n,p}=\mu_{n,p}\) and \(\Var(\rv{M}_{n,p})=\sigma_{n,p}^2\) throughout this section.

\medskip

\begin{proposition}[CLT for the number of edges]\label{prop:CltEdges}
    Let \(p<\nicefrac{1}{2}\). Then
        \begin{equation*}
        \frac{\rv{M}_{n,p}-\mu_{n,p}}{\sigma_{n,p}}\overset{}{\longrightarrow}\mathcal{N}(0,1),
    \end{equation*}
    in distribution as \(n\to\infty\), where 
    \begin{equation*}
        \begin{aligned}
            & \mu_{n,p} \sim n\sqrt{\frac{\log(n)}{2\log(\nicefrac{1}{w})}}, \quad \text{ and } \\
            & \sigma_{n,p}^2 \sim n\frac{\sqrt{\log n}^3}{\sqrt{2}\sqrt{\log\nicefrac{1}{w}}^3}\cdot\frac{e_{w,0}(\tau)e_{w,2}(\tau)-e_{w,1}(\tau)^2}{e_{w,0}(\tau)},
        \end{aligned}
    \end{equation*}
    where the periodic, bounded functions \(e_{w,j}(\tau)\), \(j=0,1,2\), are given in \Cref{lem:E-asymptotics}.
\end{proposition}
\begin{proof}
	We follow the same approach as in the previous proof of \cref{prop:CltCluster}. By~\Cref{th:exact:pgf}, the moment generating function of \(\rv{M}_{n,p}\) is 
	\[
		\PGF_{\rv{M}_{n,p}}(e^s) = \frac{n! [z^n] e^{C(e^s w,z)}}{n! [z^n] e^{C(w,z)}}, \qquad s< \log\nicefrac{1}{w}.
	\]
	Note that the restriction \(s< \log\nicefrac{1}{w}\) guarantees \(e^s w<1\) so that \Cref{th:application:hayman} can still be applied when \(w\) is replaced by \(e^s w\) (where we additionally chose \(k=0, v=1\) and also consider \(e^0=1\) in the exponential).
	Hence, by \Cref{th:application:hayman},
	\begin{equation*}
		\begin{aligned}
			\PGF_{\rv{M}_{n,p}}(e^s) & \sim \big(\tfrac{\log\nicefrac{1}{w}}{\log\nicefrac{1}{w}-s}\big)^{1/4} \exp\big(H(s)-H(0)\big),
		\end{aligned}
	\end{equation*}
	where now \(H(s)=C(e^s w,e^{\gamma(s)})-n\gamma(s)\) and \(\gamma(s)\) is defined implicitly via \(C_1(e^s w, e^{\gamma(s)})=n\). Using Taylor expansion with Lagrange remainder, we obtain
	\begin{equation*}\label{eq:mgfEdges}
		\PGF_{\rv{M}_{n,p}}(e^s) \sim \big(\tfrac{\log\nicefrac{1}{w}}{\log\nicefrac{1}{w}-s}\big)^{1/4} \exp\Big(H'(0) s + H''(0) \tfrac{s^2}{2}+H'''(t)\tfrac{t^3}{6}\Big), \quad \text{ for some } \quad t\in(-s,s).
	\end{equation*}
	Writing now \(\partial_1\) for the partial derivative with respect to the first and \(\partial_2\) for the derivative with respect to the second argument, we have for all \(r\in\mathbb{N}\)
	\[
		\partial_1 C_r(w,z) = \tfrac{1}{2} w^{-1}(C_{r+2}(w,z)-C_{r+1}(w,z) \quad \text{ and } \quad z\partial_2 C_r(w,z) = C_{r+1}(w,z).
	\] 
	This yields for the derivatives of \(H\)
	\begin{equation*}
		\begin{aligned}
			H'(s)
			& = \tfrac{1}{2}\big(C_2(e^s w, e^{\gamma(s)})-C_1(e^s w, e^{\gamma(s)})\big), \\
			H''(s) 
			& = \tfrac{1}{4}\big(C_4(e^s w, e^{\gamma(s)})-2C_3(e^s w, e^{\gamma(s)})+C_2(e^s w, e^{\gamma(s)})\big)+\tfrac{1}{2}\gamma'(s)\big(C_3(e^s w, e^{\gamma(s)})-C_2(e^s w, e^{\gamma(s)})\big)
		\end{aligned}
	\end{equation*}
	as well as
	\begin{equation*}
		\begin{aligned}
			H'''(s) 
			& = \tfrac{1}{8}\big(C_6(e^s w, e^{\gamma(s)})-3C_5(e^s w, e^{\gamma(s)})+3C_4(e^s w, e^{\gamma(s)})-C_3(e^s w, e^{\gamma(s)})\big) \\
			& \ +\tfrac{\gamma'(s)}{4}\big(C_5(e^s w, e^{\gamma(s)})-2C_4(e^s w, e^{\gamma(s)})+C_3(e^s w, e^{\gamma(s)})\big)+\tfrac{(\gamma'(s))^2}{4}\big(C_4(e^s w, e^{\gamma(s)})-C_3(e^s w, e^{\gamma(s)})\big) \\
			& \quad + \tfrac{\gamma''(s)}{4}\big(C_4(e^s w, e^{\gamma(s)})-C_3(e^s w, e^{\gamma(s)})\big).
		\end{aligned}
	\end{equation*}
	We infer from differentiating \(C_1(e^s w,e^{\gamma(s)})=n\) with respect to \(s\) that
	\begin{equation*}
		\begin{aligned}
			\gamma'(s) = - \frac{C_3(e^s w, e^{\gamma(s)})-C_2(e^s w, e^{\gamma(s)})}{2C_2(e^s w, e^{\gamma(s)})}\sim -\frac{\tau(s)}{2},
		\end{aligned}
	\end{equation*}
	where we applied again \Cref{lem:C-asymptotics} and \(\tau(s)\) is the solution of \(e^{\gamma(s)}=\tau(s)(\nicefrac{1}{w})^{\tau-1/2}\). 	\begin{equation*}
		\begin{aligned}
			2 \gamma''(s) & =\frac{C_4(e^s w, e^{\gamma(s)})C_3(e^s w, e^{\gamma(s)})-C_3(e^s w, e^{\gamma(s)})^2+2\gamma'(s)C_3(e^s w, e^{\gamma(s)})^2}{2C_2(e^s w, e^{\gamma(s)})^2} \\
			& \qquad  - \frac{C_5(e^s w, e^{\gamma(s)})-C_4(e^s w, e^{\gamma(s)})+2\gamma'(s)C_4(e^s w, e^{\gamma(s)})}{2C_2(e^s w, e^{\gamma(s)})} \\
			& = \bigO(\tau(s)^3).
		\end{aligned}
	\end{equation*}
	Performing similar calculations as above in~\eqref{eq:expandEr}, we obtain, writing \(\gamma=\gamma(0)\) and \(\tau=\tau(0)\), and using \Cref{lem:C-asymptotics} and \Cref{lem:subcritical-tau-asymp},
	\begin{equation*}
		\begin{aligned}
			\mu_{n,p}& \sim H'(0)=\frac{C_2(w,e^{\gamma})-C_1(w,e^\gamma)}{2}=\frac{n}{2}\Big(\frac{\tau E_{w,2}(\tau)}{E_{w,1}(\tau)}-1\Big)\sim \frac{n \tau}{2} \sim n \sqrt{\frac{\log n}{2\log\nicefrac{1}{w}}} 
		\end{aligned}
	\end{equation*}
	as well as
	\begin{equation*}
		\begin{aligned}
			\sigma_{n,p}^2 & \sim H''(0) = \frac{C_2(w,e^\gamma)C_4(w,e^\gamma)-C_3(w,e^\gamma)^2}{4C_2(w,e^\gamma)} = \frac{n\tau^3}{4}\cdot \frac{E_{w,4}(\tau)E_{w,2}(\tau)-E_{w,3}(\tau)^2}{E_{w,2}(\tau)} \\
			& \sim n\frac{\sqrt{\log n}^3}{\sqrt{2}\sqrt{\log\nicefrac{1}{w}}^3}\cdot\frac{e_{w,0}(\tau)e_{w,2}(\tau)-e_{w,1}(\tau)^2}{e_{w,0}(\tau)}
		\end{aligned}
	\end{equation*}
	Further, by a similar calculation for any \(t\in(-s,s)\)
	\[
		H'''(t) = \bigO(n \tau(t)^6).
	\]
	Combining the above and writing \(s_n=s/\sigma_{n,p}\) yields
	\[
		e^{- \mu_{n,p}s_n}\PGF_{\rv{M}_{n,p}}(e^{s_n}) \sim \big(\tfrac{\log\nicefrac{1}{w}}{\log\nicefrac{1}{w}-s_n}\big)^{1/4}\exp\big(\tfrac{s^2}{2}+\bigO(\tau^{3/2}/n^{1/2})\big) \sim e^{s^2/2}
	\]
	and thus concludes the proof.
\end{proof}

\subsection{Degree distribution}\label{sec:subcritical_degree}
In this section, we derive the limit of the degree of a randomly chosen vertex $\rv{D}_{n,p}$ for $p<\nicefrac{1}{2}$. The main result of this section is a refinement of \cref{thm:mainDegree}~(iii), which is stated in \cref{prop:subcritDegree}. For its proof, we rely on the following lemma in order to deal with the periodicity of the functions $e_{w,r}(\tau,0)$.

\begin{lemma}\label{lem:tau-subseq}
    Let $\tau(n)$ be defined implicitly by
    \[
    C_1(w,\tau(n)\cdot(1/w)^{\tau(n)-\nicefrac{1}{2}})=n.
    \]
    Then, for every $\lambda\in[0,1)$, there exists a subsequence $(\tau(n_k^{(\lambda)}))_{k\in\mathbb{N}}$, such that
    \[
    \tau(n_k^{(\lambda)})-\lfloor\tau(n_k^{(\lambda)})\rfloor\downarrow\lambda.
    \]
\end{lemma}

\begin{proof}
From \cref{lem:subcritical-tau-asymp}, we obtain
\[
    \tau(n) =
    \sqrt{\frac{2 \log n}{\log(1/w)}}
    - \frac{1}{\log(1/w)}
    + o(1),
\]
from which we infer $\tau(n+1)-\tau(n)\rightarrow0$.
Consider $\lambda\in[0,1)$ fixed.
We show that there exists a subsequence $(n^{(\lambda)}_k)_{k\in\mathbb{N}}$
such that the fractional part of $\tau(n^{(\lambda)}_k)$ converges to $\lambda$ from above.
Fix a sequence \(\varepsilon_k \downarrow 0\) with $\varepsilon_0 < 1 - \lambda$
and assume that the first values $(n_\ell^{(\lambda)})_{\ell < k}$
have already been chosen.
There exists $N > n^{(\lambda)}_{k-1}$ large enough such that
$|\tau(n+1)-\tau(n)| < \varepsilon_k$ holds for all $n\geq N$.
If $\tau(N) - \lfloor \tau(N) \rfloor \in (\lambda, \lambda + \varepsilon_k)$,
we set $n_k^{(\lambda)} = N$.
Otherwise, we choose \(n_k^{(\lambda)}>N\) to be the  
smallest 
integer such that
\[
    \tau(n_k^{(\lambda)}) >
    \lfloor \tau(N) \rfloor + 1 + \lambda,
\]
which always exist as $\tau(n)$ increases to infinity. We deduce from
$\tau(n_k^{(\lambda)} - 1) \leq \lfloor \tau(N) \rfloor + 1 + \lambda$
and $|\tau(n_k^{(\lambda)}) - \tau(n_k^{(\lambda)} - 1)| < \varepsilon_k$
that
\[
    \tau(n_k^{(\lambda)})
    \in
    \big[
        \lfloor \tau(N) \rfloor + 1 + \lambda,\ 
        \lfloor \tau(N) \rfloor + 1 + \lambda + \varepsilon_k
    \big].
\]
Since $\lambda + \epsilon_k < 1$, this implies
$\lfloor \tau(n_k^{(\lambda)}) \rfloor = \lfloor \tau(N) \rfloor + 1$
and
$\tau(n_k^{(\lambda)}) - \lfloor\tau(n_k^{(\lambda)}) \rfloor \in [\lambda,\lambda + \varepsilon_k]$.
As $\varepsilon_k \downarrow 0$, we obtain
$\tau(n_k^{(\lambda)}) - \lfloor \tau(n_k^{(\lambda)}) \rfloor \downarrow \lambda$.
\end{proof}

Recall that the degree of a uniformly chosen vertex in \(\rv{CG}_{n,p}\) is denoted by \(\rv{D}_{n,p}\).

\begin{proposition}[Asymptotic degree distribution]\label{prop:subcritDegree}
	Let \(p<\nicefrac{1}{2}\) and $w=\nicefrac{p}{1-p}$. Define $\tau:=\tau(n)$ implicitly by $C_1(w,\tau(1/w)^{\tau-\nicefrac{1}{2}})=n$. Then, as $n\rightarrow\infty$, \begin{align*}
    \E[\rv{D}_{n,p}]&= \tau-1+\frac{e_{w,1}(\tau,0)}{e_{w,0}(\tau,0)}+o(1),\\
    \operatorname{Var}(\rv{D}_{n,p})&= \frac{e_{w,2}\left(\tau,0\right)}{e_{w,0}\left(\tau,0\right)}-\left(\frac{e_{w,1}\left(\tau,0\right)}{e_{w,0}\left(\tau,0\right)}\right)^2+o(1).\\
\end{align*}
Moreover, for $\lambda\in[0,1)$ and the subsequence $n_k^{(\lambda)}$ as given in~\cref{lem:subcritical-tau-asymp}, we have as $k\rightarrow\infty$,
\[\rv{D}_{n_k^{(\lambda)},p}-\lfloor\tau(n_k^{(\lambda)})-1\rfloor\stackrel{\mathcal{D}}{\longrightarrow}X_\lambda,\]
where $X_\lambda$ is as random variable with distribution given by
\[
\proba(X_\lambda=d)=\frac{w^{(d-\lambda)^2/2}}{e_{w,0}(\lambda,0)},\quad\text{for }d\in\mathbb{Z}.
\]
In particular, all central moments of $\rv{D}_{n,p}$ are bounded.
\end{proposition}
\begin{proof}
We obtain from~\eqref{eq:degree-pgf-exact} that the PGF of $\rv{D}_{n,p}$ is given by
\[
\PGF_{\rv{D}_{n,p}}(u) =
    \frac{[z^n] C_1(w, u\, z) e^{C(w, z)}}{u [z^n] C_1(w, z) e^{C(w, z)}}.
\]
We use ~\cref{th:hayman} to write
\[
[z^n]C_1(w,uz)e^{C(w,z)}\sim\sqrt{\frac{\log(1/w)}{2\pi n\gamma}}C_1(w,ue^{\gamma}) e^{C(w,e^{\gamma})-n\gamma},
\]
where $\gamma$ is defined implicitly by $C_1(w,e^\gamma)=n$.
Hence, the asymptotics of the PGF are
\[
\PGF_{\rv{D}_{n,p}}(u)\sim\frac{1}{u}\frac{C_1(w,ue^{\gamma})}{C_1(w,e^{\gamma})}.
\]
After substituting $u=e^{is}$, and applying
$$
C_1(w,e^{\gamma+is})=w^{-\tau^2/2}e^{(1+is)\tau}\sqrt{\tau}\frac{E_{w,1}(\tau,s)}{\sqrt{2\pi}},
$$
from~\cref{lem:C-asymptotics}, where $\tau=\tau(s)$ is related to $\gamma$ by $e^\gamma=\tau\cdot(1/w)^{\tau-\nicefrac{1}{2}}$, we obtain
\begin{equation}\label{eq:subcritical-degree-cf}
\CF_{\rv{D}_{n,p}}(s)\sim e^{i(\tau-1)}\frac{E_{w,1}(\tau,s)}{E_{w,1}(\tau,0)}\sim e^{i(\tau-1)}\frac{e_{w,0}(\tau,s)}{e_{w,0}(\tau,0)}.
\end{equation}

To compute the mean and variance of $\rv{D}_{n,p}$, note that the characteristic function of $\rv{D}_{n,p}-(\tau-1)$ is asymptotically equivalent to
\[
\frac{e_{w,0}(\tau,s)}{e_{w,0}(\tau,0)}.
\]
The central moments of $\rv{D}_{n,p}$ are equal to the central moments of $\rv{D}_{n,p}-(\tau-1)$. The regular moments of the latter are given by
\begin{align*}
\E[(\rv{D}_{n,p}-(\tau-1))^r]
&=\frac{d^r}{i^r\cdot ds^r}\frac{e_{w,0}(\tau,s)}{e_{w,0}(\tau,0)}\biggr\rvert_{s=0}+o(1)=\frac{e_{w,r}(\tau,0)}{e_{w,0}(\tau,0)}+o(1),
\end{align*}
where we used $\tfrac{d}{ds}e_{w,r}(\tau,s)=i\cdot e_{w,r+1}(\tau,s)$. In particular, these moments are bounded for all $\tau$, so that the central moments of $\rv{D}_{n,p}-(\tau-1)$ (and hence $\rv{D}_{n,p}$) are bounded as well. We use the first two moments of $\rv{D}_{n,p}-(\tau-1)$ to calculate the mean and variance of $\rv{D}_{n,p}$ and derive
\[
\E[\rv{D}_{n,p}]=(\tau-1)+\E[\rv{D}_{n,p}-(\tau-1)]=(\tau-1)+\frac{e_{w,1}(\tau,0)}{e_{w,0}(\tau,0)}+o(1),
\]
and
\[
\Var(\rv{D}_{n,p})=\Var(\rv{D}_{n,p}-(\tau-1))=\frac{e_{w,2}(\tau,0)}{e_{w,0}(\tau,0)}-\left(\frac{e_{w,1}(\tau,0)}{e_{w,0}(\tau,0)}\right)^2+o(1).
\]

In order to prove the desired convergence in distribution, consider $\lambda\in[0,1)$ and let  $\tau(n_k^{(\lambda)})$ be a subsequence for which $\tau(n_k^{(\lambda)})-\lfloor\tau(n_k^{(\lambda)})\rfloor\rightarrow\lambda$ holds (which exists due to~\cref{lem:tau-subseq}). Then, using~\eqref{eq:subcritical-degree-cf}, we obtain for the characteristic function of $\rv{D}_{n_k^{(\lambda)},p}-\lfloor\tau(n_k^{(\lambda)})-1\rfloor$,
\begin{align*}
\CF_{\rv{D}_{n,p}}(s)e^{-i\lfloor\tau(n_k^{(\lambda)})-1\rfloor}
&\longrightarrow e^{i\lambda s}\frac{e_{w,0}(\lambda,s)}{e_{w,0}(\lambda,0)}
=\frac{\sum_{\lambda+t\in\mathbb{Z}}w^{t^2/2}e^{i(\lambda+t)s}}{e_{w,0}(\lambda,0)}
=\frac{\sum_{d\in\mathbb{Z}}w^{(d-\lambda)^2/2}e^{ids}}{e_{w,0}(\lambda,0)}
=\mathbb{E}\left[e^{iX_\lambda s}\right].
\end{align*}
This concludes the proof. 
\end{proof}

\section{The supercritical regime  ($p>\nicefrac{1}{2}$)} \label{sec:supercritical}

In this section, we consider $\rv{CG}_{n,p}$ for $p > \nicefrac{1}{2}$ fixed.
Our main result, \cref{th:proba:clique:asymptotic:expansion},
shows that with high probability,
the graph contains only one clique, i.e., $\rv{C}_{n,p}\stackrel{\proba}{\rightarrow}1$. This proves~\cref{thm:mainClust}~(i), \cref{thm:mainEdges}~(i) and \cref{thm:mainDegree}~(i).
In addition, \cref{th:proba:clique:asymptotic:expansion} provides an asymptotic expansion of the probability $\P(\rv{C}_{n,p}=K_n)$.

The choice $p > 1/2$
implies that $w = \frac{p}{1-p}$ is greater than $1$,
so that the series
\[
    C(w,z) =
    \sum_{n \geq 1}
    w^{\binom{n}{2}}
    \frac{z^n}{n!}
\]
has a zero radius of convergence.
In this context, tools based on complex analysis
are difficult to apply
(though not impossible, as illustrated by \cite{FSS04}).
We will instead rely on tools dedicated to divergent series.
We state a theorem derived by Dovgal and Nurligareev~\cite{dovgal2023asymptotics}.
For general asymptotical techniques for extracting coefficients from divergent series, we refer to \cite{Be75}. Techniques for factorially diverging series can be found in \cite{Bo16}.

\begin{theorem}[\protect{Dovgal and Nurligareev~\cite[proof of Proposition 3.4 p.17]{dovgal2023asymptotics}}]
\label{th:divergent:series}
Consider a series $A(z)$
and assume that there exist $\alpha$, $\beta$
and coefficients $(a^{\circ}_{m,\ell})_{m,\ell}$
such that for any $n$ and $R$,
there exists $M(R)$ satisfying the expansion
\[
    n! [z^n] A(z) =
    \alpha^{\beta \binom{n}{2}}
    \bigg(
    \sum_{m=0}^{R-1}
    \alpha^{-m n}
    \sum_{\ell=0}^{M(R) - 1}
    n (n-1) \cdots (n-\ell+1)
    a^{\circ}_{m,\ell}
    + \bigO \left( \alpha^{-R n} n^{M(R)} \right)
    \bigg).
\]
Consider a function $F(z)$ analytic at $0$,
then there exists a sequence
$(\eta^{\circ}_{m,\ell})_{m,\ell}$
and a function $M_0(R)$
such that for any $n$ and $R$,
the following expansion holds,
\[
    n! [z^n] F(A(z)) =
    \alpha^{\beta \binom{n}{2}}
    \bigg(
    \sum_{m=0}^{R-1}
    \alpha^{-m n}
    \sum_{\ell=0}^{M_0(R) - 1}
    n (n-1) \cdots (n-\ell+1)
    \eta^{\circ}_{m,\ell}
    + \bigO \left( \alpha^{-R n} n^{M_0(R)} \right)
    \bigg).
\]
Moreover, the coefficients $(\eta^{\circ}_{m,\ell})_{m,\ell}$
are equal to
\[
    \eta^{\circ}_{m,\ell} =
    \sum_{k=0}^{\ell}
    \alpha^{k m - \beta \binom{k}{2}}
    a^{\circ}_{m - \beta k, \ell - k}
    [z^k] F'(A(z)).
\]
\end{theorem}

\begin{proposition}
\label{th:proba:clique:asymptotic:expansion}
Consider a fixed $p \in (1/2,1)$
and $w = \frac{p}{1-p}$.
For any $\ell \geq 0$,
let $P_{\ell}(y)$ denote the polynomial
\[
    P_{\ell}(y) =
    [x^{\ell}]
    \bigg(
    1 +
    \sum_{m \geq 1}
    x^m
    y (y-1) \cdots (y-m+1)
    w^{\binom{m+1}{2}}
    [z^m] e^{C(w,z)}
    \bigg)^{-1}.
\]
For any $R \in\mathbb{N}$,
the probability of $\rv{CG}_{n,p}$
being the complete graph is
\[
    \proba(\rv{CG}_{n,p} = K_n) =
    1 +
    \sum_{m=1}^{R-1}
    w^{-m n}
    P_m(n)
    + \bigO \left( w^{-R n} n^R \right).
\]
The first values of the polynomials $P_m$ are 
$$
\begin{array}{ll}
P_1   =  - wn, \\
P_2 = \frac{w^2}{2} \left((w^2+w)n + (2-w-w^2)n^2 \right).
\end{array}
$$
\end{proposition}
Note that $P_1$ is negative, so that the expansion of the probability does not exceed $1$.

\begin{proof}
From~\cref{lem:CG:pmf}, we obtain
\begin{equation} \label{eq:PGnpKn}
    \proba(\rv{CG}_{n,p} = K_n) =
    \frac{w^{\binom{n}{2}}}
    {n! [z^n] e^{C(w,z)}}.
\end{equation}
To estimate the denominator,
\cref{th:divergent:series} is applied with
\begin{align*}
    A(z) &= 1 + C(w,z),
    \\
    \alpha &= w = \frac{p}{1-p},
    \\
    \beta &= 1,
    \\
    a^{\circ}_{m,\ell} &=
        \begin{cases}
        1 & \text{if } m = \ell = 0,\\
        0 & \text{otherwise,}
        \end{cases}
    \\
    F(z) &= e^{z-1}.
\end{align*}
Hence, there exists $M_0(R)$ such that
for all $n$ and $R$,
\[
    n! [z^n] e^{C(w,z)} =
    w^{\binom{n}{2}}
    \left(
    \sum_{m=0}^{R-1}
    w^{-m n}
    \sum_{\ell=0}^{M_0(R) - 1}
    n (n-1) \cdots (n-\ell+1)
    \eta^{\circ}_{m,\ell}
    + \bigO \left( w^{-R n} n^{M_0(R)} \right)
    \right),
\]
with
\[
    \eta^{\circ}_{m, \ell} =
    \begin{cases}
        w^{\binom{m+1}{2}}
        [z^m] e^{C(w,z)}
        & \text{if } \ell = m,
        \\
        0 & \text{otherwise.}
    \end{cases}
\]
We choose $M_0(R) = R$ as
for all $m \leq R-1$ and all $\ell \geq R$,
we have $\eta^{\circ}_{m,\ell} = 0$,
and obtain
\begin{equation} \label{eq:expA:expansion}
    n! [z^n] e^{C(w,z)} =
    w^{\binom{n}{2}}
    \bigg(
    \sum_{m=0}^{R-1}
    w^{-m n}
    n (n-1) \cdots (n-m+1)
    w^{\binom{m+1}{2}}
    [z^m] e^{C(w,z)}
    + \bigO \left( w^{-R n} n^R \right)
    \bigg).
\end{equation}
This expansion is injected in \cref{eq:PGnpKn}
\[
    \proba(\rv{CG}_{n,p} = K_n) =
    \bigg(
    \sum_{m=0}^{R-1}
    w^{-m n}
    n (n-1) \cdots (n-m+1)
    w^{\binom{m+1}{2}}
    [z^m] e^{C(w,z)}
    + \bigO \left( w^{-R n} n^R \right)
    \bigg)^{-1}.
\]
Let $P_{\ell}(y)$ denote the polynomial of degree $\ell$
defined in the theorem.
Observe that the term corresponding to $m=0$ is $1$,
so that the inverse is well defined
from a formal power perspective.
Then
\[
    \proba(\rv{CG}_{n,p} = K_n) =
    \sum_{m=0}^{R-1}
    w^{-m n}
    P_m(n)
    + \bigO \left( w^{-R n} n^R \right).
\]
In particular, we have $P_0(y) = 1$,
so this probability tends to $1$
for any $w > 1$ and hence particularly
for any fixed $p \in (1/2,1)$.
\end{proof}

\cref{th:proba:clique:asymptotic:expansion} tells us that $\P(\rv{CG}_{n,p}=K_n)=1-nw^{1-n}+\bigO(n^2w^{-2n})$. At first glance, one might hope to use this result to locate the phase transition more precisely by solving $nw^{1-n}=\nicefrac{1}{2}$.
However, the resulting $w_n=\left(\tfrac{1}{2n}\right)^{1/(1-n)}$ yields an error of the order $n^2w_n^{-2n}=\bigO(1)$, so that \cref{th:proba:clique:asymptotic:expansion} is not strong enough to tell us whether this sequence is subcritical or supercritical. In the next section, we provide tools to better locate the phase transition.

\section{The near-critical regime (\(p\downarrow \nicefrac{1}{2}\))}\label{sec:window}
The goal of this section is to identify the critical window around the established phase transition at \(p=\nicefrac{1}{2}\) and to study the near-critical behaviour from above and ultimately prove \Cref{thm:mainWindow}. For some $q\in(0,1)$, let $p_n(q)$ be the sequence that satisfies $\proba(\rv{C}_{n,p_n(q)}=1)=q$. Recall that the function \(p\mapsto \P(\rv{C}_{n,p}=1)\) is continuous and increasing in \(p\) for each fixed \(n\), which guarantees the existence of such a sequence. Further, for each \(q\) and \(p\leq \nicefrac{1}{2}\) there exists \(n_0\) with \(\P(\rv{C}_{n,p}=1)<q\) for all \(n\geq n_0\) by Theorem~\ref{thm:mainClust}~(ii) and~(iii). On the contrary, for all \(p>\nicefrac{1}{2}\) and large enough \(n\), we have \(\P(\rv{C}_{n,p}=1)>q\) by Theorem~\ref{thm:mainClust}~(i). Our goal is to derive the leading order term of $p_n(q)-\nicefrac{1}{2}$. Further, we investigate the structure of $\rv{CG}_{n,p_n(q)}$ when it is \emph{not} the complete graph.

Let us write $t=\log p-\log(1-p)$ in the following. Further, we denote by  $p(t)=(1+e^{-t})^{-1}$ the inverse of this relation, which plays an important role to describe the bounds on \(p_n(q)\).
We consider $t_n:=\log p_n(q)-\log (1-p_n(q))$ and derive a lower and an upper bound for it in order to prove \Cref{thm:mainWindow}. These bounds are given in \cref{lem:critical-lower} and \cref{lem:critical-upper}.

\begin{lemma}\label{lem:critical-lower}
    The sequence $t_n=\log p_n(q)-\log (1-p_n(q))$ satisfies the lower bound:
    \[
    t_n\geq\frac{\log B_n+\log q}{{n\choose 2}}=2\frac{\log n-\log\log n-1+o(1)}{n}.
    \]
\end{lemma}
\begin{proof}
    Since $t_n>0$ and $B_n(e^t)$ is monotonously increasing, we have $B_n(e^{t_n})\geq B_n(1)=B_n$, implying
    \begin{align*}
        q&=\proba(\rv{C}_{n,p_n(q)}=1)=\frac{e^{{n\choose 2}t_n}}{B_n(e^{t_n})}\leq\frac{e^{{n\choose 2}t_n}}{B_n}.
    \end{align*}
    Rearranging terms yields
    \[
        t_n\geq \frac{\log B_n+\log q}{{n\choose 2}},
    \]
    and we infer the asymptotic order by substituting
    \[
    \frac{\log B_n}{n}=\log n-\log\log n - 1 +o(1),
    \]
    as proven in~\cite{bruijn1981asymptotic}.
\end{proof}

Proving an upper bound for the critical sequence is considerably more challenging. Note that for $t\leq 0$, we have $B_n(e^{t})\leq n!$, while for constant $t>0$, we have $B_n(e^{t})/n!\rightarrow\infty$. We show in the following lemma that the sequence $t_n'$ defined by $B_n(e^{t_n'})= n!$ provides the desired upper bound.% Recall the function $p(t)=(1+e^{-t})^{-1}$.

\begin{lemma}\label{lem:critical-upper}
    Let $t_n'$ be the sequence defined by $B_n(e^{t_n'})=n!$ for every $n$. Then as \(n\to\infty\), we have $\rv{C}_{n,p(t_n')}\xrightarrow{\proba}1$ and
    \[
        t_n'\leq\frac{\log n!}{{n\choose 2}}=2\frac{\log n-1}{n}+\bigO\left(\frac{\log n}{n^2}\right).
    \]
\end{lemma}
\begin{proof}
    Note that the event $\{\rv{C}_{n,p}=1\}$ is equivalent to $\{\rv{S}_{n,p}=n\}$. 
    The bound $t_n'\leq {\log n!}/{\binom{n}{2}}$ follows immediately from 
    \[
    1\geq\proba(\rv{S}_{n,p(t_n')}=n)=\frac{e^{{n\choose 2}t'_n}}{B_n(e^{t_n'})}=\frac{e^{{n\choose 2}t'_n}}{n!}.
    \]
    The asymptotics are obtained by the Stirling approximation $\log n!=n\log n -n +\bigO(\log n)$.
    
    In the remainder, we will prove
    \[
    \proba(\rv{S}_{n,p(t'_n)}\neq n)=\sum_{s=1}^{n-1}\proba(\rv{S}_{n,p(t'_n)}=s)\rightarrow 0.
    \]
    We split this sum into three parts and prove that each of them vanishes. 
    Let $\varepsilon\in(0,1)$.
    \paragraph{Case $s>\frac{1+\varepsilon}{2}n$.}
    Let us write $s=n-r$. Our aim is to make use of the bound
    \begin{equation}\label{eq:factorial-bound-upper1}
        \proba(\rv{S}_{n,p(t_n')}=n-r)<\frac{B_r}{r!}\left(\frac{n}{n!^{(n-r)/{n\choose 2}}}\right)^r.
    \end{equation}
    To derive this bound, we note that $\rv{S}_{n,p(t_n')}=1+\rv{D}_{n,p(t_n')}$, so that  by~\cref{cor:exact:degree-pmf}, we have
    \[
    \proba(\rv{S}_{n,p(t_n')}=n-r)=\proba(\rv{D}_{n,p(t_n')}=n-r-1)={n-1\choose r}\frac{B_r(e^{t'_n})}{B_n(e^{t_n'})}e^{t_n'{n-r\choose 2}}.
    \]
    Since ${n-1\choose r}<n^r/r!$, and
    \[
    B_r(e^{t'_n})=\sum_{G\in \mathcal{CG}_r} e^{t'_nm(G)}< |\mathcal{CG}_r|\cdot e^{t'_n{r\choose 2}}=B_r\cdot e^{t'_n{r\choose 2}},
    \]
    as well as $B_n(e^{t_n'})=n!$, we find
    \[
        {n-1\choose r}\frac{B_r(e^{t'_n})}{B_n(e^{t_n'})}e^{t_n'{n-r\choose 2}}<\frac{n^rB_r}{r!n!}e^{\left( 
{r\choose 2}+{n-r\choose 2} \right)t_n'}.
    \]
    To obtain~\eqref{eq:factorial-bound-upper1}, we substitute $t_n'\leq \frac{\log n!}{{n\choose 2}}$ and infer
    \[
        \frac{n^rB_r}{r!n!}e^{\left( {r\choose 2}+{n-r\choose 2} \right)t_n'}
        \leq
        \frac{n^rB_r}{r!}n!^{\left( {r\choose 2}+{n-r\choose 2} \right)/{n\choose 2}-1}
        =\frac{B_r}{r!}\frac{n^r}{n!^{\frac{r(n-r)}{{n\choose 2}}}}
        =\frac{B_r}{r!}\left(\frac{n}{n!^{(n-r)/{n\choose 2}}}\right)^r.
    \]

    We use this upper bound by noting that $n-r>\frac{1+\varepsilon}{2}n>\frac{1+\varepsilon}{2}(n-1)$, so that
    \[
    \frac{n}{n!^{(n-r)/{n\choose 2}}}<\frac{n}{n!^{\frac{1+\varepsilon}{n}}}=:z_n.
    \]
    Note that $n!^{1/n}\sim n/e$ and hence $z_n\sim e^{1+\varepsilon}n^{-\varepsilon}\to 0$. Combined, we infer
    \[
        \sum_{1\leq r<\frac{1-\varepsilon}{2}n}\mathbb{P}(\rv{S}_{n,p(t_n')}=n-r)<\sum_{1\leq r<\frac{1-\varepsilon}{2}n}\frac{B_r}{r!}z_n^r<\sum_{r=1}^\infty \frac{B_r}{r!}z_n^r=e^{e^{z_n}-1}-1\longrightarrow 0,
    \]
    as \(n\to\infty\), where we uses the fact that the exponential generating function of the Bell sequence equals
    \[
    \sum_{r=0}^\infty B_r\frac{x^r}{r!}=e^{e^{x}-1}.
    \]

    \paragraph{Case $s\leq \log n$.} For the following two cases, we use the different bound
    \begin{equation}\label{eq:factorial-bound-upper2}
        \proba(\rv{S}_{n,p(t'_n)}=s)\leq\frac{n!^{{s\choose 2}/{n\choose 2}}}{n(s-1)!}.
    \end{equation}
    In order to prove this bound, we use the that $t_n'$ is a decreasing sequence while $B_n(\cdot)$ is monotonously increasing, implying $B_{n-s}(e^{t'_n})\leq B_{n-s}(e^{t'_{n-s}})=(n-s)!$. This yields
    \[
        \proba(\rv{S}_{n,p(t_n')}=s)
        ={n-1\choose s-1}\frac{B_{n-s}(e^{t'_n})}{B_n(e^{t_n'})}e^{t_n'{s\choose 2}}
        \leq \frac{(n-1)!}{(n-s)!(s-1)!}\frac{(n-s)!}{n!}e^{{s\choose 2}t'_n}=\frac{1}{n(s-1)!}e^{{s\choose 2}t'_n}.
    \]
    We again obtain~\eqref{eq:factorial-bound-upper2} by substituting $t'_n\leq{\log n!}/{{n\choose 2}}$. Furthermore, for $s\leq\log n$, this bound is $\bigO(n^{-1})$, and therefore
    \[
        \lim_{n\to\infty}\sum_{1\leq s<\log n}\mathbb{P}(\rv{S}_{n,p(t_n')}=s)=0.
    \]
    \paragraph{Case $\log n <s\leq\frac{1+\varepsilon}{2}n$.} We further bound~\eqref{eq:factorial-bound-upper2} by using ${s\choose 2}<s^2/2$, $n!<n^n$ and $s!>(s/e)^s$ and derive
    \[
    \frac{n!^{{s\choose 2}/{n\choose 2}}}{n(s-1)!}<\frac{s}{n}\frac{(n^n)^{\frac{s^2}{n(n-1)}}}{(s/e)^s}=\frac{s}{n}\exp\left( \frac{\log n}{n-1}s^2-s\log s \right)=:\frac{s}{n}e^{\phi_n(s)}.
    \]
    We now inspect the function $\phi_n(s)$. We show that for $n>e^e$, it is first decreasing and then increasing on the interval $s\in(\log n,\tfrac{1+\varepsilon}{2}n)$. We compute the derivative to find the stationary points
    \[
    \phi'_n(s)=2s\frac{\log n}{n-1}-\log s=0\quad\Longrightarrow\quad \frac{\log s}{s}=2\frac{\log n}{n-1}.
    \]
    Note that $\frac{\log s}{s}$ is monotonously decreasing for $s>e$, so that $n>e^e$ and $s>\log n$ imply that there is at most one stationary point of $\phi_n(s)$ in this interval.
    Hence, $\phi_n(s)$ is either first decreasing and then increasing or monotone and it is therefore bounded by the maximum of the two endpoints. Plugging the two endpoints in, we obtain
    \[
    \phi_n(\log n)=\frac{(\log n)^3}{n-1}-\log n \log\log n+\log n\sim -\log n \log\log n,
    \]
    as well as 
        \[
    \phi_n\left(\frac{1+\varepsilon}{2}n\right)=\left(\frac{1+\varepsilon}{2}\right)^2\frac{n^2\log n}{n-1}-\frac{1+\varepsilon}{2}n\log\left(\frac{1+\varepsilon}{2}n\right)+\frac{1+\varepsilon}{2}n\sim -\frac{1-\varepsilon^2}{4}n\log n
    \]
    both implying $e^{\phi_n(s)}=n^{-\Omega(\log\log n)}$. We therefore finally infer
    \begin{equation*}
    	\begin{aligned}
    	\sum_{\log n<s<\frac{1+\varepsilon}{2}n}\proba(\rv{S}_{n,p(t'_n)}=s)&\leq \sum_{\log n<s<\frac{1+\varepsilon}{2}n}\frac{s}{n}e^{\phi_n(s)}\leq \max\left\{e^{\phi_n(\log n)},e^{\phi_n\left(\frac{1+\varepsilon}{2}n\right)}\right\}\sum_{\log n<s<\frac{1+\varepsilon}{2}n}\frac{s}{n} \\
    	&\leq \sum_{\log n<s<\frac{1+\varepsilon}{2}n}\frac{s}{n^{1+\Omega(\log\log n)}}\longrightarrow 0.
    	\end{aligned}
    \end{equation*}
   This concludes the proof.
\end{proof}
\begin{proof}[Proof of \Cref{thm:mainWindow}.] 
Define
\begin{equation}\label{eq:critical-bounds}
	p^{L}_n=p\left(2\frac{\log n-\log\log n-1}{n}\right),\quad\text{ and }\quad p^{U}_n=p\left(2\frac{\log n-1}{n}\right)
\end{equation}
Since \(p_n(q)=p(t_n)\), we have by \(p^L_n\leq p_n(q)\leq p^U_n\) by applying the Lemmas~\ref{lem:critical-lower} and~\ref{lem:critical-upper} and the fact that \(p(t)\) is a monotone function. The proof finishes with the observation that both bounds are of the desired order.
\end{proof}

\paragraph{The structure of \(\rv{CG}_{n,p_n(q)}\).} 
We now turn to the the graph's structure whenever it is not the complete graph (which happens with probability \(1-q\)). One may wonder at first whether there exists a clique of size $n-1$ and a single isolated vertex as the first error term in \Cref{th:proba:clique:asymptotic:expansion}, which is of order $ne^{-(n-1)t}$, corresponds to the $n$ different ways of isolating a single vertex. We show however in the our next result that this is not the case. 

\begin{proposition}\label{prop:critical-almost-complete}
    For any sequence $p_n\in[0,1]$, we have
    \[
       \lim_{n\to\infty} \proba(\rv{S}_{n,p_n}=n-1)= 0.
    \]
\end{proposition}
\begin{proof}
    We prove the claim by contradiction and suppose there exists a sequence $p_n$ and an $\varepsilon\in(0,1)$ such that
    \begin{equation}\label{eq:criticalContra}
    	\limsup_{n\rightarrow\infty}\proba(\rv{S}_{n,p_n}=n-1)>\varepsilon.
    \end{equation}
    Then there must exists a subsequence \(n_j\to\infty\) with $\proba(\rv{S}_{n_j,p_{n_j}}=n_j-1)\geq\varepsilon$ for all \(j\). Let $t_j=\log{p_{n_j}}-\log{(1-p_{n_j})}$ as before. We then have
    \begin{equation*}
    	\begin{aligned}
        	\proba(\rv{S}_{n_j,p_{n_j}}=n_j)
        		&=\proba(\rv{S}_{n_j,p_{n_j}}=n_j-1)\cdot\frac{\proba(\rv{S}_{n_j,p_{n_j}}=n_j)}{\proba(\rv{S}_{n_j,p_{n_j}}=n_j-1)}=\proba(\rv{S}_{n_j,p_{n_j}}=n_j-1)\cdot\frac{e^{{n_j\choose 2}t_j}}{(n_j-1)e^{{n_j-1\choose 2}t_j}}\\
        		&=\proba(\rv{S}_{n_j,p_{n_j}}=n_j-1)\cdot\frac{e^{(n_j-1)t_j}}{n_j-1}.
    	\end{aligned}
    \end{equation*}
    Using the bound $\proba(\rv{S}_{n_j,p_{n_j}}=n_j)\leq 1-\proba(\rv{S}_{n_j,p_{n_j}}=n_j-1)$, we get
    \[
        \frac{1-\proba(\rv{S}_{n_j,p_{n_j}}=n_j-1)}{\proba(\rv{S}_{n_j,p_{n_j}}=n_j-1)}\geq \frac{e^{(n_j-1)t_j}}{n_j-1}\quad\Longrightarrow\quad 
        t_j\leq \frac{1}{n_j-1}\left( \log(n_j-1)+\log\frac{1-\proba(\rv{S}_{n_j,p_{n_j}}=n_j-1)}{\proba(\rv{S}_{n_j,p_{n_j}}=n_j-1)} \right).
    \]
    By our assumption~\eqref{eq:criticalContra}, we have $\proba(\rv{S}_{n_j,p_{n_j}}=n_j-1)>\varepsilon$ for all \(j\) and therefore
    \[
        t_j\leq \frac{1}{n_j-1}\left( \log(n_j-1)+\log\frac{1-\varepsilon}{\varepsilon} \right)=\frac{\log n_j}{n_j}+\bigO(n_j^{-1}).
    \]
    Now recall $\log B_n\sim n\log n$. Hence, we infer for all sufficiently large \(n_j\)
    \begin{equation*}
    	\begin{aligned}
        	\proba(\rv{S}_{n,p_n}=n-1)
        	&=\frac{(n-1)e^{{n-1\choose 2}t_n}}{B_n(e^{t_n})}\leq \frac{n-1}{B_n}e^{{n-1\choose 2}t_n}\leq \frac{n-1}{B_n}e^{\frac{1}{2}n\log n+\bigO(n)}<\varepsilon,
    	\end{aligned}
    \end{equation*}
    contradicting~\eqref{eq:criticalContra} and thus concluding the proof.
\end{proof}

The previous theorem states that $\{\rv{S}_{n,p_n}=n-1\}$ does not occur with high probability in any regime and therefore particularly does not occur with high probability inside critical window. In order to illustrate the structure of $\rv{CG}_{n,p_n(q)}$, we compute the expected number of cliques of each size via the formula 
\[\E\left[C^{(s)}_{n,p(t)}\right]=n\P(\rv{S}_{n,p(t)}=s)={n\choose s}e^{{s\choose 2}t}\frac{B_{n-s}(e^t)}{B_n(e^{t})}.\] 
The results for $n=100$ and $n=500$ for \(p_n(\nicefrac{1}{2})\) are shown in~\cref{fig:ncliques100} and~\cref{fig:ncliques500} respectively. We see that the size distribution of the cliques is bimodal, with one mode at $n$ and the other one close to zero. 
Despite the fact that $\proba(\rv{C}_n=1)=\tfrac{1}{2}$, the expected number of cliques $\E[\rv{C}_n]$ still appears to grow with $n$.
We only proved the absence of a cluster of size \(n-1\) with high probability, however we believe that this is true for all linear sized components other then \(K_n\). % We hence believe \(\P(\rv{C}_{n,p_n}^\text{max}=o(n)\mid \rv{C}_{n,p_n}\neq 1)\to 1\) where \(\rv{C}_{n,p_n}^\text{max}\) denotes the largest cluster in our graph.  

\begin{figure}
\centering
\begin{subfigure}{.5\textwidth}
  \centering
  \includegraphics[width=\linewidth]{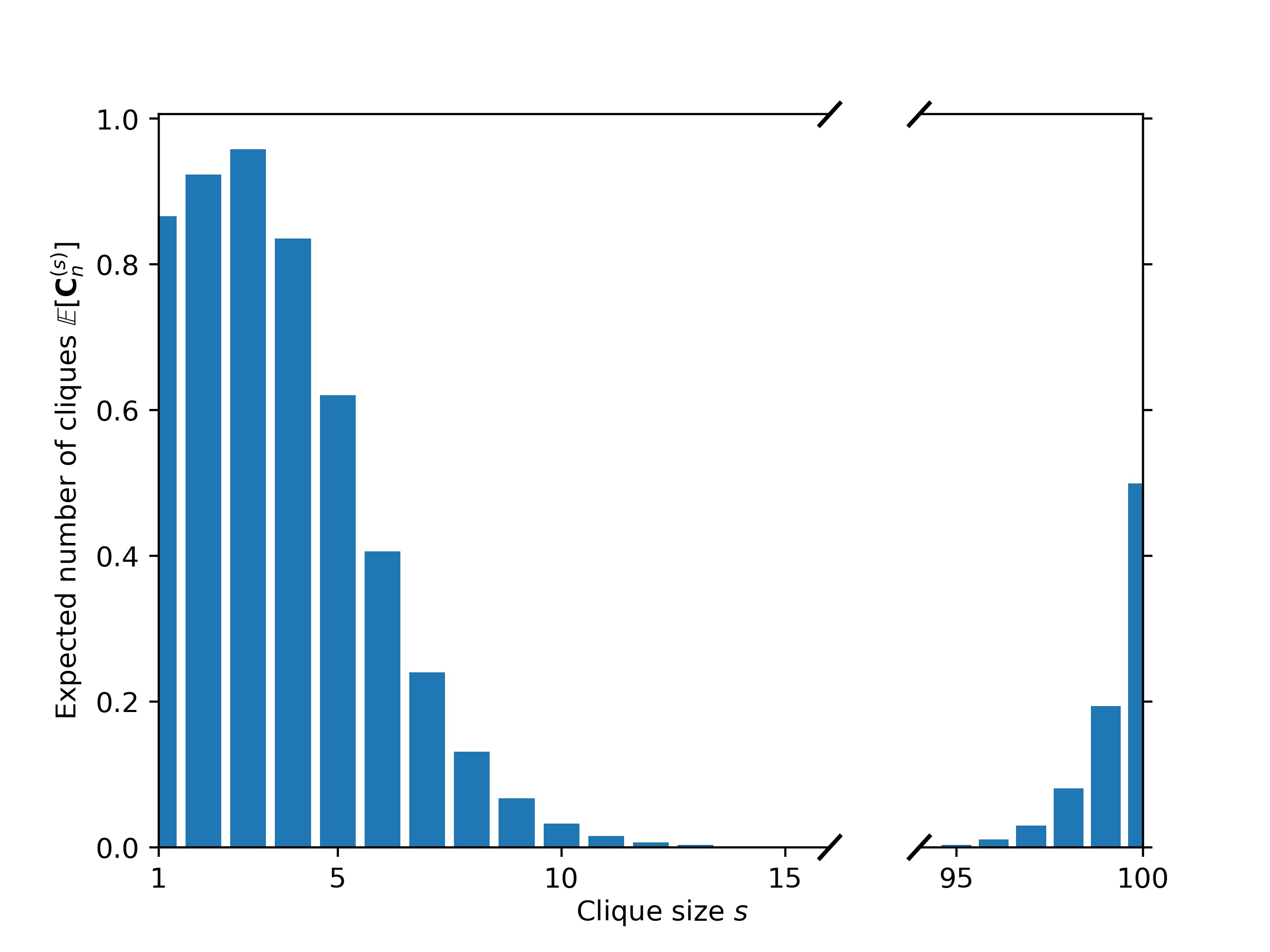}
  \caption{$n=100$}
  \label{fig:ncliques100}
\end{subfigure}%
\begin{subfigure}{.5\textwidth}
  \centering
  \includegraphics[width=\linewidth]{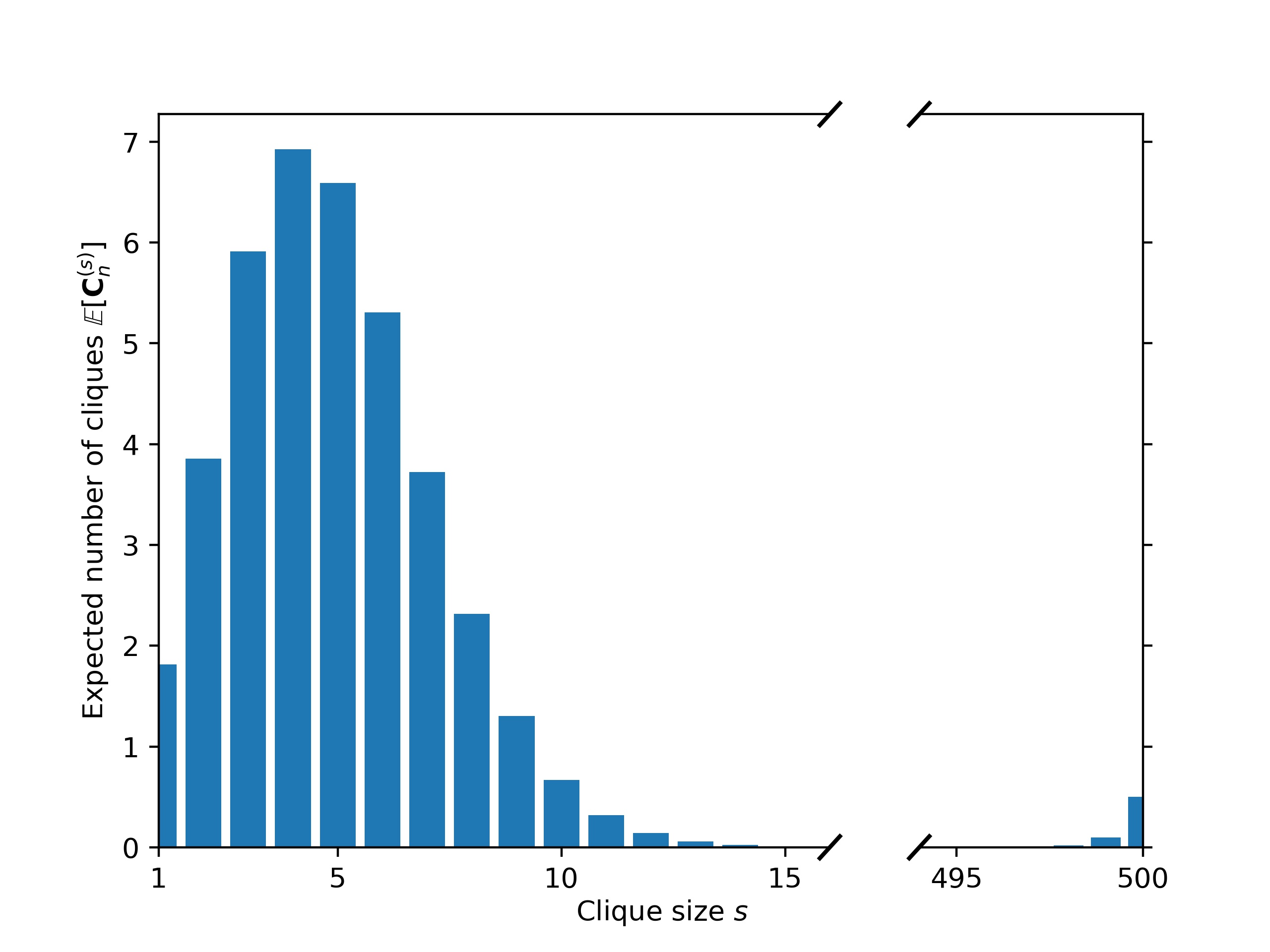}
  \caption{$n=500$}
  \label{fig:ncliques500}
\end{subfigure}
\caption{For $n\in\{100,500\}$, we show the expected number of cliques of each size at criticality ($p=p_n(1/2)$).}
\label{fig:ncliques}
\end{figure}

\section{The sparse regime ($p\downarrow 0$)} \label{sec:vanishing}
In this section, we investigate the structure of \(\rv{CG}_{n,p_n}\) when \(p_n\downarrow0\), where we focus on the case $p_n= n^{-\alpha+o(1)}$ for some $\alpha>0$. While classically the choice of \(p_n=\lambda/n\) refers to the sparse regime of the underlying \ER graph, we refer to each such choice of \(p_n\) to a sparse regime for simplicity. 
The main goal of this section is to prove \Cref{thm:mainDegreeSparse} and \Cref{thm:mainFixedDegSparse}. To this end, we rely on the following two lemmas.

\begin{lemma}\label{lem:vanishing-large}
     Let $p_n=(1+e^{-t_n})^{-1}$ with $t_n\sim -\alpha\log n$ for $\alpha>0$, then $\rv{S}_{n,p_n}<\nicefrac{8}{\alpha}+1$ with high probability. That is,
    \[
    \lim_{n\to\infty}\proba\big(\rv{S}_{n,p_n}<\tfrac{8}{\alpha}+1\big)=1.
    \]
\end{lemma}
\begin{proof}
    We have for any \(x\geq 2\)
    \begin{equation*}
    	\begin{aligned}
        	\proba(\rv{S}_{n,p_n}\geq x)&= \sum_{s= x}^n{n-1\choose s-1}e^{t_n{s\choose 2}}\frac{B_{n-s}(e^{t_n})}{B_n(e^{t_n})}\leq \sum_{s=x}^n n^{s-1}e^{t_n{s\choose 2}}\leq \sum_{s=x}^n n^{s-1+\frac{t_n}{\log n}{s\choose 2}}.
    	\end{aligned}
    \end{equation*}
    Now note that for all \(s\) satisfying     
    \[
    	s\geq \big(1+\tfrac{1}{s}\big)\tfrac{\log n}{-t_n} +1,
    \]
    we have \(s+\binom{2}{2} t_n/\log n\leq -1\). Further, since \(\log n/(-t_n)\leq {2}/{\alpha}\), this inequality is particularly satisfied by all \(s\geq \nicefrac{8}{\alpha}+1\). Hence,  
    \begin{equation*}
    	\P\big(\rv{S}_{n,p_n}\geq \tfrac{8}{\alpha}+1\big)\leq \sum_{s=\nicefrac{8}{\alpha}+1}^n n^{s-1+\frac{t_n}{\log n}{s\choose 2}} \leq n^{-1}, 
    \end{equation*} 
    concluding the proof.
\end{proof}

\begin{lemma}\label{lem:vanishing-kappa}
    Let $p_n=(1+e^{-t_n})^{-1}$ with $t_n\sim -\alpha\log n$ for $\alpha>0$. For any two $s,s'\in\mathbb{N}$, we have 
    \[
    \proba(\rv{S}_{n,p_n}=s)\leq n^{\left[ \alpha\frac{s'-1}{2} +\frac{1}{s'}\right]s-1-\alpha{s\choose 2}+o(1)}.
    \]
\end{lemma}
\begin{proof}
    We use the asymptotics of $t_n$ together with \cref{cor:bell-frac-bounds} to obtain
    \begin{equation*}
    	\begin{aligned}
        	\proba(\rv{S}_{n,p_n}=s')&={n-1\choose s'-1}e^{t_n{s'\choose 2}}\frac{B_{n-s'}(e^{t_n})}{B_n(e^{t_n})}=n^{s'-1+o(1)-(\alpha+o(1)){s'\choose 2}}\left(\frac{B_{n-1}(e^{t_n})}{B_n(e^{t_n})}\right)^{s'} \\ 
        	&=\kappa_n^{s'}n^{-1-\alpha{s'\choose 2}+o(1)}\leq 1,
    	\end{aligned}
    \end{equation*}
    where we have written $\kappa_n=n\frac{B_{n-1}(e^{t_n})}{B_n(e^{t_n})}$. This yields
    \[
        \kappa_n\leq n^{\frac{1}{s'}+\alpha\frac{s'-1}{2}+o(1)}, \quad \forall s'\in\N.
    \]
    Hence, for $d\in\mathbb{N}$,
    \[
    \proba(\rv{S}_{n,p_n}=s)= \kappa_n^{s}n^{-1-\alpha{s\choose 2}+o(1)}\leq n^{\left[ \alpha\frac{s'-1}{2} +\frac{1}{s'}\right]s-1-\alpha{s\choose 2}+o(1)},
    \]
    as desired.
\end{proof}

\begin{proof}[Proof of \Cref{thm:mainFixedDegSparse}.]
    We prove the equivalent statement for $\rv{S}_{n,p_n}=\rv{D}_{n,p_n}+1$.
    We set
    \[
    t_n=\log\frac{p_n}{1-p_n}=-\frac{2+o(1)}{(d+1)^2}\log n+\bigO(p_n)=-\frac{2+o(1)}{(d+1)^2}\log n
    \]
    and apply the Lemmas~\ref{lem:vanishing-large} and~\ref{lem:vanishing-large} for $\alpha=2(d+1)^{-2}$.
    For $s'=d+1$, Lemma~\ref{lem:vanishing-large} yields
    \[
    \proba(\rv{S}_{n,p_n}=s)\leq n^{\left[ \frac{d}{(d+1)^2} +\frac{1}{d+1}\right]s-1-\frac{s(s-1)}{(d+1)^2}+o(1)}= n^{-\left(\frac{s-d-1}{(d+1)^2}\right)^2+o(1)}.
    \]
    Substituting $\rv{D}_{n,p_n}=\rv{S}_{n,p_n}-1$ and $d'=s-1$ already proves~\eqref{eq:vanishing-asymptotics}.
    To finish the proof, we infer from Lemma~\ref{lem:vanishing-large} and the previous step that
    \begin{equation*}
    	\begin{aligned}
    		1+o(1) &= \proba(\rv{S}_{n,p_n}=s) = \proba(\rv{S}_{n,p_n}<4(d+1)^2+1) = \sum_{s=1}^{\lfloor4(d+1)^2+1\rfloor}\proba(\rv{S}_{n,p_n}=s) \\
    			&= \P(\rv{S_{n,p_n}}=d+1) +(\lfloor4(d+1)^2+1\rfloor-1)o(1),
    	\end{aligned}
    \end{equation*}
    yielding the desired result.
\end{proof}

We demonstrate \cref{thm:mainFixedDegSparse} in \cref{fig:vanishing-d2} for $d=2$ and $n\in\{200,2000\}$. As predicted by the proposition, the distribution of $\rv{S}_{n}$ appears to concentrate around $\rv{S}_{n}=3$.
However, the rate of convergence is quite slow, as can be seen from the fact that $\proba(\rv{S}_{n,p_n}=2)=n^{-1/9+o(1)}$ according to \cref{thm:mainFixedDegSparse}.%~\eqref{eq:vanishing-asymptotics}.
\begin{figure}
    \centering
    \includegraphics[width=0.8\linewidth]{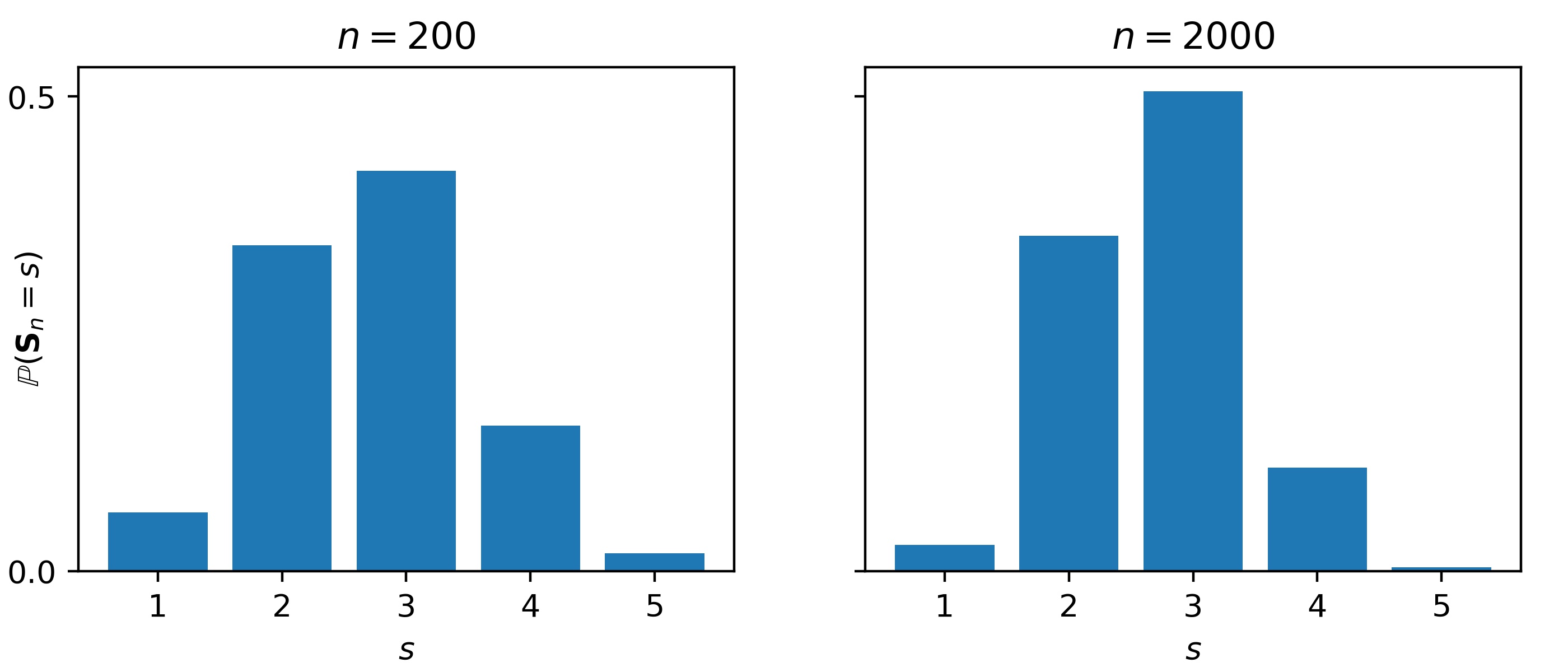}
    \caption{Demonstration of \cref{thm:mainFixedDegSparse} for $d=2$ and $n\in\{200,2000\}$. We take $p_n=n^{-2/9}$, resulting in $p_{200}\approx 0.31$ and $p_{2000}\approx 0.18$. We show the distribution of $\rv{S}_{n,p_n}$.}
    \label{fig:vanishing-d2}
\end{figure}

We now turn to the classical sparse regime \(p_n=\lambda/n\) and prove \Cref{thm:mainDegreeSparse}. Note that \Cref{thm:mainFixedDegSparse} cannot be applied directly as we require \(d\) to be an integer. However, we can still apply the previously derived lemmas. This leads to the interesting behaviour that the graph mainly consists of isolated vertices and single edges.  

\begin{proof}[Proof of \Cref{thm:mainDegreeSparse}.]
Set $t_n=\log\lambda-\log n+o(1)\sim-\log n$. Applying Lemma~\ref{lem:vanishing-kappa} for \(\alpha=1\) and with $s'=1$ yields
\[
\proba(\rv{S}_{n,p_n}=n)=n^{s-1-{s\choose 2}+o(1)}=n^{-{s-1\choose 2}+o(1)}.
\]
We immediately derive that this probability vanishes for all $s\geq 3$.
Furthermore, by Lemma~\ref{lem:vanishing-large}
\[
   1+o(1)=\P(\rv{S}_{n,p_n}\leq 8) =\proba(\rv{S}_{n,p_n}=1)+\proba(\rv{S}_{n,p_n}=2)+ o(1).
\]
We finally derive the exact limits of these two probabilities. We have
\begin{equation*}
    1+o(1)=\frac{B_{n-1}(e^{t_n})}{B_n(e^{t_n})}+(n-1)e^{t_n}\frac{B_{n-2}(e^{t_n})}{B_n(e^{t_n})}=z_n+(\lambda+o(1))z_n^2,
\end{equation*}
for $z_n=\frac{B_{n-1}(e^{t_n})}{B_n(e^{t_n})}$, and using $\frac{B_{n-2}(e^{t_n})}{B_n(e^{t_n})}=z_n^2(1+\bigO(n^{-1}))$ by Corollary~\ref{cor:bell-frac-bounds}.  Solving the quadratic equation $(\lambda+o(1)) z_n^2+z_n-(1+o(1))=0$ yields the unique positive solution
\[
z_n=\frac{-(1+o(1))\pm\sqrt{4\lambda+1+o(1)}}{2\lambda+o(1)}.
\]
Hence,
\[
\proba(\rv{S}_{n,p_n}=0)=z_n\rightarrow\frac{\sqrt{4\lambda+1}-1}{2\lambda}.
\]
To show the relation to the golden ratio \(\rho\), simply observe that \(\tfrac{\sqrt{5}-1}{2}\cdot \rho =1\).

\end{proof}

\paragraph{Acknowledgements.} This project has been initiated during the \emph{RandNET Summer School and Workshop on Random Graphs} in Eindhoven in August 2022. It was supported by the RandNET project, MSCA-RISE - Marie Sk\l{}odowska-Curie Research and Innovation Staff Exchange Programme (RISE), Grant agreement 101007705. 
MG was supported by the Netherlands Organisation for Scientific Research (NWO) through the Gravitation {\sc NETWORKS} grant no.\ 024.002.003.
LL was supported by the Leibniz Association within the Leibniz Junior Research Group on \emph{Probabilistic Methods for Dynamic Communication Networks} as part of the Leibniz Competition. EM was supported by the Deutsche Forschungsgemeinschaft
(DFG) (project number 443759178) through grant \emph{SPP2265 Random Geometric Systems}, Project
P01: \emph{Spatial Coagulation and Gelation}.
MG and EdP thank the Institut de Recherche en Informatique Fondamentale (IRIF),  Universit\'e Paris Cit\'e, for hosting them during part of this research.
MN was Supported by grants PID2020-113082GB-I00 and the Severo Ochoa and
Mar\'{i}a de Maeztu Program for Centers and Units of Excellence in R\&D (CEX2020-001084-M).
We thank Sergey Dovgal and Khaydar Nurligareev for providing early access to their paper~\cite{dovgal2023asymptotics}.
We would further like to thank Remco van der Hofstad for his useful feedback.

\paragraph{Competing interests:} The author declare none.

{\footnotesize
\printbibliography
}

\appendix
\section{Numerics}\label{sec:numerics}
In this section, we explain how we numerically evaluate $B_n(\nicefrac{p}{1-p})$, how we use this to sample $\rv{CG}_{n,p}$, and we show how to approximate the critical sequence.

\paragraph{Computing $B_n(w)$.}
To compute $B_n(w)$ for $w>0$, we make use of \cref{cor:Bn:recursion}, together with $B_0(w)=1$. This allows one to compute $B_n(w)$ for any $n$ and $w$, and to generate Figures~\ref{fig:phase-transition}, \ref{fig:ncliques}, \ref{fig:vanishing-d2}, and~\ref{fig:golden-ratio}.

\paragraph{Sampling $\rv{CG}_{n,p}$.}
To sample a cluster graph, as we have done in~\cref{fig:cg_samples}, we use \cref{cor:exact:degree-pmf} with $\rv{S}_{n,p}=1+\rv{D}_{n,p}$ to sample the size of the clique that the vertex with index $n$ belongs to. Conditioned on this value of $\rv{S}_{n,p}$, the remainder of the cluster graph has the same distribution as $\rv{CG}_{n-\rv{S}_{n,p},p}$.
This allows us to again sample $\rv{S}_{n-\rv{S}_{n,p},p}=1+\rv{D}_{n-\rv{S}_{n,p},p}$ using \cref{cor:exact:degree-pmf} to recursively sample a cluster graph.

\paragraph{Approximating the critical sequence.}
For Figures~\ref{fig:phase-transition} and~\ref{fig:ncliques}, we need to compute the value $p$ for which $\proba(\rv{C}_{n,p}=1)=\nicefrac{1}{2}$. We do this using Newton-Raphson iteration. Let $p(t)=(1-e^{-t})^{-1}$, then
\[
\proba(\rv{C}_{n,p(t)}=1)=\frac{e^{{n\choose 2}t}}{B_n(e^t)}.
\]
Thus, we want to find the fixed point of 
\[
\log\proba(\rv{C}_{n,p(t)}=1)-\log\nicefrac{1}{2}={n\choose 2}t-\log B_n(e^t)+\log 2=:f(t),
\]
where $t=\log(\nicefrac{p}{1-p})$. We take the derivative \wrt $t$ and obtain
\[
f'(t)={n\choose 2}-\frac{d}{dt}\log \sum_{G \in \mathcal{CG}_n} e^{t\cdot m(G)}={n\choose 2}-\E[\rv{M}_{n,p(t)}].
\]
To compute $m_n(t):=\E[\rv{M}_{n,p(t)}]$, we substitute $w=e^t$ into the recursion formula of~\cref{cor:Bn:recursion} and take the derivative \wrt $t$:
\[
m_n(t)B_n(e^t)=\sum_{s=1}^n {n-1\choose s-1}e^{t{s\choose 2}}\cdot \left({s\choose 2}+ m_{n-s}(t)\right)B_{n-s}(e^t)
\]
Dividing by $B_n(e^t)$ yields the following recursion for the expected number of edges:
\[
m_n(t)=\sum_{s=1}^n {n-1\choose s-1}e^{t{s\choose 2}}\frac{B_{n-s}(e^t)}{B_n(e^t)}\cdot \left({s\choose 2}+ m_{n-s}(t)\right),
\]
and $m_0(t)=0$.
After each step, we use the following update rule to improve our estimate:
\[
t_{k+1}=t_k-\frac{f(t_k)}{{n\choose 2}-m_n(t_k)}.
\]
We initialize with
\[
t_0=\frac{\log B_n(1)-\log2}{{n\choose 2}},
\]
which is the solution of
\[
\frac{e^{t_0{n\choose 2}}}{B_n(1)}=\frac{1}{2}.
\]
Although we have not proven the convergence of the above Newton-Raphson procedure, it converged in all cases needed to generate Figures~\ref{fig:phase-transition} and~\ref{fig:ncliques}.

\section{Nomenclature}\label{sec:nomenclature}

For all random variables, we write $\rv{X}_{n,p}$ to describe its dependency on $\rv{CG}_{n,p}$ and we omit the $p$ if it is clear from the context.
\begin{itemize}
    \item $\G_n$ set of simple graphs on $n$ vertices
    \item $\mathcal{CG}$ set of all cluster graphs
    \item $\mathcal{CG}_n$ set of all cluster graphs of $n$ vertices
    \item $C(w,z)$ generating function of cliques
    \item $C_r(w,z)=(z\partial_z)^rC(w,z)$ 
    \item Coefficient extraction $[z^n]f(z)$
    \item $p$ connection probability of the \ER random graph
    \item $w:=p/(1-p)$
    \item $t:=\log w$
    \item $p(t):=(1+e^{-t})^{-1}$
    \item $B_n(w)=\sum_{G\in\mathcal{CG}_n}w^{m(G)}$ partition function, with $B_n(1)=B_n$
    \item $W(n)$ Lambert W-function satisfying $W(n)e^{W(n)}=n$
    \item $\rv{CG}_{n,p}$ random cluster graph. $p$ can be omitted if it is clear from the context 
    \item $\rv{S}_{n,p}$ random variable denoting the clique size of a randomly chosen vertex in $\rv{CG}_{n,p}$.
    \item $\rv{D}_{n,p}=\rv{S}_{n,p}-1$ degree of a random vertex
    \item $\rv{M}_{n,p}$  random variable denoting the number of edges in $\rv{CG}_{n,p}$
    \item $\rv{C}_{n,p}$  number of cliques/components/blocks
    \item $\rv{C}^{(s)}_{n,p}$  number of cliques/components/blocks of size $s$
    \item $\PGF_X$ probability generating function  of a random variable $X$.
    \item $\text{MGF}_{X}$
    moment generating function  of a random variable $X$.
    \item $\text{CF}_{X}$
    characteristic function  of a random variable $X$.
\end{itemize}

\end{document}